\documentclass{article}
\usepackage[a4paper,top=3cm,bottom=3cm,left=2.1cm,right=2.1cm,marginparwidth=1.75cm]{geometry}

\usepackage{graphicx} 
\usepackage{bm}
\usepackage{amsmath, amsfonts, amssymb, amsthm}
\usepackage{mathrsfs, dsfont}
\usepackage[dvipsnames]{xcolor}
\usepackage{tikz}
\usepackage{verbatim}
\usepackage{subcaption}
\usepackage{booktabs}
\usepackage{float}
\usepackage{comment}
\usepackage{ulem}
\usepackage{bigints}
\usepackage{pifont}
\usepackage{enumitem}
\usepackage{mathtools}
\usepackage{bbm}
\usepackage{authblk}
\usepackage{shuffle}
\usepackage[numbers]{natbib}
\usepackage[colorlinks=true, allcolors=blue]{hyperref}
\usepackage{xparse} 
\usepackage{hyperref}
\usepackage{tasks}
\usepackage{algorithm}
\usepackage{algpseudocode}
\usepackage{enumitem}

\settasks{style=itemize}

\newtheorem{definition}{Definition}
\newtheorem{theorem}{Theorem}
\newtheorem{proposition}{Proposition}
\newtheorem{remark}{Remark}
\newtheorem{lemma}{Lemma}
\newtheorem{corollary}{Corollary}
\newtheorem{example}{Example}

\newcommand{\emptyword}{{\textup{\textbf{\o{}}}}}
\newcommand{\word}[1]{{{\boldsymbol{#1}}}}

\newcommand{\length}[1]{\vert \vert \vert #1 \vert \vert \vert}

\title{Linear independence properties of the signature components of time-augmented stochastic processes}
\author[1]{Arthur Bourdon \footnote{Corresponding author : arthur.bourdon@milliman.com; arthur.bourdon@enpc.fr}}
\author[2]{Benjamin Jourdain}
\author[3]{Hervé Andrès}

\affil[1,3]{Milliman R\&D, Paris, France}
\affil[1,2]{CERMICS, ENPC, Institut Polytechnique de Paris, CNRS, Marne-la-Vallée, France \& MATHRISK team-project, Inria Paris, France}
\date{\today}
\begin{document}

\maketitle

\begin{center}
\textbf{\large Abstract}
\end{center}

\setlength{\leftskip}{1cm} 
\setlength{\rightskip}{1cm} 
Adding the time as a component of a stochastic process before computing its signature terminal value ensures injectivity and supports universal approximation results, but it induces linear dependence among the components of the signature terminal value. For any natural number $N$, the terminal values of the signature components associated with words of length not greater than $N$ are the image of the terminal values of the signature components associated with words of length $N$ by some universal linear map. We generalize this result by exhibiting other subfamilies of components -- represented by subfamilies of words -- with the same representation property. When considering the signature of the solution to a stochastic differential equation with a uniformly elliptic diffusion coefficient, we show that any such subfamily of components is linearly independent for the almost-sure equality and therefore provides a basis of the linear span of all components associated with words of length not greater than $N$. The linear independence of these subfamilies is preserved for the affine interpolation of this solution on a grid with a sufficiently small time step. We characterize bases of components with minimal computation cost. Finally, we remark that the subfamilies of words obtained above share a similar representation property when applied to the time-augmented EFM signature recently introduced in \citep{abi2025exponentially}. For a Brownian semimartingale with a non-degenerate diffusion coefficient, we show that any such subfamily of components of its time-augmented EFM signature is almost-surely linearly independent for the $dt$-a.e. equality.
\setlength{\leftskip}{0cm} 
\setlength{\rightskip}{0cm} 

\tableofcontents

\section{Introduction}

The signature plays a fundamental role in the analysis of paths and functionals defined on a path space. The signature of a path has important properties.

\begin{enumerate}[label=(\roman*)]
    \item \textbf{The signature encodes the path uniquely} up to translation, reparameterization and tree-like equivalence \citep{hambly2010uniqueness}. The signature uniquely determines the path if we fix its starting point and if we add the time variable as a component \citep{fermanian2022functional}. These two transformations are known as the \textit{base-point augmentation} and \textit{time-augmentation} of a path.
 
    \item \textbf{Linear functionals on signature are universal approximators}. Indeed the linear functionals of the signature form a point-separating algebra -- comparable to polynomials on finite-dimensional
spaces -- and thus any continuous functional defined on the space of paths starting at $0$ at $t=0$ can be well approximated on compact sets with a linear functional of the terminal value of the signature of the time-augmented path. We denote by $\widehat{X}=(\widehat{X}_{t}^{\word{0}},\widehat{X}_{t}^{\word{1}},\cdots,\widehat{X}_{t}^{\word{d}})_{t \in [0,T]} :=\left(t,X_{t}^{\word{1}},\cdots,X_{t}^{\word{d}}\right)_{t \in [0,T]} :[0,T] \rightarrow \mathbb{R}^{d+1}$ the time-augmented path of $(X_t)_{t \in [0,T]}=\left(X_{t}^{\word{1}},\cdots,X_{t}^{\word{d}}\right)_{t \in [0,T]}$.
\end{enumerate}

These two important properties of the signature have enabled the development of several results on the approximation of functionals defined on a path space. While the original formulation involves the \textit{uniform} approximation on \textit{compact} sets, other formulations exist when the compactness assumption is relaxed \citep{cuchiero2023global,bayer2023primal} and when the approximation is in an $L^p$ sense \citep{bayer2023primal,bayer2025pricingamericanoptionsrough} using weighted spaces and the notion of robust signature. The signature of a path is well defined in a pathwise and non-probabilistic manner if the path is of finite variation. For a given path of bounded variation $X:=(X_{t}^{\word{1}},\cdots,X_{t}^{\word{d}})$, the signature of the time-augmented path $(\widehat{X}_{t \le T}):=(t,X_t)_{t \le T}$ is defined as the collection of all running iterated integrals of $\widehat{X}$ against itself. For each word $\word{i_1 \cdots i_n}$ with each letter taking its values in the alphabet $\mathcal{A}=\left \{\word{0},\cdots,\word{d}\right \}$ and for all $t \le T$ we define

\begin{align*}
\begin{cases}
\widehat{\mathbb{X}}_{t}^{\word{i_1 \cdots i_n}} = \int_{0 < t_1 < \cdots <t_n<t} d\widehat{X}_{t_1}^{\word{i_1}} \cdots d\widehat{X}_{t_n}^{\word{i_n}}, \\
\widehat{\mathbb{X}}_{t}^{\emptyword}=1,
\end{cases}
\end{align*}
where each letter encodes a canonical basis vector of $\mathbb{R}^{d+1}$: $\widehat{X}_{s}^{\word{0}}=s$ and $\widehat{X}_{s}^{\word{i}} = X_{s}^{\word{i}}$ for $\word{i} \in \{\word{1},\cdots,\word{d}\}$ and $s \in [0,T]$. The signature of $\widehat{X}$ evaluated at time $t$ is the collection of all these integrals: $\widehat{\mathbb{X}}_t = \left(\widehat{\mathbb{X}}_{t}^{\word{i_1 \cdots i_n}}\right)_{(\word{i_1}, \cdots, \word{i_n}) \in \mathcal{A}^n, n \in \mathbb{N}}$ where $\mathbb{N}$ denotes the set of non-negative integers. For a given truncation order $N \in \mathbb{N}$, we can define the $N$-step truncated time-augmented signature of $X$ as $\widehat{\mathbb{X}}^{\le N}:=\left(\widehat{\mathbb{X}}^{\word{i_1 \cdots i_n}}\right)_{(\word{i_1}, \cdots, \word{i_n}) \in \mathcal{A}^n, n \le N} : [0,T] \rightarrow \mathbb{R}^{\frac{(d+1)^{N+1}-1}{d}}$. There also exists a definition of the signature in a probabilistic setting when considering stochastic processes that are càdlàg semimartingales which do not have finite variation \citep{friz2010multidimensional,friz2017general}. In the case of continuous semimartingales, the signature is then defined as the collection of all iterated integrals using the Stratonovich integration rule.

In a learning task context, we want to approximate an unknown functional $F : \Lambda_T \rightarrow \mathbb{R}$ defined on a path space $\Lambda_T \subset (\mathbb{R}^d)^{[0,T]}$. In practice, we cannot cover uniformly $\Lambda_T$: we need to define a probability measure $\mathbb{P}_X$ on $\Lambda_T$. Thus we want to approximate $F$ in $L^2(\mathbb{P}_X)$. For a fixed truncation order $N \in \mathbb{N}$, we want to minimize the following loss function

\begin{align}
\label{loss}
\nonumber \mathcal{L} : ~~&\mathbb{R}^{\frac{(d+1)^{N+1}-1}{d}} \longrightarrow \mathbb{R}^{+} \\
   &\theta \longmapsto \mathbb{E}\left[\left|F(X) - \theta^{\top} \widehat{\mathbb{X}}_{T}^{\le N}\right|^2\right],
\end{align}
where $X \sim \mathbb{P}_X$, $u^{\top}$ denotes the transposed row vector associated with the column vector $u$ and $u^{\top} v$ denotes the scalar product between $u$ and $v$ for any given column vectors $u,v \in \mathbb{R}^m$. We define a $N$-step signature proxy as a function $X \mapsto \theta^{\top} \widehat{\mathbb{X}}_{T}^{\le N}$ for a given $\theta \in \mathbb{R}^{\frac{(d+1)^{N+1}-1}{d}}$. A $N$-step signature proxy is said to be optimal if it is the closest $N$-step signature proxy to $F$ with respect to the $L^2(\mathbb{P}_X)$ induced metric. The loss function is strictly convex -- and hence admits a unique argmin -- if and only if the second order moment of the features $\mathbb{E}\left[\left(\widehat{\mathbb{X}}_{T}^{\le N}\right) \left(\widehat{\mathbb{X}}_{T}^{\le N}\right)^{\top} \right]$ is positive definite, which is equivalent to the linear independence of the features in $L^2$. We will see that this property does not hold since for any given truncation order $N \in \mathbb{N}$, the signature components associated with words of length $N$ have the same linear span as all signature components of the $N$-step truncated signature: a strict subset of features is sufficient to span the other ones. Finding a basis of this span is interesting for several reasons:

\begin{enumerate}
    \item it is necessary and sufficient for the uniqueness of the representation of the optimal $N$-step signature proxy in the considered basis. Hence the optimal model is identifiable and this improves interpretability, 
    \item the storage cost of features is reduced,
    \item the computational cost is reduced,
    \item the prediction performance is preserved in theory, while empirical improvements may arise in practice due to improved conditioning, reduced variance, or more stable cross-validation.
\end{enumerate}
We can cite \cite{kohavi1997wrappers,blum1997selection,guyon2003introduction} which study methods to find subsets of features that have a good generalization performance. However these methods can be heavily time consuming, in particular when the number of features is high (which is the case with the signature). Another way of reducing the dimension of the problem is to perform a Partial Components Regression which relies on a spectral decomposition of the empirical variance matrix of the features but the construction of orthogonal features which is learned on a training set can suffer from overfitting and lack of generalization performance, in particular in a high dimension setup. The Lasso regression, which promotes sparsity in the estimated coefficients, is widely used as a means of implicitly performing feature selection. However, this approach can suffer from instability and lack of consistency in the selected subset of variables \cite{zhao2006model}. In particular, the theoretical guarantees of the Lasso -- such as sign consistency -- rely on conditions like the \textit{irrepresentable condition}, which assumes that the active features are linearly independent and not excessively correlated with the inactive ones. When the features are linearly dependent or highly collinear, the notion of a unique active subset becomes ill-defined, and the irrepresentable condition necessarily fails. In such cases, the Lasso may still provide accurate predictions but cannot be expected to identify a unique or stable subset of relevant variables. Moreover, the method may involve a substantial computational cost when applied to large-scale or ill-conditioned problems. Knowing a-priori which subfamilies of features are linearly independent in $L^2(\mathbb{P}_X)$ is the ideal setting since the computational cost of ignoring a given set of features is null. We can cite \cite{guo2025consistency} which studies the consistency of Lasso regression using signature, both theoretically and numerically: however it relies on the assumption that the joint distribution of the signature components is non-degenerate, finding a basis of the features is a necessary condition to satisfy this assumption. Let us now explain why the shuffle product property of the signature implies that whatever $\mathbb{P}_X$, the components of $\widehat{\mathbb{X}}_{T}^{\le N}$ can never be linearly independent. We define $\mathcal{W}_{0} = \{\emptyword\}$ and $\mathcal{W}_N = \{ \word{i_1} \word{i_2} \cdots \word{i_N} ~|~ (\word{i_1},\cdots,\word{i_N}) \in \mathcal{A}^{N} \}$  for $N \in \mathbb{N}^*=\mathbb{N} \setminus \{0\}$. We also define $\mathcal{W}_{\le N} = \bigsqcup_{k=0}^{N} \mathcal{W}_{k}$ for $N \in \mathbb{N}$ and $\mathcal{W}_{< \infty} = \bigsqcup_{k=0}^{\infty} \mathcal{W}_{k}$, where $\bigsqcup$ denotes the disjoint union. We denote the length of a word $\word{w} \in \mathcal{W}_{< \infty}$ by $\|\word{w}\|$.

\begin{definition}[Concatenation]
\label{concatenation}

We define the concatenation between two words in $\mathcal{W}_{< \infty}$. Let $N,M \in \mathbb{N}^*$, $\word{w} = \word{i_1} \cdots \word{i_N} \in \mathcal{W}_{N}$ and $\word{v}= \word{j_1} \cdots \word{j_M} \in \mathcal{W}_{M}$. Then $\word{w}\word{v} = \word{i_1} \cdots \word{i_N} \word{j_1} \cdots \word{j_M} \in \mathcal{W}_{N+M} $. The word $\emptyword$ is the only word such that $\word{w} \emptyword = \emptyword \word{w} = \word{w}$ for every $\word{w} \in \mathcal{W}_{<\infty}$.

\end{definition}

\begin{definition}[Span of words set]
We define the set of all finite linear combinations of words $\word{w} \in B \subset \mathcal{W}_{< \infty}$ by 

\begin{align*}
    \text{Span}(B) = \left\{ \sum_{\word{w} \in B} c_{\word{w}} \word{w} ~|~ \{ \word{w} \in B , c_{\word{w}} \neq 0 \} \text{ is finite} \right\}.
\end{align*}
If $B$ is a finite subset of $\text{Span}(\mathcal{W}_{<\infty})$, we define $\text{Span}(B)$ analogously. 
\end{definition}

We can extend the concatenation between two elements of $\text{Span}(\mathcal{W}_{<\infty})$ by bilinearity. We are now able to define the shuffle product between two words.

\begin{definition}[Shuffle product]
\label{def shuffle product}

The shuffle product between two words $\shuffle : \mathcal{W}_{<\infty} \times \mathcal{W}_{<\infty}  \longrightarrow \text{Span}(\mathcal{W}_{<\infty})$ is defined inductively for all $\word{w},\word{v} \in \mathcal{W}_{<\infty}$ and $\word{i},\word{j} \in \mathcal{A}$ as follows:

\begin{align*}
\begin{cases}
\word{w}\shuffle\emptyword=\emptyword \shuffle \word{w} = \word{w} \\
(\word{w}\word{i}) \shuffle (\word{v}\word{j}) = (\word{w} \shuffle (\word{v}\word{j}))\word{i} + ((\word{w}\word{i})\shuffle \word{v})\word{j}.
\end{cases}
\end{align*}
\end{definition}

The shuffle product between two words corresponds to the shuffling of these two words, while keeping the order of letters in each word as illustrated by the following example:

\begin{example}

\begin{align*}
    &\word{1} \shuffle \word{21} = \word{121} + 2 \cdot \word{211}
\end{align*}
\end{example}

The shuffle product can be extended to $\text{Span}(\mathcal{W}_{< \infty}) \times\text{Span}(\mathcal{W}_{< \infty})$ by bilinearity.

\begin{remark}
\label{remark shuffle}
We can easily observe that the words which appear in the shuffle of $\word{w}$ and $\word{v}$ all have length $\|\word{w}\| + \|\word{v}\|$.
\end{remark}

\begin{definition}[Dual bracket]
\label{dual bracket}

Let $\ell =\sum_{\word{w} \in \mathcal{W}_{< \infty}} c_{\word{w}} \word{w} \in \text{Span}(\mathcal{W}_{< \infty})$, and $\widehat{\mathbb{X}}_{T}$ be the terminal value of the time-augmented signature of some path $X : [0,T] \rightarrow \mathbb{R}^d$. We define the dual bracket:

\begin{align*}
    \left<\ell, \widehat{\mathbb{X}}_t \right>&:= \sum_{\word{w} \in \mathcal{W}_{< \infty}} c_{\word{w}} \widehat{\mathbb{X}}_{t}^{\word{w}} \in \mathbb{R}.
\end{align*}

\end{definition}

The signature of any given path respects the so-called shuffle product property which is a strong algebraic relation. This property tells us that any product of terms of the signature can be rewritten as a linear combination of higher order terms in an explicit fashion. To illustrate this point let us consider a path $X:=(X^{\word{1}},X^{\word{2}}) : [0,T] \rightarrow \mathbb{R}^2$. According to the integration by parts formula we get $\mathbb{X}^{\word{1}}_{T}\mathbb{X}^{\word{2}}_{T} = \int_{0}^{T} dX_{t}^{\word{1}}\int_{0}^{T} dX_{t}^{\word{2}} = \int_{0}^{T} \int_{0}^{t} dX_{s}^{\word{1}} dX_{t}^{\word{2}} + \int_{0}^{T} \int_{0}^{t} dX_{s}^{\word{2}} dX_{t}^{\word{1}} = \mathbb{X}_{T}^{\word{12}} + \mathbb{X}_{T}^{\word{21}}$: the integration by parts formula is just some particular case of the shuffle product property.

\begin{proposition}[Shuffle product property]

Let $X :[0,T] \rightarrow \mathbb{R}^d$ be a path such that its time-augmented signature $\widehat{\mathbb{X}}$ is well defined. Then 

\begin{align*}
\forall t \in [0,T], \forall \word{w},\word{v} \in \mathcal{W}_{<\infty} , \left<\word{w},\widehat{\mathbb{X}}_t\right>\left<\word{v},\widehat{\mathbb{X}}_t\right>=\left<\word{w} \shuffle \word{v},\widehat{\mathbb{X}}_t\right>.
\end{align*}
    
\end{proposition}

We denote the word $\overbrace{\word{0} \word{0} \cdots \word{0}}^{k}$ by $\word{0}_k$ for $k \in \mathbb{N}$. Let $\widehat{\mathbb{X}}_{T}$ be the terminal value of the time-augmented signature of some path $X$. Since $\left< \word{0}_{k}, \widehat{\mathbb{X}}_T\right> = \widehat{\mathbb{X}}_{T}^{\word{0}_{k}} = \int_{0}^{T} \int_{0}^{t_1} \cdots \int_{0}^{t_{k-1}} dt_{k} \cdots dt_1 = \frac{T^k}{k!}>0$ and $\widehat{\mathbb{X}}_{T}^{\word{w}} = \left<\word{w},\widehat{\mathbb{X}}_{T}\right>$, we obtain $\widehat{\mathbb{X}}_{T}^{\word{w}} = \frac{k!}{T^k} \widehat{\mathbb{X}}_{T}^{\word{w}} \left<\word{0}_{k},\widehat{\mathbb{X}}_{T}\right>  = \frac{k!}{T^k} \left<\word{w} ,\widehat{\mathbb{X}}_{T}\right> \left<\word{0}_{k},\widehat{\mathbb{X}}_{T}\right>$. Thanks to the shuffle product property we get

\begin{align}
\label{egalité bracket}
\left<\word{w}, \widehat{\mathbb{X}}_{T}\right> =  \left< \frac{k!}{T^k}(\word{w} \shuffle \word{0}_{k}),\widehat{\mathbb{X}}_{T}\right>, \forall k \in \mathbb{N}.
\end{align}

Let $(\Omega,\mathcal{F},\mathbb{P})$ be a probability space on which we define a path valued random variable $X : \Omega \rightarrow (\mathbb{R}^d)^{[0,T]}$ such that its time-augmented signature is well defined. In the bounded variation framework or in the càdlàg semimartingale framework, the terminal value of the (non truncated) time-augmented signature is measurable. For any given subset of words $B \subset \mathcal{W}_{< \infty}$, we define $\text{Span}(\widehat{\mathbb{X}}_{T}^{\word{w}} ~~|~~ \word{w} \in B) := \left\{ \left<\ell,\widehat{\mathbb{X}}_{T}\right> ~|~ \ell \in \text{Span}(B)\right\}$ the quotient space for the almost-sure equality equivalence relation. An immediate consequence of Equation \eqref{egalité bracket} is:

\begin{align*}
    \text{Span}(\widehat{\mathbb{X}}_{T}^{\word{w}} ~|~ \word{w} \in \mathcal{W}_{\le N}) = \text{Span}(\widehat{\mathbb{X}}_{T}^{\word{w}} ~|~ \word{w} \in \mathcal{W}_{N}),
\end{align*}
and it reveals an upper bound for the dimension of $\text{Span}(\widehat{\mathbb{X}}_{T}^{\word{w}} ~|~\word{w} \in \mathcal{W}_{\le N})$: 

\begin{align*}
\dim(\text{Span}(\widehat{\mathbb{X}}_{T}^{\word{w}} ~|~\word{w} \in \mathcal{W}_{\le N})) \le \text{Card}(\mathcal{W}_{N}) = (d+1)^N.
\end{align*}
We can now define an appropriate notion of basis of words. 

\begin{definition}[Basis of words for $\mathcal{W}_N$]
\label{basis of words}
Let $B \subset \mathcal{W}_{\le N}$ be a subset of words. We say that $B$ is a basis of words for $\mathcal{W}_{N}$ if $(\word{w} \shuffle \word{0}_{N - \|\word{w}\|})_{\word{w} \in B}$ is a basis of $\text{Span}(\mathcal{W}_{N})$.
\end{definition}

The motivations are twofold:

\begin{enumerate}
    \item If $B$ is a basis of words for $\mathcal{W}_{N}$, then the family of random variables $(\widehat{\mathbb{X}}_{T}^{\word{w}})_{\word{w} \in B}$ has the same linear span as $(\widehat{\mathbb{X}}_{T}^{\word{w}})_{\word{w} \in \mathcal{W}_{N}}$. Hence adding another feature from the $N$-step truncated signature will bring the family to be linearly dependent.
    \item If $B$ is a basis of words for $\mathcal{W}_{N}$ and if the family of random variables $(\widehat{\mathbb{X}}_{T}^{\word{w}})_{\word{w} \in \mathcal{W}_{N}}$ is linearly independent for the $\mathbb{P}$-almost sure equality then it naturally follows that $(\widehat{\mathbb{X}}_{T}^{\word{w}})_{\word{w} \in B}$ is a basis of $\text{Span}(\widehat{\mathbb{X}}_{T}^{\word{w}} ~|~\word{w} \in \mathcal{W}_{\le N})$.
\end{enumerate}

We will indeed exhibit cases when the family of random variables $(\widehat{\mathbb{X}}_{T}^{\word{w}})_{\word{w} \in \mathcal{W}_{N}}$ is linearly independent and it follows that $\dim(\text{Span}(\widehat{\mathbb{X}}_{T}^{\word{w}} ~|~\word{w} \in \mathcal{W}_{\le N})) = (d+1)^N$.
\newline

\subsection*{Contributions.} In this work, we introduce the notion of basis of words for $\mathcal{W}_N$ which are subsets of $\mathcal{W}_{\le N}$ such that the components of the terminal value of the time-augmented signature associated with these subsets have the same linear span as the whole family of the terminal value of the $N$-step truncated time-augmented signature components, regardless of the law of the underlying stochastic process. Shuffling a word $\word{w}$ with $\word{0}_{N - \|\word{w}\|}$ does not produce any letters other than $\word{0}$ and those originally present in $\word{w}$ with their order preserved. we obtain that a basis of words for $\mathcal{W}_N$ can be decomposed into bases of words for $\mathcal{W}_{N}(\word{\gamma})$ for $\word{\gamma}$ in the set of words with no letter $\word{0}$, and $\mathcal{W}_{N}(\word{\gamma})$ denotes the set of words with length $N$ containing $\word{0}$ and the ordered letters of $\word{\gamma}$ (see Theorem \ref{main result -1}). We give necessary and sufficient conditions for a subset of words of length not greater than $N$ to be a basis of words for $\mathcal{W}_N(\word{\gamma})$ in Theorem \ref{main result 0} and Theorem \ref{main result 1}. In particular, we recover the result from \citep{dupire2022functional}[Theorem 3.9] which states that the subset of words of length not greater than $N$ which do not end by the letter $\word{0}$ is a basis of words for $\mathcal{W}_N$. This particular choice leads to the minimal computation time using Chen's relation in a backward way (see Theorem \ref{main result 4}). The same optimal computation time is symmetrically achieved in a forward way for the subset of words of length not greater than $N$ which do not start by the letter $\word{0}$. We show that if the stochastic process is the solution to a SDE with an uniformly elliptic diffusion coefficient, then all the subfamilies of signature components associated with bases of words for $\mathcal{W}_{N}$ also are linearly independent and thus form bases of the linear span of the $N$-step truncated time-augmented signature. We show that for any given truncation order $N \in \mathbb{N}$, the linear independence of these subfamilies is preserved when we consider an affine interpolation of the solution to this SDE for a sufficiently small time step. We remark that the notion of basis of words for $\mathcal{W}_N$ is also relevant for the (time-augmented) EFM signature introduced in \citep{abi2025exponentially}. Indeed, for all $N \in \mathbb{N}$, the family of components of the time-augmented EFM signature -- which are stochastic processes -- associated with a basis of words for $\mathcal{W}_N$ has the same linear span as the family of components associated with words of length not greater than $N$. Finally, when considering a Brownian semimartingale, we show that under non-degeneracy of its diffusion term, the components associated with a basis of words for $\mathcal{W}_N$ of its time-augmented EFM signature are almost-surely linearly independent for the $dt$-a.e equality.

\subsection*{Outline.} The main theoretical results are presented in Section \ref{part1}. Numerical results are presented in Section \ref{part2}. Technical lemmas and the proofs of the main results are postponed to Section \ref{part3}.

\section{Main results}
\label{part1}

\subsection{Bases of words and linear independence of signature components}

Shuffling a word $\word{w}$ such that $\|\word{w}\| \le N$ with $\word{0}_{N - \|\word{w}\|}$ does not produce any letters other than $\word{0}$ and those that were originally present in the word ($\word{0}$ included). Moreover, it does not change the relative order of appearance of letters different from $\word{0}$. This is why the word which consists in these ordered letters plays an important role. 

\begin{definition}[Pure word]
Let $P : \mathcal{W}_{< \infty} \rightarrow \mathcal{W}_{< \infty}$ as the operator that removes the $\word{0}$'s of a word, defined recursively as:

\begin{align*}
P(\emptyword) = \emptyword ~~\text{and}~~ P(\word{wi}) = 
\begin{cases}
      P(\word{w})\word{i} ~~\text{if}~~ \word{i} \neq \word{0} \\
      P(\word{w}) ~~\text{if}~~ \word{i} = \word{0}
\end{cases}.
\end{align*}
We define $\mathcal{PW}_{0} =\{\emptyword\}$, $\mathcal{PW}_{\le N} = P(\mathcal{W}_{\le N})$ for $N \in \mathbb{N}$ and $\mathcal{PW}_{< \infty} =P(\mathcal{W}_{< \infty})$.
\end{definition}

\begin{definition}
For $\word{\gamma} \in \mathcal{PW}_{< \infty}$ we define $\mathcal{W}_{< \infty}(\word{\gamma})=P^{-1}(\{\word{\gamma}\})$. For $N \ge \| \word{\gamma}\|$ we define $\mathcal{W}_{N}(\word{\gamma}) = \mathcal{W}_{< \infty}(\word{\gamma}) \cap \mathcal{W}_{N}$ and $\mathcal{W}_{\le N}(\word{\gamma}) = \bigsqcup_{k=\| \word{\gamma} \|}^{N} \mathcal{W}_{k}(\word{\gamma}) = \mathcal{W}_{< \infty}(\word{\gamma}) \cap \mathcal{W}_{\le N}$.

\end{definition}

Similarly to Definition \ref{basis of words} , we introduce the notion of basis of words for $\mathcal{W}_{N}(\word{\gamma})$ when $\word{\gamma}$ is a pure word.

\begin{definition}[Basis of words for $\mathcal{W}_N(\word{\gamma})$]
Let $N \in \mathbb{N}$ be a truncation order and $\word{\gamma} \in \mathcal{PW}_{\le N}$. We say that a subset $B \subset \mathcal{W}_{\le N}(\word{\gamma})$ is a basis of words for $\mathcal{W}_N(\word{\gamma})$ if $(\word{w} \shuffle \word{0}_{N - \| \word{w}\|})_{\word{w} \in B}$ is a basis of $\text{Span}(\mathcal{W}_{N}(\word{\gamma}))$.
\end{definition}

For each $\word{\gamma} \in \mathcal{PW}_{\le N}$ and for each $\word{w} \in \mathcal{W}_{\le N}(\word{\gamma})$, the shuffle $\word{w} \shuffle \word{0}_{N - \|\word{w}\|}$ belongs to $\text{Span}(\mathcal{W}_N(\word{\gamma}))$. Moreover, for distinct pure words $\word{\gamma} \neq \tilde{\word{\gamma}}$, the sets $\mathcal{W}_N(\word{\gamma})$ and $\mathcal{W}_N(\tilde{\word{\gamma}})$ are disjoint. As motivated earlier, this allows us to reduce the study of the spanning properties of the $N$-step truncated signature to each class of words sharing the same pure word.

\begin{theorem}
\label{main result -1}
Let $N \in \mathbb{N}$ be a truncation order and $B \subset \mathcal{W}_{\le N}$ a subset of words. The set $B$ is a basis of words for $\mathcal{W}_{N}$ if and only if for each $\word{\gamma} \in \mathcal{PW}_{\le N}$, $B \cap \mathcal{W}_{\le N}(\word{\gamma})$ is a basis of words for $\mathcal{W}_{N}(\word{\gamma})$.
\end{theorem}

\begin{theorem}[Necessary condition]
\label{main result 0}

Let $N \in \mathbb{N}$ be a truncation order and $\word{\gamma} \in \mathcal{PW}_{\le N}$. If a subset $B_{\word{\gamma}}$ of $\mathcal{W}_{\le N}(\word{\gamma})$ is a basis of words for $\mathcal{W}_{N}(\word{\gamma})$ then 

\begin{align*}
\text{Card}(B_{\word{\gamma}}) = \binom{N}{\| \word{\gamma} \|} ~~\text{and}~~ \forall m \in \{\|\word{\gamma}\|,\cdots,N-1\},  \text{Card}(B_{\word{\gamma}} \cap \mathcal{W}_{\le m}) \le  \binom{m}{\| \word{\gamma} \|}.
\end{align*}

\end{theorem}

\begin{remark}
\label{remarque contre exemple necessary condition}
There exist pure words $\word{\gamma} \in \mathcal{PW}_{\le N}$ and subsets of $\mathcal{W}_{\le N}$ which satisfy this condition but are not bases of words for $\mathcal{W}_N(\word{\gamma})$: let us take the pure word $\word{ii}$ with $\word{i} \in \{\word{1},\cdots,\word{d}\}$ and the subset $B_{\word{ii}} = \{\word{i 0 i}, \word{i i 0}, \word{0 i 0 i},\word{0 i i 0}, \word{i 0 0 i}, \word{i 0 i 0}\} \subset \mathcal{W}_{\le 4}(\word{ii})$ which satisfies the assumptions of Theorem \ref{main result 0}, but we have

\begin{align*}
\begin{pmatrix}
\word{i 0 i} \shuffle \word{0} \\
\word{i i 0} \shuffle \word{0} \\
\word{0 i 0 i} \\
\word{0 i i 0} \\
\word{i 0 0 i} \\
\word{i 0 i 0}
\end{pmatrix}&=\begin{pmatrix}
0 & 1 & 0 & 2 & 1 & 0 \\
0 & 0 & 1 & 0 & 1 & 2 \\
0 & 1 & 0 & 0 & 0 & 0 \\
0 & 0 & 1 & 0 & 0 & 0 \\
0 & 0 & 0 & 1 & 0 & 0 \\
0 & 0 & 0 & 0 & 1 & 0
\end{pmatrix}
\begin{pmatrix}
\word{0 0 i i} \\
\word{0 i 0 i} \\
\word{0 i i 0} \\
\word{i 0 0 i} \\
\word{i 0 i 0} \\
\word{i i 0 0}
\end{pmatrix}.
\end{align*}
Since the first column of the matrix vanishes, $B_{\word{ii}}$ is not a basis of words for $\mathcal{W}_{4}(\word{ii})$.

\end{remark}

Since for each $m \in \{0, \cdots, N\}$, we have $\text{Card}(B \cap \mathcal{W}_{\le m}) = \sum_{k=0}^{m}\sum_{\gamma \in \mathcal{PW}_{k}} \text{Card}(B \cap \mathcal{W}_{\le m}(\word{\gamma}))$ with $\text{Card}(\mathcal{PW}_{k})=d^k$, we obtain the following corollary.

\begin{corollary}
\label{cardinal bases}
Let $N \in \mathbb{N}$ be a truncation order, and $B \subset \mathcal{W}_{\le N}$ be a subset of words. If $B$ is a basis of words for $\mathcal{W}_N$ then for each $m \in \{0,\cdots,N-1\}$, $\text{Card}(B \cap \mathcal{W}_{\le m}) \le (d+1)^m$ and $\text{Card}(B)=(d+1)^N$.
\end{corollary}

We now define the sets of words of length $N \in \mathbb{N}$ which do not end by the letter $\word{0}$ (called prefixes) and the sets of words of length $N \in \mathbb{N}$ which do not start by the letter $\word{0}$ (called suffixes). The prefixes (resp. suffixes) which share the same pure word $\word{\gamma}$ are called $\word{\gamma}$-prefixes (resp. $\word{\gamma}$-suffixes).

\begin{definition}[Prefixes]
We define $\mathcal{W}_{0}^{*}=\{ \emptyword \}.$ and $\mathcal{W}_{N}^{*} =\{ \word{\word{i_1}} \cdots \word{\word{i_N}} \in \mathcal{W}_N ~|~ \word{\word{i_N}} \neq \word{0} \}$ for $N \in \mathbb{N}^*$. We also define $\mathcal{W}_{\le N}^{*} = \bigsqcup_{k=0}^{N} \mathcal{W}_{k}^{*}$ for $N \in \mathbb{N}$ and $\mathcal{W}_{< \infty}^{*} = \bigsqcup_{k=0}^{\infty} \mathcal{W}_{k}^{*}$.
\end{definition}

\begin{definition}[$\word{\gamma}$-Prefixes]
For $\word{\gamma} \in \mathcal{PW}_{< \infty}$ and $N \ge \| \word{\gamma}\|$ we define $\mathcal{W}_{N}^{*}(\word{\gamma}) = \mathcal{W}_{N}(\word{\gamma}) \cap \mathcal{W}_{< \infty}^{*}$ and $\mathcal{W}_{\le N}^{*}(\word{\gamma})= \bigsqcup_{k=\|\word{\gamma} \|}^{N}\mathcal{W}_{k}^{*}(\word{\gamma})=\mathcal{W}_{\le N}(\word{\gamma}) \cap \mathcal{W}_{< \infty}^{*}$.
\end{definition}

\begin{definition}[Suffixes]
We define ${}^{*}\mathcal{W}_{0}=\{ \emptyword \}.$ and ${}^{*}\mathcal{W}_{N} =\{ \word{\word{i_1}} \cdots \word{\word{i_N}} \in \mathcal{W}_N ~|~ \word{\word{i_1}} \neq \word{0} \}$ for $N \in \mathbb{N}^*$. We also define ${}^{*}\mathcal{W}_{\le N} = \bigsqcup_{k=0}^{N} {}^{*}\mathcal{W}_{k}$ for $N \in \mathbb{N}$ and ${}^{*}\mathcal{W}_{< \infty} = \bigsqcup_{k=0}^{\infty} {}^{*}\mathcal{W}_{k}$.
\end{definition}

\begin{definition}[$\word{\gamma}$-Suffixes]
For $\word{\gamma} \in \mathcal{PW}_{< \infty}$ and $N \ge \| \word{\gamma}\|$ we define ${}^{*}\mathcal{W}_{N}(\word{\gamma}) = \mathcal{W}_{N}(\word{\gamma}) \cap {}^{*}\mathcal{W}_{< \infty}$ and ${}^{*}\mathcal{W}_{\le N}(\word{\gamma})= \bigsqcup_{k=\|\word{\gamma} \|}^{N} {}^{*}\mathcal{W}_{k}(\word{\gamma})=\mathcal{W}_{\le N}(\word{\gamma}) \cap {}^{*}\mathcal{W}_{< \infty}$.
\end{definition}
 
\begin{theorem}[Sufficient conditions]
\label{main result 1}

Let $N \in \mathbb{N}$ be a truncation order and $\word{\gamma} \in \mathcal{PW}_{\le N}$. We have the following

\begin{enumerate}[label=(\roman*)]
    \item For each family of non-negative integers $(m_{\word{w}})_{\word{w} \in \mathcal{W}_{\le N}^{*}(\word{\gamma})}$ such that for each $\word{\gamma}$-prefix $\word{w} \in \mathcal{W}_{\le N}^{*}(\word{\gamma})$, $\|\word{w}\| + m_{\word{w}} \le N$, the family $\{ \word{w} \word{0}_{m_{\word{w}}} ~|~ \word{w} \in \mathcal{W}_{\le N}^{*}(\word{\gamma}) \}$ is a basis of words for $\mathcal{W}_{N}(\word{\gamma})$.
    \item For each family of non-negative integers $(m_{\word{w}})_{\word{w} \in {}^{*}\mathcal{W}_{\le N}(\word{\gamma})}$ such that for each $\word{\gamma}$-suffix $\word{w} \in {}^{*}\mathcal{W}_{\le N}(\word{\gamma})$, $\|\word{w}\| + m_{\word{w}} \le N$, the family $\{ \word{0}_{m_{\word{w}}}\word{w} ~|~ \word{w} \in {}^{*}\mathcal{W}_{\le N}(\word{\gamma})\}$ is a basis of words for $\mathcal{W}_{N}(\word{\gamma})$.

\end{enumerate}

\end{theorem}

\begin{remark}
To construct a basis of words for $\mathcal{W}_N$, Theorem \ref{main result -1} allows us to freely choose a basis of words for each pure word.
Theorem \ref{main result 1} then provides a systematic way to construct such bases for all pure words. Note that one can choose the construction (i) based on prefixes for some pure words and (ii) based on suffixes for the remaining ones.  
\end{remark}

\begin{remark}
\label{remark contre exemple}
Let $\word{i} \in \{ \word{1}, \cdots, \word{d}\}$. The set $\{\word{i},\word{0i},\word{0i0}\}$ is a basis of words for $\mathcal{W}_{3}(\word{i})$ even if it does not satisfy the sufficient conditions given in the Theorem \ref{main result 1}. This is a consequence of the following more general statement: 

\begin{enumerate}
\item[(i)] Each $B_{\word{i}} \subset \mathcal{W}_{\le N}(\word{i})$ which contains one element with each length $k \in \{1,\cdots,N\}$ is a basis of words for $\mathcal{W}_{N}(\word{i})$.
\end{enumerate}
The proof of this statement relies on an induction argument and on the following property: 

\begin{enumerate}
\item[(ii)] For each truncation order $N \in \mathbb{N}^*$ and each word $\word{\tilde{w}} \in \mathcal{W}_{N+1}(\word{i})$, the set $\mathcal{W}_{N}(\word{i}) \cup \{ \word{\tilde{w}}\}$ is a basis of words for $\mathcal{W}_{N+1}(\word{i})$.  
\end{enumerate}
\end{remark}

In the following we define the length of $B \subset \mathcal{W}_{< \infty}$ as $\length{B} = \sum_{\word{w} \in B} \| \word{w}\|$.

\begin{corollary}
\label{corrolary length}

For any truncation order $N \in \mathbb{N}$, the subsets $\mathcal{W}_{\le N}^{*}$ and ${}^{*}\mathcal{W}_{\le N}$ are bases of words for $\mathcal{W}_{N}$ which have the minimal length.

\end{corollary}

Our fourth main result states that in the case where the path valued random variable is an Itô process solution to SDE with a uniformly elliptic diffusion coefficient, the family of random variables $(\widehat{\mathbb{X}}_{T}^{\word{w}})_{\word{w} \in B}$ is a basis of $\text{Span}(\widehat{\mathbb{X}}_{T}^{\word{w}} ~|~ \word{w} \in \mathcal{W}_{\le N})$ when $B$ a basis of words for $\mathcal{W}_{N}$.

\begin{theorem}
\label{main result 2}
Let $\left(\Omega,\mathcal{F},\mathbb{P}\right)$ be a probability space, $T >0$ be a fixed time horizon, $m,d \in \mathbb{N}^*$, $(\beta_t)_{t \in [0,T]}$ be a $m$-dimensional Brownian motion, and $Y$ be a $\mathbb{R}^d$-valued random vector independent of $(\beta_t)_{t \in [0,T]}$. Let $(X_t)_{t \in [0,T]}:=(X_{t}^{\word{1}},\cdots,X_{t}^{\word{d}})_{t \in [0,T]}$ be the solution to:

\begin{align}
\label{Ito form}
    dX_t &= \sigma(t,X_t) d\beta_t + b(t,X_t) dt, \quad X_0=Y,
\end{align}
where $b: [0,T] \times \mathbb{R}^{d} \rightarrow  \mathbb{R}^d$ is $\mathcal{C}^3$ and $\sigma : [0,T] \times \mathbb{R}^{d} \rightarrow \mathbb{R}^{d \times m}$ is $C^{4}$, both with linear growth. Let $N \in \mathbb{N}$ be a truncation order, and $B \subset \mathcal{W}_{\le N}$ be a basis of words for $\mathcal{W}_{N}$. Then, under the uniform ellipticity condition:
\[
\exists \lambda>0, \forall t \in [0,T], \forall x \in \mathbb{R}^d,\forall \xi \in \mathbb{R}^d  \quad \xi\sigma(t,x) \sigma(t,x)^{\top} \xi^{\top} \ge \lambda\|\xi\|^2,
\]
the family of random variables $(\widehat{\mathbb{X}}^{\word{w}}_{T})_{\word{w} \in B}$ is a basis of $\text{Span}(\widehat{\mathbb{X}}_{T}^{\word{w}} ~|~ \word{w} \in \mathcal{W}_{\le N})$. 
\end{theorem}

\begin{remark}
\label{remark base libre and existence and uniqueness of solution}
\begin{enumerate}[label=(\roman*)]
\item We know that for each truncation order $N \in \mathbb{N}$, $\text{dim}(\text{Span}(\widehat{\mathbb{X}}_{T}^{\word{w}} ~|~ \word{w} \in \mathcal{W}_{\le N})) \le (d+1)^N$. If we exhibit a particular basis of words $\tilde{B}$ for $\mathcal{W}_N$ (which necessarily has the cardinality $\text{Card}(\tilde{B})=(d+1)^N$ by Corollary \ref{cardinal bases}) and such that $(\widehat{\mathbb{X}}_{T}^{\word{w}})_{\word{w} \in \tilde{B}}$ is linearly independent for the almost-sure equality, then it immediately comes that $\text{dim}(\text{Span}(\widehat{\mathbb{X}}_{T}^{\word{w}} ~|~ \word{w} \in \mathcal{W}_{\le N})) = (d+1)^N$, and for any basis of words $B$ for $\mathcal{W}_N$, the family $(\widehat{\mathbb{X}}_{T}^{\word{w}})_{\word{w} \in B}$ is a basis of $\text{Span}(\widehat{\mathbb{X}}_{T}^{\word{w}} ~|~ \word{w} \in \mathcal{W}_{\le N})$. We show that $(\widehat{\mathbb{X}}_{T}^{\word{w}})_{\word{w} \in \mathcal{W}_{< \infty}^{*}}$ is a family of random variables linearly independent for the almost-sure equality.  This ensures that for each $N \in \mathbb{N}$ the family of random variables $(\widehat{\mathbb{X}}_{T}^{\word{w}})_{\word{w} \in \mathcal{W}_{\le N}^{*}}$ is linearly independent for the almost-sure equality.
\item When $\sigma \equiv 1$, the result still holds when $b$ is only measurable with linear growth in space uniformly in time. Indeed, in that case \eqref{Ito form} admits a weak solution \citep{karatzas2014brownian}[Proposition V.3.6], which is unique in law \citep{karatzas2014brownian}[Proposition V.3.10], and the measure $\mathbb P^*$ obtained from $\mathbb P$ by Girsanov's
transform, namely $\frac{d \mathbb{P}^*}{d \mathbb{P}}
  = \exp\left(-\int_{0}^{T} b(t,X_t) \cdot d\beta_t - \frac{1}{2} \int_{0}^{T}
  \|b(t,X_t) \|^2\, dt\right)$, is equivalent to $\mathbb P$ and removes the drift. More precisely, under $\mathbb{P}^*$, the process
$(X_t-Y)_{t \le T}$ is a Brownian motion independent of $Y$. Since the linear independence of the
family $(\widehat{\mathbb{X}}^{\word{w}}_{T})_{\word{w} \in \mathcal{W}_{< \infty}^{*}}$ is unaffected
by the change to the equivalent measure $\mathbb P^*$, and since the time-augmented signature is
invariant under translations of the path, the statement reduces to the case of a Brownian motion, which is the content of Proposition \ref{prop brownian}.
\end{enumerate}
\end{remark}

Our fifth main result states that the linear independence of the terminal value of the time-augmented signature components still holds when the considered process is the piecewise linear interpolation of the solution to the SDE \eqref{Ito form}, under assumptions of Theorem \ref{main result 2}. 
\begin{theorem}
\label{main result 3}
Let $T>0$, $(X_t)_{t \le T}$ be the solution to \eqref{Ito form}, $N \in \mathbb{N}^*$ be a truncation order, $B \subset \mathcal{W}_{\le N}$ be a basis of words for $\mathcal{W}_{N}$ and $(D_n = \{0=t_0 < \cdots < t_n=T \})_{n \in \mathbb{N}^*}$ be a sequence of dissections of $[0,T]$ such that $|D_n|:=\sup\limits_{0 \le k \le n-1 } |t_{k+1} - t_k| \rightarrow 0$. We denote by $X^{(n)}$ the piecewise affine interpolation of $X$ on $D_n$, and $\left(\widehat{\mathbb{X}^{(n)}}_{t}\right)_{t \le T}$ the signature of $(\widehat{X}^{(n)}_t)_{t \le T}:=(t,X^{(n)}_t)_{t \le T}$. Then under the assumptions of Theorem \ref{main result 2}, there exists $n_0 \in \mathbb{N}$ such that for all $n \ge n_0$, $\left(\widehat{\mathbb{X}^{(n)}}_{T}^{\word{w}}\right)_{\word{w} \in B}$ is a basis of $\text{Span}\left(\widehat{\mathbb{X}^{(n)}}_{T}^{\word{w}}~|~ \word{w} \in \mathcal{W}_{\le N}\right)$.

\end{theorem}

\begin{remark}
For a linear regression purpose, it is necessary to have the terminal value of the time-augmented signature components in $L^2$ for \eqref{loss} to be well defined. It is possible to show by induction on $N  \in \mathbb{N}$ that, if $Y \in L^{2N}$, then for each word $\word{w} \in \mathcal{W}_N$, we have $\|\widehat{\mathbb{X}}^{\word{w}}\|_{\infty} \in L^2$ and for all $n \in \mathbb{N}$, $\|\widehat{\mathbb{X}^{(n)}}^{\word{w}}\|_{\infty} \in L^2$.  
\end{remark}

\subsection{Optimal computation of signature components associated to a basis of words}

Our goal is to efficiently compute the terminal value of the truncated time-augmented signature of a piecewise affine path. Since an explicit formula exists for the signature of an affine path, we combine it with Chen’s relation when dealing with concatenated paths. For each word $\word{w} \in \mathcal{W}_{< \infty}$ we denote by $S^{\word{w}}$ the map that associates to a path of finite variation $(X_t)_{t \in [0,T]}$  the terminal value $\widehat{\mathbb{X}}_{T}^{\word{w}}$ of the corresponding component of its time-augmented signature.

\begin{proposition}[Time-augmented signature of an affine path]
\label{signature segment}

Let $T>0$, $X:=(X_{t}^{\word{1}},\cdots,X_{t}^{\word{d}})_{t \in [0,T]}$ be an affine path taking values in $\mathbb{R}^d$. Then for all $N \in \mathbb{N}^*$ and for all $\word{w}=\word{i_1 \cdots i_N} \in \mathcal{W}_N$ we have the following equality
\[
    S^{\word{w}}(X) = \frac{1}{N!} \prod_{k=1}^{N} \left(\widehat{X}_{T}^{\word{i_k}} - \widehat{X}_{0}^{\word{i_k}} \right), 
\]
where $\widehat{X}_t^{\word{0}}=t$ and $\widehat{X}_{t}^{\word{i}}=X_t^{\word{i}}$ for $\word{i} \in \{ \word{1}, \cdots, \word{d} \}$.
    
\end{proposition} 

We define the concatenation of two paths.

\begin{definition}[Concatenation of paths]

Let $T_1,T_2 >0$, and $X : [0,T_1] \rightarrow \mathbb{R}^d$, $Y : [0,T_2] \rightarrow \mathbb{R}^d$ be two paths. We denote by $X*Y : [0,T_1+T_2] \rightarrow \mathbb{R}^d$ the concatenation of $X$ and $Y$ defined as

\[
X*Y : t \mapsto \begin{cases} X_t ~~\text{if}~~ t \in [0,T_1] \\
    Y_{t-T_1} - Y_0 +X_{T_1} ~~\text{if}~~ t \in [T_1,T_1 + T_2]\end{cases}
\]

\end{definition}

\begin{proposition}[Chen's relation]
\label{chen relation}

Let $T_1,T_2 >0$, and $X : [0,T_1] \rightarrow \mathbb{R}^d$, $Y : [0,T_2] \rightarrow \mathbb{R}^d$ be two paths of finite variation. For any word $\word{i_1 \cdots i_m}$ we have the following relation

\[
S^{\word{i_1 \cdots i_m}}(X*Y) = \sum_{k=0}^{m} S^{\word{i_1 \cdots i_k}}(X) S^{\word{i_{k+1} \cdots i_m}}(Y),
\]
with the convention $\word{i_1 \cdots i_k} = \emptyword$ if $k=0$, and $\word{i_{k+1} \cdots i_m}=\emptyword$ if $k=m$.

\end{proposition}

In the following, we consider a time horizon $T >0$, an integer $K \ge 2$, and $X = \Delta X_1* \cdots *\Delta X_K$ defined as the concatenation of $K$ affine paths $\Delta X_k : [0,\frac{T}{K}] \rightarrow \mathbb{R}^d$. We denote by $X_{|k} = \Delta X_1 * \cdots *\Delta X_k$ and $\prescript{}{k|}{X} = \Delta X_{K-k+1} * \cdots *\Delta X_K$ the forward and backward partial concatenations, respectively. By Proposition \ref{chen relation}, the terminal value of the $N$-step truncated time-augmented signature of $X$ can be computed either via a forward or via a backward recursive procedure. For clarity of exposition, we focus on the forward approach, as the backward approach is entirely analogous. We first compute the terminal value of the $N$-step truncated time-augmented signature of the initial segment $X_{|1}=\Delta X_1$ using Proposition \ref{signature segment}. Then, for each $k \in \{1,\cdots,K-1\}$, we iteratively compute the signature of $X_{|k+1}=X_{|k}*\Delta X_{k+1}$ by applying Chen’s relation from Proposition \ref{chen relation}, until we obtain the signature of the full path $X = X_{|K}$. In particular, for each $m \in \{1,\cdots,N\}$ and each word $\word{i_1 \cdots i_m} \in \mathcal{W}_{m}$ we have

\[
S^{\word{i_1 \cdots i_m}}(X_{|k+1}) = \sum_{j=0}^{m} S^{\word{i_1 \cdots i_j}}(X_{|k})S^{\word{i_{j+1} \cdots i_m}}(\Delta X_{k+1}).
\]
For the forward approach, the computation of $S^{\word{i_1 \cdots i_m}}(X_{|k+1})$ requires 

\begin{enumerate}[label=(\roman*)]
    \item computing $(S^{\word{i_1 \cdots i_j}}(X_{|k}))_{j \in \{0,\cdots, m\}}$,
    \item computing $(S^{\word{i_{j+1} \cdots i_m}}(\Delta X_{k+1}))_{j \in \{0,\cdots, m\}}$,
    \item performing $2m$ elementary operations. This corresponds to parsing the $m+1$ elements in the summation, and performing $m-1$ products (since there is no products for $j=0$ and $j=m$). 
\end{enumerate}

In the following, we simplify the analysis of the total computational cost by assuming that the computational cost of $S^{\word{i_1 \cdots i_m}}(\Delta X_k)$ is equal to one, whatever the value of $m$. We denote by $\mathcal{P}(\mathcal{W}_{\infty})$ the power set of $\mathcal{W}_{<\infty}$ and by $\Phi_{\text{f}} : \mathcal{P}(\mathcal{W}_{<\infty}) \rightarrow \mathcal{P}(\mathcal{W}_{< \infty})$ defined as follows

\begin{align*}
    \Phi_{\text{f}}(B) = \{ \word{w} \in \mathcal{W}_{< \infty} ~|~ \exists \word{v} \in \mathcal{W}_{<\infty}, \word{wv} \in B\}.
\end{align*}
Since $\emptyword \in \mathcal{W}_{< \infty}$, we have $B \subset \Phi_{\text{f}}(B)$. Thus $\Phi_{\text{f}}(B) \subset \Phi_{\text{f}} \circ \Phi_{\text{f}}(B)$. By definition of $\Phi_{\text{f}}$, we also have $\Phi_{\text{f}} \circ \Phi_{\text{f}}(B) \subset \Phi_{\text{f}}(B)$, so $\Phi_{\text{f}} \circ \Phi_{\text{f}} = \Phi_{\text{f}}$. For $k \in \{2,\cdots,K\}$, and $B \subset \mathcal{W}_{<\infty}$, the computation of $(S^{\word{w}}(X_{|k}))_{\word{w} \in B}$ requires (i) computing $(S^{\word{w}}(X_{|k-1}))_{\word{w} \in \Phi_{\text{f}}(B)}$, (ii) performing $\sum_{\word{w} \in B} 2\|\word{w}\|=2\length{B}$ elementary operations. We denote by $C(B,k)$ the required number of elementary operations to compute $(S^{\word{w}}(X_{|k}))_{\word{w} \in B}$. By combining (i) and (ii) we obtain the following recursive formula

\[
C(B,k)=C(\Phi_{\text{f}}(B),k-1) + 2\length{B}.
\]
Since $\Phi_{\text{f}} \circ \Phi_{\text{f}}=\Phi_{\text{f}}$ and $C(B,1)=\text{Card}(B)$, we obtain the following formula for each $k \ge 2$

\[
    C(B,k) = \text{Card}(\Phi_{\text{f}}(B)) + 2(k-2)\length{\Phi_{\text{f}}(B)}  + 2\length{B}.
\]
We can use similar recursive arguments for the cost of the backward approach replacing $\Phi_{\text{f}}$ by $\Phi_{\text{b}} : \mathcal{P}(\mathcal{W}_{<\infty}) \rightarrow \mathcal{P}(\mathcal{W}_{< \infty})$ defined as follows

\[
    \Phi_{\text{b}}(B) = \{ \word{w} \in \mathcal{W}_{< \infty} ~|~ \exists \word{v} \in \mathcal{W}_{<\infty}, \word{vw} \in B\}.
\]
Finally, we obtain formulas for the total cost for the forward and backward approaches.

\begin{align}
\textbf{Forward approach: }  C_{\text{f}}(B,K) &= \text{Card}(\Phi_{\text{f}}(B)) 
+ 2(K-2)\length{\Phi_{\text{f}}(B)}  
+ 2\length{B}.
\label{eq:Cf} \\
\textbf{Backward approach: }C_{\text{b}}(B,K) &= \text{Card}(\Phi_{\text{b}}(B)) 
+ 2(K-2)\length{\Phi_{\text{b}}(B)} 
+ 2\length{B}.
\label{eq:Cb}
\end{align}

\begin{theorem}[Optimality of prefixes and suffixes]
\label{main result 4}
For any truncation order $N \in \mathbb{N}$, the sets $\mathcal{W}_{\le N}^{*}$ and ${}^{*}\mathcal{W}_{\le N}$ form bases of words for $\mathcal{W}_{N}$ with minimal computational cost when using, respectively, the backward and the forward approach. 
\end{theorem}

\subsection{The case of Exponential Fading Memory Signature}

We highlight that the notion of bases of words for $\mathcal{W}_N$ also plays a key role in the context of the recently introduced (time-augmented) Exponentially Fading Memory (EFM) signature \citep{abi2025exponentially}. For a time-augmented continuous path $(\widehat{X}_t)_{t \in (-\infty,T]} = (t,X_t^{\word{1}},\cdots, X_t^{\word{d}})_{t \in (-\infty,T]}$, a vector $\mathbf{\lambda}=(\lambda_{\word{0}},\cdots,\lambda_{\word{d}})$ with positive entries, a word $\word{i_1 \cdots i_n} \in \mathcal{W}_n$ and $t \in (-\infty,T]$, these authors consider the following quantities:

\begin{align*}
\begin{cases}
\widehat{\mathbb{X}}^{\mathbf{\lambda}, \word{i_1 \cdots i_n}}_t = \int_{-\infty <t_1<\cdots<t_n<t} e^{- \lambda_{\word{i_1}} (t-t_1)} d\widehat{X}_{t_1}^{\word{i_1}} \cdots  e^{- \lambda_{\word{i_n}} (t-t_n)} d\widehat{X}_{t_n}^{\word{i_n}}=\int_{-\infty <t_n <t}e^{-\lambda_{\word{i_1 \cdots i_n}}(t-t_n)}\widehat{\mathbb{X}}_{t_n}^{\mathbf{\lambda},\word{i_1 \cdots i_{n-1}}} d\widehat{X}_{t_n}^{\word{i_n}}, \\
\widehat{\mathbb{X}}^{\mathbf{\lambda},\emptyword}_t = 1,
\end{cases}
\end{align*}
with $\lambda_{\word{i_1\cdots i_n}}:=\sum_{k=1}^{n} \lambda_{\word{i_k}}$. The time-augmented EFM signature is defined as $\widehat{\mathbb{X}}^{\mathbf{\lambda}}:=(\widehat{\mathbb{X}}^{\mathbf{\lambda}, \word{w}})_{\word{w} \in \mathcal{W}_{<\infty}}$ As for the classical signature, we can define the time-augmented EFM signature for continuous semimartingales using the Stratonovich integration rule. Since $\widehat{\mathbb{X}}_{t}^{\mathbf{\lambda}, \word{0_k}} = \frac{1}{\lambda_{\word{0}}^k k!}>0$ for all $t \in (-\infty,T]$, using the shuffle product property of the time-augmented EFM signature \citep{abi2025exponentially}[Proposition 4.8], we obtain

\begin{align}
\label{egalité bracket efm}
\forall \word{w} \in \mathcal{W}_{< \infty},\quad \left<\word{w},\widehat{\mathbb{X}}^{\mathbf{\lambda}}_t\right> = \left<\lambda_{\word{0}}^k k! (\word{w} \shuffle \word{0_k}),\widehat{\mathbb{X}}^{\mathbf{\lambda}}_t\right>, \forall k \in \mathbb{N}, \forall t \in (-\infty,T].
\end{align}

The equality \eqref{egalité bracket efm} is stronger than its counterpart \eqref{egalité bracket} for the classical time-augmented signature since it is uniform in time (which is not the case for \eqref{egalité bracket} due to the presence of the factor $\frac{k!}{T^{k}}$). When considering either a $\mathbb{R}^d$-valued path or a $\mathbb{R}^d$-valued stochastic process $(X_t)_{t \in (-\infty,T]}$, equation \eqref{egalité bracket efm} implies that 
\[
\text{Span}(\widehat{\mathbb{X}}^{\mathbf{\lambda},\word{w}} ~|~ \word{w} \in \mathcal{W}_{\le N})=\text{Span}(\widehat{\mathbb{X}}^{\mathbf{\lambda},\word{w}} ~|~ \word{w} \in \mathcal{W}_{N})
\] 
where the linear span is understood up to $dt$-a.e. equality in the deterministic case or $dt \otimes d\mathbb{P}$-a.e. equality in the stochastic case. By replacing: (i) \eqref{egalité bracket} with \eqref{egalité bracket efm}, (ii) $\widehat{\mathbb{X}}_T$ with $\widehat{\mathbb{X}}^{\mathbf{\lambda}}$, (iii) $\mathbb{P}$-a.s. equality with either $dt$-a.e. equality or $dt \otimes d \mathbb{P}$-a.e. equality, we see that the motivations 1 and 2 of the bases of words for $\mathcal{W}_N$ are preserved for the time-augmented EFM signature. Working with bases of $\text{Span}(\widehat{\mathbb{X}}^{\mathbf{\lambda},\word{w}} ~|~ \word{w} \in \mathcal{W}_{\le N})$ is relevant in numerous applications of the time-augmented EFM signature like the approximation of a functional using a single path. In a
linear regression built on one single sample path, the relevant design space is the
space of time functions
\[
[0,T] \ni t\longmapsto
\widehat{\mathbb X}^{\mathbf\lambda,\word w}_t(\omega).
\]
One cannot use information coming from other sample paths to remove linear dependencies along this path. It is therefore the $\mathbb{P}$-a.s. pathwise linear independence for the
$dt$-a.e. equality which is relevant rather than the weaker linear independence for the $dt \otimes d\mathbb{P}$-a.e. equality. The following theorem provides a sufficient condition for this property.

\begin{theorem}
\label{main result EFM}
Let $m \in \mathbb{N}^*$ and $(\Omega,\mathcal{F},\mathbb{P})$ be a probability space, on which we define a $m$-dimensional Brownian motion $(W_t)_{t \in [0,T]}$ with its natural filtration $(\mathcal{F}_t)_{t \in [0,T]}$. Let $(X_t)_{t \in (-\infty,T]}$ be a stochastic process taking values in $\mathbb{R}^{d}$ with $d\le m$, and $\eta \in \mathcal{C}^{1-\text{var}}((-\infty,0],\mathbb{R}^d)$ be a continuous path of finite variation, such that for all $t \in (-\infty,T]$:
\begin{align*}
X_t =
\begin{cases}
\eta_t &\text{ if } t \le 0, \\
\eta_{0} + \int_{0}^{t} b_s ds + \int_{0}^{t} \sigma_s \cdot \circ dW_s &\text{ if } t > 0,
\end{cases}
\end{align*}
with $(b_t)_{t \in [0,T]}$ and $(\sigma_t)_{t \in [0,T]}$ progressively measurable processes respectively taking values in $\mathbb{R}^d$ and in $\mathbb{R}^{d\times m}$, such that $\mathbb{P}\left[\int_{0}^{T} \|b_t\|dt+\int_{0}^{T}\operatorname{Tr}(\sigma_t\sigma_{t}^{\top})dt < \infty\right]=1$. Let $N \in \mathbb{N}$ be a truncation order. Let $B \subset \mathcal{W}_{\le N}$ be a basis of words for $\mathcal{W}_N$. Assume that there exists a Borel set $C \in \mathcal{B}([0,T])$ such that $\mathcal{L}(C\cap I)>0$ for each non-empty open interval $I\subset[0,T]$, and such that $\sigma_t \sigma_t^{\top}$ is non-singular $\mathbbm{1}_{C}(t)dt \otimes d\mathbb{P}$-a.e.. Then for $\mathbb{P}$-almost all $\omega \in \Omega$, the family $(\widehat{\mathbb{X}}^{\mathbf{\lambda},\word{w}}(\omega))_{ \word{w} \in B}$ is a basis of $\text{Span}(\widehat{\mathbb{X}}^{\mathbf{\lambda},\word{w}}(\omega) ~|~ \word{w} \in \mathcal{W}_{\le N})$ for the $\mathbbm{1}_{C}(t)dt$-a.e. equality. 
\end{theorem}

\begin{remark}
\label{remark EFM}
Like in Remark \ref{remark base libre and existence and uniqueness of solution}, it suffices to provide a particular basis $B$ of words for $\mathcal{W}_N$ such that,  $\mathbb{P}$-a.s., the family $(\widehat{\mathbb{X}}^{\mathbf{\lambda},\word{w}})_{ \word{w} \in B}$ is linearly independent for the $\mathbbm{1}_{C}(t)dt$-a.e. equality. We will show that $\mathcal{W}_N$ is such a set.
\end{remark}

\section{Numerical experiments}
\label{part2}

Many problems can be formulated as follows: 

\begin{align*}
Y = F( (X_t)_{t \in [0,T]}),    
\end{align*}

where $(X_t)_{t \in [0,T]}$ is some input stochastic process taking values in $\mathbb{R}^d$ and $F$ is some (possibly unknown) functional defined on a path space $\Lambda_T \subset (\mathbb{R}^d)^{[0,T]}$ \cite{levin2016learningpastpredictingstatistics,bayer2025pricingamericanoptionsrough,arribas2020sig,fermanian2022functional}. Motivated by universal approximation theorems which ensure that linear functionals of the truncated time-augmented signature can approximate any continuous functionals of paths arbitrarily well as the truncation order increases \cite{cuchiero2022universal,cuchiero2023global,bayer2023primal}, we can approximate the random variable $Y$ by a finite linear combination of components of the truncated signature terminal value of $(t,X_t)_{t \in [0,T]}$, instead of applying the functional $F$ to the input signal. This can be relevant if $F$ is unknown or if $F$ is known but computationally burdening. The knowledge of universal subfamilies of components of the truncated signature terminal value which are bases of the linear span of all components can naturally be leveraged for several purposes in this linear approximation framework:

\begin{enumerate}
    \item the optimal model is identifiable, 
    \item the storage cost of the features is reduced,
    \item the computation time during the training phase is reduced,
    \item the prediction performance is preserved in theory, while empirical improvements may arise in practice due to improved conditioning, reduced variance, or more stable cross-validation.
\end{enumerate}

The identifiability of the model and the reduction in storage cost are direct consequences of the extraction of a basis of regressors. In particular, for a stochastic process taking values in $\mathbb{R}^d$ and a given truncation order $N \in \mathbb{N}$, the cardinality of any basis of the linear span of the components of the terminal value of its truncated time-augmented signature is at most $(d+1)^N$ whereas the cardinality of all components $\frac{(d+1)^{N+1} - 1}{d}$. We denote by $\mathcal{G}_{d,N}$ the ratio of storage cost when switching from the full set of features to a basis. We then have

\begin{align*}
    \mathcal{G}_{d,N}:=\frac{d(d+1)^{N}}{(d+1)^{N+1}-1} \underset{N \to \infty}{\longrightarrow} \frac{d}{d+1},
\end{align*}

which is minimal for $d=1$: the gain of storage cost is maximal when we consider one-dimensional stochastic processes, and this gain is decreasing with the dimension of the process. In the following, we study the impact on the computation time and the performance of such an a priori selection of the features when performing a linear regression. We decide to study this problem for the following functional $F$

\begin{align*}
F((X_t)_{t \in [0,T]})=\beta_{\text{true}}^* \widehat{\mathbb{X}}_{T}^{\le N_{\text{true}}},   
\end{align*}

where $(X_t)_{t \in [0,T]} \in \Lambda_T$, $N_{\text{true}} \in \mathbb{N}^*$ and $\beta_{\text{true}} \in \mathbb{R}^{\frac{(d+1)^{N_{\text{true}}+1} - 1}{d}}$ denotes the (column) parameter vector.

Note that this particular choice of functional is not restrictive if $N_{\text{true}}$ is large enough because any continuous functional can be arbitrarily well approximated by functionals of this form according to universal approximation theorems. Motivated by Theorem \ref{main result 4}, we will use the components of the truncated time-augmented signature associated with the set of suffixes to obtain a basis of features. In what follows we will use the notations:

\begin{itemize}
    \item $n \in \mathbb{N}^*$ denotes the sample size,
    \item $N \in \{2,\cdots,6\} \subset  \{1, \cdots, N_{\text{true}} \}$ denotes the chosen truncation order,
    \item $\mathbf{X}_{\text{all}} \in \mathbb{R}^{\frac{(d+1)^{N+1}-1}{d} \times n}$ denotes the design matrix where the columns $(\mathbf{X}_{\text{all},\cdot,k})_{1 \le k \le n}$ are independent and identically distributed samples of $\widehat{\mathbb{X}}_{T}^{\le N} \in \mathbb{R}^{\frac{(d+1)^{N+1}-1}{d}}$,
    \item $\mathbf{X}_{\text{suffix}} \in \mathbb{R}^{(d+1)^N \times n}$ denotes the design matrix where the columns $(\mathbf{X}_{\text{suffix},\cdot,k})_{1 \le k \le n}$ are independent and identically distributed samples of $(\widehat{\mathbb{X}}_{T}^{\word{w}})_{\word{w} \in {}^{*}\mathcal{W}_{\le N}} \in \mathbb{R}^{(d+1)^N}$,
    \item $\mathbf{Y} \in \mathbb{R}^{1 \times n}$ denotes the output row vector where the entries $(\mathbf{Y}_k)_{1 \le k \le n}$ are independent and identically distributed samples of $Y$.
\end{itemize}

In particular, each observation $(\mathbf{X}_{\text{all},\cdot,k},Y_{k})$ has the same distribution as $( \widehat{\mathbb{X}}_{T}^{\le N}, F((X_t)_{t \in [0,T]})$ and each $(\mathbf{X}_{\text{suffix},\cdot,k},Y_{k})$ has the same distribution as $( (\widehat{\mathbb{X}}_{T}^{\word{w}})_{\word{w} \in {}^{*}\mathcal{W}_{\le N}}, F((X_t)_{t \in [0,T]})$. For our experiments, we used the Python library iisignature \cite{iisignature} which implements the computation of the signature and scikit-learn \footnote{\url{https://scikit-learn.org}} which implements the Ridge regression with a cross-validation of the regularization hyperparameter. Before going further, we define the following 

\begin{align*}
\hat{\beta}_{\text{all}}(\lambda)&:= \left(\frac{1}{n} \mathbf{X}_{\text{all}}\mathbf{X}_{\text{all}}^{*} + \lambda I_{\frac{(d+1)^{N+1}-1}{d}}\right)^{-1}\frac{\mathbf{X}_{\text{all}}\mathbf{Y}^* }{n}, \\
\hat{\beta}_{\text{suffix}}(\lambda)&:= \left(\frac{1}{n} \mathbf{X}_{\text{suffix}}\mathbf{X}_{\text{suffix}}^{*}+ \lambda I_{(d+1)^{N}}\right)^{-1}\frac{\mathbf{X}_{\text{suffix}}\mathbf{Y}^{*}}{n},
\end{align*}

which are the Ridge estimators of the models taking all the features (resp. only those associated with suffixes) with regularization parameter $\lambda \ge 0$. We recall that we recover the Ordinary Least Squares estimator when $\lambda=0$. For a given set of features, the search for the optimal model consists in computing $\hat{\beta}_{\text{all}}(\hat{\lambda}_{\text{all}})$ and $\hat{\beta}_{\text{suffix}}(\hat{\lambda}_{\text{suffix}})$ where the optimal regularization parameters $\hat{\lambda}_{\text{all}}$ and $\hat{\lambda}_{\text{suffix}}$ are computed with an efficient Leave-One-Out cross-validation technique. Before fitting the Ridge models, the features are standardized to have zero mean and unit variance. In short, the only difference between both models lies in the features set; the regression procedure, regularization strategy and cross-validation protocol are otherwise identical. We recall that for every sample size $n$, the empirical variance matrix $\frac{1}{n}\mathbf{X}_{\text{all}}\mathbf{X}_{\text{all}}^{*}$ is almost-surely singular and the OLS estimator ($\lambda=0)$ is not well defined. It is standard to introduce a regularization parameter $\lambda>0$ to obtain a well defined Ridge estimator. On the other hand, the empirical variance matrix $\frac{1}{n}\mathbf{X}_{\text{suffix}}\mathbf{X}_{\text{suffix}}^{*}$ is almost-surely invertible for a sufficiently large sample size $n$ and the OLS estimator is well defined. However, due to the high dimensionality and strong correlations of the signature features, the OLS estimator may suffer from high variance and poor generalization performance. Introducing two positive regularization parameters $\lambda_{\text{all}}$ and $\lambda_{\text{suffix}}$ for the computation of $\hat{\beta}_{\text{all}}(\lambda_{\text{all}})$ and $\hat{\beta}_{\text{suffix}}(\lambda_{\text{suffix}})$ allows us to control their variance and can significantly improve their out-of-sample prediction accuracies. The grid used for the cross-validation of $\lambda_{\text{all}}$ and $\lambda_{\text{suffix}}$ is the following: $\{0\}\cup \bigl\{ 10^{k_i} \;\big|\; k_i = -2 + i \,\frac{6}{99},\ i = 0,1,\dots,99 \bigr\}$. We will see further in Section \ref{gen error analysis} that the cross-validated regularization parameter $\hat{\lambda}_{\text{suffix}}$ is much smaller than the cross-validated regularization parameter $\hat{\lambda}_{\text{all}}$ and may be null in some cases. In the following we will respectively consider a Brownian motion $(W_t)_{t \in [0,1]}$ and an Ornstein-Uhlenbeck process defined on the time interval $[0,1]$ taking their values in $\mathbb{R}$: these two stochastic processes satisfy the assumptions of Theorem \ref{main result 2}. The Ornstein-Uhlenbeck process is the solution to $dX_t = -(X_t+1)dt + dW_t$ starting from the initial condition $X_0=0$. The paths are discretized on the uniform grid $(t_k)_{0 \le k \le K}$ defined by $t_k = \frac{k}{K}$ for $K=100$. For the construction of the output, we consider $N_{\text{true}}=10$ and $\beta_{\text{true}}$ takes the following values $\beta_{\text{true}}=(1,1,\cdots,1)^*$, $\beta_{\text{true}}=(1,2,\cdots,2^{11}-1)^*$ and $\beta_{\text{true}}=(2^{11}-1,\cdots,2,1)^*$. We choose these parameters to test models that respectively assign: (i) equal weight to higher and lower order terms, (ii) more weight on higher order terms, and (iii) more weight on lower order terms.

\subsection{Training time analysis} 

We can decompose the training time into: (i) the computation of the signature terms $\widehat{\mathbb{X}}_{T}^{\le N}$ (resp. $(\widehat{\mathbb{X}}_{T}^{\word{w}})_{\word{w} \in {}^{*}\mathcal{W}_{\le N}}$), and (ii) the computation of $\hat{\beta}_{\text{all}}(\hat{\lambda}_{\text{all}})$ (resp. $\hat{\beta}_{\text{suffix}}(\hat{\lambda}_{\text{suffix}})$). Using the total cost formula for the forward approach \eqref{eq:Cf}, we obtain the following theoretical ratio of computational costs 

\begin{align*}
    \frac{C({}^{*}\mathcal{W}_{\le N},K)}{C(\mathcal{W}_{\le N},K)} &= \frac{\text{Card}({}^{*}\mathcal{W}_{\le N}) + 2(K-1)\length{{}^{*}\mathcal{W}_{\le N}}}{\text{Card}(\mathcal{W}_{\le N}) + 2(K-1)\length{\mathcal{W}_{\le N}}} \\
    &\underset{K \to \infty}{\longrightarrow} \frac{\length{ {}^{*}\mathcal{W}_{\le N} }}{\length{\mathcal{W}_{\le N}}} = \frac{d}{d+1}=\frac{1}{2}.
\end{align*}

We provide in Figure \ref{figure::sig_suffix} the observed computation times for the signature terms, together with a comparison between the empirical scaling ratio and the corresponding theoretical ratio. 

\begin{figure}[H]
\caption{\textbf{Computation time of signature components and comparison with theoretical scaling.} We sample $200$ batches of size $n=500$ of discretized Brownian motion sample paths, and we compute the empirical mean over the batches of the time required to obtain the signature components.}

    \centering
    \subfloat[Computation time.]{%
        \includegraphics[width=0.45\linewidth]{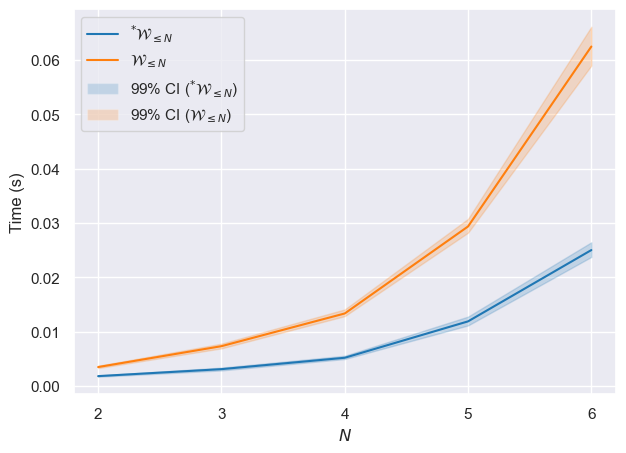}
    }\hfill
    \subfloat[Computation time ratio.]{%
        \includegraphics[width=0.45\linewidth]{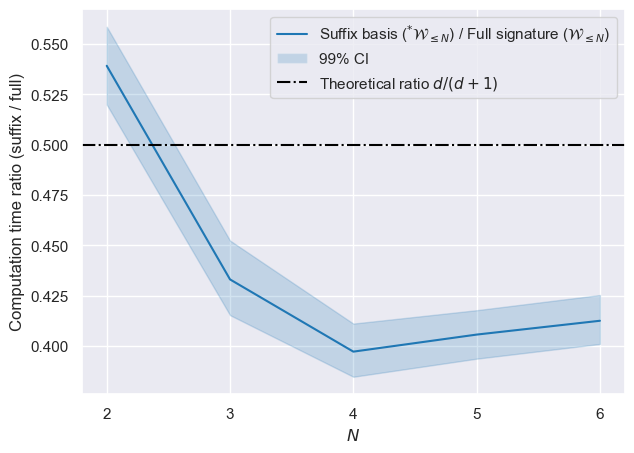}
    }
\label{figure::sig_suffix}
\end{figure}

We observe that the empirical computation time of the suffix-based signature is lower than the theoretical prediction suggests, in the sense that the measured time ratio is smaller than the expected factor $\frac{1}{2}$. This additional gain can be attributed to several practical effects that are not captured by the operation count underlying the theoretical
analysis. The theoretical complexity estimate assumes a uniform cost for all elementary algebraic operations, whereas the general signature implementation incurs additional overheads related to data structure management, memory
allocations and intermediate copies. We now turn to point (ii), namely the computational cost of fitting the Ridge regression model. We assume that $n \ge 2^{N+1} - 1$ for all $N \in \{2,\cdots,6\}$. In this regime, reducing the number of features from $\mathcal{W}_{\le N}$ to ${}^{*}\mathcal{W}_{\le N}$ leads to a proportional reduction in the model fitting time, as confirmed by the numerical experiments 
reported in Figure \ref{figure::fitting_time}. \footnote{Although the theoretical worst-case complexity of ridge regression scales quadratically with the number 
of features, the observed scaling in our experiments is close to linear scaling due to the moderate 
features dimension considered.}

\begin{figure}[H]
\caption{\textbf{Computation time for $\hat{\beta}_{\text{all}}(\hat{\lambda}_{\text{all}})$ and $\hat{\beta}_{\text{suffix}}(\hat{\lambda}_{\text{suffix}})$.} We sample $200$ batches of size $n=500$ of discretized Brownian motion paths, and we compute the empirical mean over the batches of the time required to obtain $\hat{\beta}_{\text{all}}(\hat{\lambda}_{\text{all}})$ (resp. $\hat{\beta}_{\text{suffix}}(\hat{\lambda}_{\text{suffix}})$) once the design matrix $\mathbf{X}_{\text{all}}$ (resp. $\mathbf{X}_{\text{suffix}}$) is computed. Note that the fitting time does not depend on the choice of $\beta_{\text{true}} \in \{(1,1,\cdots,1)^*, (1,2,\cdots,2^{11} - 1)^*, (2^{11}-1,\cdots,2,1)^*\}$ nor on whether $(X_t)_{t \in [0,T]}$ is a Brownian motion or an Ornstein-Uhlenbeck process.}

    \centering
    \subfloat[Fitting time.]{%
        \includegraphics[width=0.45\linewidth]{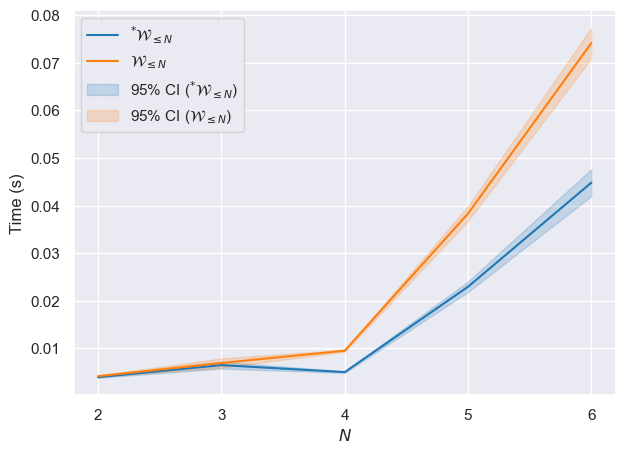}
    }\hfill
    \subfloat[Fitting time ratio.]{%
        \includegraphics[width=0.45\linewidth]{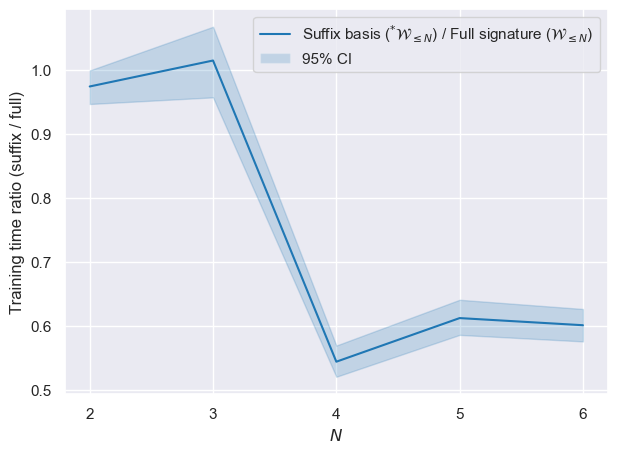}
    }
\label{figure::fitting_time}
\end{figure}

\subsection{Generalization error analysis}
\label{gen error analysis}

We compare the predictive performance of (i) a linear regression using all components of the truncated time-augmented signature and (ii) a linear regression using only the components associated with the set of suffixes. The generalization errors are defined as 

\begin{align*}
    \mathcal{E}_{\text{gen,all}} &:= \mathbb{E}\left[ |Y - \hat{\beta}_{\text{all}}(\hat{\lambda}_{\text{all}})^* \tilde{\mathbf{X}}_{\text{all}}|^2\right], \\
    \mathcal{E}_{\text{gen,suffix}} &:= \mathbb{E}\left[ |Y - \hat{\beta}_{\text{suffix}}(\hat{\lambda}_{\text{suffix}})^* \tilde{\mathbf{X}}_{\text{suffix}}|^2\right],
\end{align*}

where $\hat{\beta}_{\text{all}}(\hat{\lambda}_{\text{all}})$ is independent of $(\tilde{\mathbf{X}}_{\text{all}},Y)$ distributed like $( \widehat{\mathbb{X}}_{T}^{\le N}, F((X_t)_{t \in [0,T]}))$, and $\hat{\beta}_{\text{suffix}}(\hat{\lambda}_{\text{suffix}})$ is independent of $(\tilde{\mathbf{X}}_{\text{suffix}},Y)$ distributed like $( (\widehat{\mathbb{X}}_{T}^{\word{w}})_{\word{w} \in {}^{*}\mathcal{W}_{\le N}}, F((X_t)_{t \in [0,T]}))$. We are interested in the sign of
$\Delta \mathcal{E}_{\text{gen}} := \mathcal{E}_{\text{gen,all}} - \mathcal{E}_{\text{gen,suffix}}$.
The model using the features associated with the suffixes (resp. the model using all the features) performs better if $\Delta \mathcal{E}_{\text{gen}} \ge 0$ (resp. $\Delta \mathcal{E}_{\text{gen}} \le 0$). The quantity $\Delta \mathcal{E}_{\text{gen}}$ is estimated with a Monte Carlo method. For each truncation order $N \in \{2,\cdots,6\}$, we sample $B_N$ independent training batches of size $n=500$, and for each training batch we compute $\hat{\beta}_{\text{all}}(\hat{\lambda}_{\text{all}})$ and $\hat{\beta}_{\text{suffix}}(\hat{\lambda}_{\text{suffix}})$. Moreover, the independent test set consisting of $n=10^4$ i.i.d. paths used to estimate $\mathbb{E}\left[ |Y - \hat{\beta}_{\text{all}}(\hat{\lambda}_{\text{all}})^* \tilde{\mathbf{X}}_{\text{all}}|^2 ~|~ \hat{\beta}_{\text{all}}(\hat{\lambda}_{\text{all}})\right]$ and $\mathbb{E}\left[ |Y - \hat{\beta}_{\text{suffix}}(\hat{\lambda}_{\text{suffix}})^* \tilde{\mathbf{X}}_{\text{suffix}}|^2 ~|~ \hat{\beta}_{\text{suffix}}(\hat{\lambda}_{\text{suffix}})\right]$ is the same for all training batches. The procedure is summarized in Algorithm \ref{algo}. The number $B_N$ of training batches is empirically chosen to obtain a confidence interval width of the estimator of $\Delta \mathcal{E}_{\text{gen}}$ that is of the same order of magnitude for all $N$. In particular, we provide in Table \ref{tab:training_batches_and_results} the standard deviation of the difference of MSEs between both strategies for each truncation order. We provide in Figures \ref{figure::error_brownian}, \ref{figure::error_ou} the estimated quantity $\Delta \mathcal{E}_{\text{gen}}$ with its confidence interval of level $95\%$ for the Brownian motion and the Ornstein-Uhlenbeck cases. Table \ref{tab:training_batches_and_results} summarizes the comparative results.

\begin{algorithm}
\caption{Estimating $\Delta \mathcal{E}_{\text{gen}}$.}
\label{algo}
\textbf{Input.} 

\begin{itemize}[noitemsep, topsep=0pt]
    \item $N$: truncation order of the time-augmented signature.
    \item $n_{\text{test}}$: sample size of the unique test set.
    \item $n_{\text{train}}$: sample size of each training batch.
    \item $B_N$: number of i.i.d. training batches.
\end{itemize}

\textbf{Output.} $\widehat{\Delta \mathcal{E}_{\text{gen}}}$.
\begin{algorithmic}[1]
\State $\widehat{\Delta \mathcal{E}_{\text{gen}}} \leftarrow 0$.
\State Sampling of the test set $(X^{\text{test}}_{k},Y_k^{\text{test}})_{1 \le k \le n_{\text{test}}}$ consisting of i.i.d. samples distributed like $((X_t)_{t \in [0,T]},Y)$.
\State Computation of the $N$-step truncated signature terminal value of the time-augmented sample paths from the test set. For each $k \in \{1,\cdots,n_{\text{test}}\}$, we denote by $\mathbf{X}_{\text{all},k}^{\text{test}}$ (resp. $\mathbf{X}_{\text{suffix},k}^{\text{test}}$) the vector which contains all the components (resp. the components associated with suffixes) of the obtained signature.   
\For{ $i \in \{1,\cdots,B_N\}$}
\State Sampling of the training set $(X^{\text{train,i}}_{k},Y_k^{\text{train},i})_{1 \le k \le n_{\text{train}}}$ consisting of i.i.d. samples distributed like $((X_t)_{t \in [0,T]},Y)$.
\State Computation of the $N$-step truncated signature terminal value of the time-augmented sample paths from the training set. For each $k \in \{1,\cdots,n_{\text{train}}\}$, we denote by $\mathbf{X}_{\text{all},k}^{\text{train},i}$ (resp. $\mathbf{X}_{\text{suffix},k}^{\text{train,i}}$) the vector which contains all the components (resp. the components associated with suffixes) of the obtained signatures.  
\State Computation of $\hat{\beta}_{\text{all}}(\hat{\lambda}_{\text{all}})$ and $\hat{\beta}_{\text{suffix}}(\hat{\lambda}_{\text{suffix}})$ using $(\mathbf{X}_{\text{all},k}^{\text{train},i})_{1 \le k \le n_{\text{train}}}$ and $(\mathbf{X}_{\text{suffix},k}^{\text{train},i})_{1 \le k \le n_{\text{train}}}$ respectively.
\State $\widehat{\Delta \mathcal{E}_{\text{gen}}} \leftarrow \widehat{\Delta \mathcal{E}_{\text{gen}}} + \frac{1}{n_{\text{test}}}\sum_{k=1}^{n_{\text{test}}} \left[|Y_{k}^{\text{test}} - \hat{\beta}_{\text{all}}(\hat{\lambda}_{\text{all}})^*\mathbf{X}_{\text{all},k}^{\text{test}}|^2 - |Y_{k}^{\text{test}} - \hat{\beta}_{\text{suffix}}(\hat{\lambda}_{\text{suffix}})^*\mathbf{X}_{\text{suffix},k}^{\text{test}}|^2\right]$.
\EndFor
\\
\Return{$\widehat{\Delta \mathcal{E}_{\text{gen}}}$.}
\end{algorithmic}
\end{algorithm}

\begin{figure}
    \centering
    \subfloat[$\beta_{\text{true}}=(1,1,\cdots,1)^*$]{%
        \includegraphics[width=0.3\linewidth]{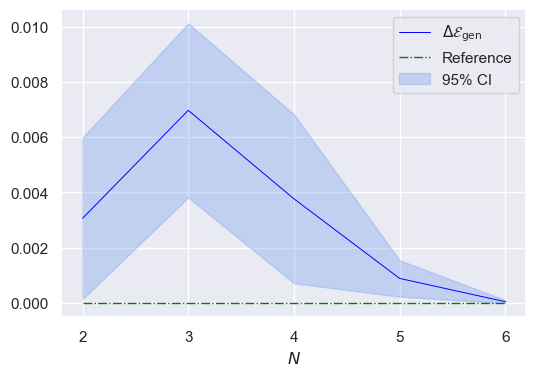}
    }\hfill
    \subfloat[$\beta_{\text{true}}=(1,2,\cdots,2^{11}-1)^*$]{%
        \includegraphics[width=0.3\linewidth]{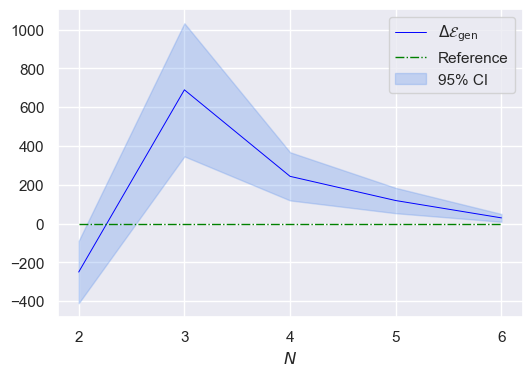}
    }
    \hfill
    \subfloat[$\beta_{\text{true}}=(2^{11}-1,\cdots,2,1)^*$]{%
        \includegraphics[width=0.3\linewidth]{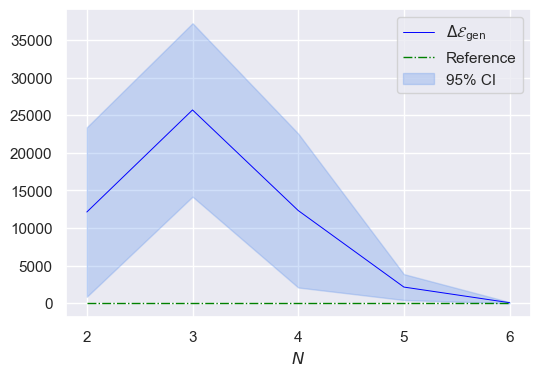}
    }
    \caption{Difference of generalization errors: Brownian motion case.}
    \label{figure::error_brownian}
\end{figure}

\begin{figure}
    \centering
    \subfloat[$\beta_{\text{true}}=(1,1,\cdots,1)^*$]{%
        \includegraphics[width=0.3\linewidth]{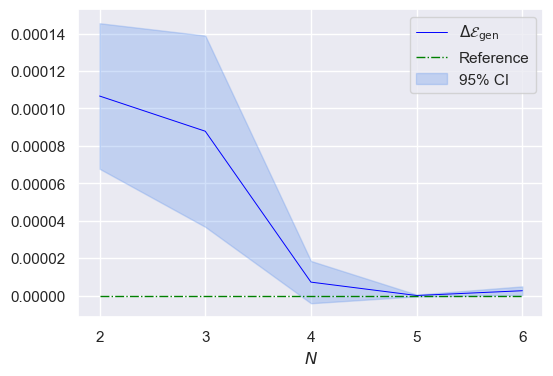}
    }\hfill
    \subfloat[$\beta_{\text{true}}=(1,2,\cdots,2^{11}-1)^*$]{%
        \includegraphics[width=0.3\linewidth]{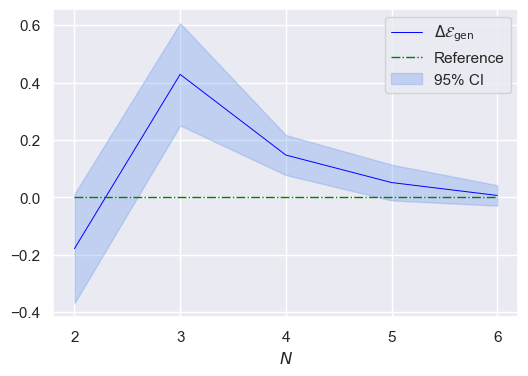}
    }
    \hfill
    \subfloat[$\beta_{\text{true}}=(2^{11}-1,\cdots,2,1)^*$]{%
        \includegraphics[width=0.3\linewidth]{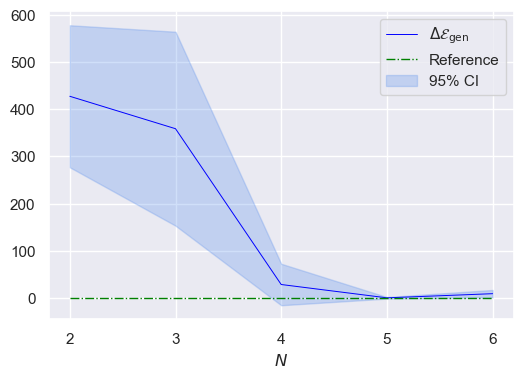}
    }
    \caption{Difference of generalization errors: Ornstein-Uhlenbeck case.}
    \label{figure::error_ou}
\end{figure}

\begin{table}
\footnotesize 
    \centering
    \caption{\textbf{Comparative results.} For each truncation order, we provide (i) the average cross-validated $\lambda$ across the training batches for both strategies, (ii) the standard error of the difference between the MSEs of both strategies, and (iii) the average $R^2$ metric across the training batches for both strategies. For each truncation order, the best $R^2$ is highlighted in bold.}
    \label{tab:training_batches_and_results}

    \begin{tabular}{c c} 
    \begin{minipage}{0.45\textwidth}
        \centering
        \textbf{Brownian case.} \\[0.5em]

        \begin{subtable}{\textwidth}
            \centering
            \caption{$\beta_{\text{true}} = (1,1,\cdots,1)^*$}
            \begin{tabular}{c||c|c|c|c|c}
            \toprule
            $N$ & 2 & 3 & 4 & 5 & 6\\
            \midrule
            \midrule
            $B_N$ & 200 & 120 & 20 & 10 & 10 \\
            \midrule
            $\hat{\lambda}_{\text{suffix}}$ & 0.02 & 0.17 & 0.01 & 0.03 & 0.01 \\
            $\hat{\lambda}_{\text{all}}$ & 1.28 & 4.35 & 1.61 & 0.99 & 0.12 \\
            \midrule
            $\sigma_{\Delta \text{MSE}}$ & 2.52 & 2.02 & 8.23 & 1.22 & 9.47  \\
            & $\times 10^{-2}$ & $\times 10^{-2}$ & $\times 10^{-3}$ & $\times 10^{-3}$ & $\times 10^{-5}$ \\
            $R_{\text{suffix}}^2$ & \textbf{89.42} & \textbf{97.89} & \textbf{99.66} & \textbf{99.96} & \textbf{99.99}\\
            $R_{\text{all}}^2$ & 89.41 & 97.86 & 99.65& 99.96 & 99.99\\
            \bottomrule
            \end{tabular}
        \end{subtable}

        \vspace{0.7em}

        \begin{subtable}{\textwidth}
            \centering
            \caption{$\beta_{\text{true}} = (1,2,\cdots,2^{11}-1)^*$}
            \begin{tabular}{c||c|c|c|c|c}
            \toprule
            $N$ & 2 & 3 & 4 & 5 & 6\\
            \midrule
            \midrule
            $B_N$ & 50 & 120 & 50 & 50 & 10 \\
            \midrule
            $\hat{\lambda}_{\text{suffix}}$ & 28.12 & 2.63 & 0.86 & 0.26 & 0.02 \\
            $\hat{\lambda}_{\text{all}}$ & 40.47 & 9.65 & 4.58 & 1.25 & 0.15 \\
            \midrule
            $\sigma_{\Delta \text{MSE}}$ & 6.91 & 2.29 & 5.33 & 2.81 & 3.85  \\
            & $\times 10^2$ & $\times 10^3$ & $\times 10^2$ & $\times 10^2$ & $\times 10^1$ \\
            $R_{\text{suffix}}^2$ & 49.28 & \textbf{65.59} & \textbf{90.43} & \textbf{96.96} & \textbf{99.24}\\
            $R_{\text{all}}^2$ & \textbf{49.57} & 64.79& 90.15 & 96.82 & 99.20\\
            \bottomrule
            \end{tabular}
        \end{subtable}

        \vspace{0.7em}

        \begin{subtable}{\textwidth}
            \centering
            \caption{$\beta_{\text{true}} = (2^{11}-1,\cdots,2,1)^*$}
            \begin{tabular}{c||c|c|c|c|c}
            \toprule
            $N$ & 2 & 3 & 4 & 5 & 6\\
            \midrule
            \midrule
            $B_N$ & 200 & 120 & 20 & 10 & 10 \\
            \midrule
            $\hat{\lambda}_{\text{suffix}}$ & 0.01 & 0.1 & 0.03 & 0.14 & 0.004 \\
            $\hat{\lambda}_{\text{all}}$ & 1.22 & 4.26 & 1.57 & 0.93 & 0.10\\
            \midrule
            $\sigma_{\Delta \text{MSE}}$ & 9.34 & 7.45 & 2.75 & 3.21 & 1.22  \\
            & $\times 10^{4}$ & $\times 10^{4}$ & $\times 10^{4}$ & $\times 10^{3}$ & $\times 10^{2}$ \\
            $R_{\text{suffix}}^2$ & \textbf{90.18} & \textbf{98.19} & \textbf{99.74} & \textbf{99.98} & \textbf{99.99}\\
            $R_{\text{all}}^2$ & 90.16 & 98.17& 99.73 & 99.97 & 99.99 \\
            \bottomrule
            \end{tabular}
        \end{subtable}

    \end{minipage}
    &
    \begin{minipage}{0.45\textwidth}
        \centering
        \textbf{Ornstein-Uhlenbeck case.} \\[0.5em]

        \begin{subtable}{\textwidth}
            \centering
            \caption{$\beta_{\text{true}} = (1,1,\cdots,1)^*$}
            \begin{tabular}{c||c|c|c|c|c}
            \toprule
            $N$ & 2 & 3 & 4 & 5 & 6\\
            \midrule
            \midrule
            $B_N$ & 800 & 120 & 20 & 10 & 10 \\
            \midrule
            $\hat{\lambda}_{\text{suffix}}$ & 0.02 & 0.04 & 0.003 & 0.00 & 0.00 \\
            $\hat{\lambda}_{\text{all}}$ & 0.26 & 0.43 & 0.14 & 0.01 & 0.01 \\
            \midrule
            $\sigma_{\Delta \text{MSE}}$ & 6.67 & 3.23 & 3.07 & 8.02 & 4.41  \\
            & $\times 10^{-4}$ & $\times 10^{-4}$ & $\times 10^{-5}$ & $\times 10^{-7}$ & $\times 10^{-6}$ \\
            $R_{\text{suffix}}^2$ & \textbf{97.46} & \textbf{99.74} & \textbf{99.98} & \textbf{99.99} & \textbf{99.99}\\
            $R_{\text{all}}^2$ & 97.46 & 99.73 & 99.98 & 99.99 & 99.99\\
            \bottomrule
            \end{tabular}
        \end{subtable}

        \vspace{0.7em}

        \begin{subtable}{\textwidth}
            \centering
            \caption{$\beta_{\text{true}} = (1,2,\cdots,2^{11}-1)^*$}
            \begin{tabular}{c||c|c|c|c|c}
            \toprule
            $N$ & 2 & 3 & 4 & 5 & 6\\
            \midrule
            \midrule
            $B_N$ & 400 & 120 & 80 & 10 & 10 \\
            \midrule
            $\hat{\lambda}_{\text{suffix}}$ & 1.03 & 0.24 & 0.004 & 0.004 & 0.001 \\
            $\hat{\lambda}_{\text{all}}$ & 2.62 & 1.37 & 0.19 & 0.08 & 0.01 \\
            \midrule
            $\sigma_{\Delta \text{MSE}}$ & 7.38 & 1.05 & 3.74 & 1.18 & 4.09  \\
            & $\times 10^{-1}$ &  & $\times 10^{-1}$ & $\times 10^{-1}$ & $\times 10^{-2}$ \\
            $R_{\text{suffix}}^2$ & 67.94 & \textbf{86.24} & \textbf{96.81} & \textbf{99.48} & \textbf{99.87} \\
            $R_{\text{all}}^2$ & \textbf{68.05} & 85.97 & 96.72 & 99.45 & 99.86 \\
            \bottomrule
            \end{tabular}
        \end{subtable}

        \vspace{0.7em}

        \begin{subtable}{\textwidth}
            \centering
            \caption{$\beta_{\text{true}} = (2^{11}-1,\cdots,2,1)^*$}
            \begin{tabular}{c||c|c|c|c|c}
            \toprule
            $N$ & 2 & 3 & 4 & 5 & 6\\
            \midrule
            \midrule
            $B_N$ & 800 & 120 & 20 & 10 & 10 \\
            \midrule
            $\hat{\lambda}_{\text{suffix}}$ & 0.02 & 0.04 & 0.003 & 0.00 & 0.00 \\
            $\hat{\lambda}_{\text{all}}$ & 0.26 & 0.43 & 0.14 & 0.01 & 0.01\\
            \midrule
            $\sigma_{\Delta \text{MSE}}$ & 2.58 & 1.30 & 1.19 & 3.16 & 1.52  \\
            & $\times 10^{3}$ & $\times 10^{3}$ & $\times 10^{2}$ & $\times 10^1$ & $\times 10^{1}$ \\
            $R_{\text{suffix}}^2$ & \textbf{97.52} & \textbf{99.75} & \textbf{99.98} & \textbf{99.99} & \textbf{99.99} \\
            $R_{\text{all}}^2$ & 97.52 & 99.75 & 99.98 & 99.99 & 99.99 \\
            \bottomrule
            \end{tabular}
        \end{subtable}

    \end{minipage}
    \end{tabular}

\end{table}

Across all experiments, using a basis of signature components associated with suffixes yields predictive performance comparable to that obtained with the full feature set, while often leading to marginal but consistent improvements. We also observe that the cross-validated regularization parameter $\lambda$ is typically much smaller when using the basis of signature features, and is null in some cases. This is consistent with reduced feature redundancy and improved numerical conditioning. The conclusions are consistent across different choices of the true regression coefficient and across different choices of stochastic processes. These results support the use of the basis of signature components associated with suffixes as a dimensionality reduction strategy in signature regression, particularly in regimes where computational cost and numerical stability are of primary concern.

\section{Proofs}
\label{part3}

\subsection{Proof of Theorem \ref{main result -1}}

Let us define the linear map $\varphi : \text{Span}(\mathcal{W}_{\le N}) \ni  \sum_{\word{w} \in \mathcal{W}_{\le N}} c_{\word{w}} \word{w} \longmapsto \sum_{\word{w} \in \mathcal{W}_{\le N}} c_{\word{w}} \word{w}\shuffle \word{0}_{N - \|\word{w}\|} \in \text{Span}(\mathcal{W}_{N})$. Let $B \subset \mathcal{W}_{\le N}$ be a subset of words. Since we have $\mathcal{W}_{\le N}= \bigsqcup_{\word{\gamma} \in \mathcal{PW}_{\le N}} \mathcal{W}_{\le N}(\word{\gamma})$ and $\mathcal{W}_{N}= \bigsqcup_{\word{\gamma} \in \mathcal{PW}_{\le N}} \mathcal{W}_{N}(\word{\gamma})$, we obtain 

\begin{align}
\label{decomposition proof 1 2}
\text{Span}(B) = \bigoplus_{\gamma \in \mathcal{PW}_{\le N}} \text{Span}(B \cap{\mathcal{W}}_{\le N}(\word{\gamma})) \quad \text{and} \quad \text{Span}(\mathcal{W}_{N}) = \bigoplus_{\word{\gamma} \in \mathcal{PW}_{\le N}} \text{Span}(\mathcal{W}_{N}(\word{\gamma})).
\end{align}
Moreover it is easy to see that for any pure word $\word{\gamma} \in \mathcal{PW}_{\le N}$, we have

\begin{align}
\label{projection}
    \varphi(\text{Span}(B \cap \mathcal{W}_{\le N}(\word{\gamma}))) \subset\text{Span}(\mathcal{W}_{N}(\word{\gamma})).
\end{align}
Using \eqref{decomposition proof 1 2} and \eqref{projection}, we obtain that $\varphi(\text{Span}(B))=\text{Span}(\mathcal{W}_N)$ if and only if for each pure word $\word{\gamma} \in \mathcal{PW}_{\le N}$, $\varphi(\text{Span}(B \cap \mathcal{W}_{\le N}(\word{\gamma}))) = \text{Span}(\mathcal{W}_{N}(\word{\gamma}))$.
Since $B$ is a basis of words for $\mathcal{W}_N$ (resp. $B \cap \mathcal{W}_{\le N}(\word{\gamma})$ is a basis of words for $\mathcal{W}_{N}(\word{\gamma})$) if and only if $\varphi(\text{Span}(B))=\text{Span}(\mathcal{W}_N)$ (resp. $\varphi(\text{Span}(B \cap \mathcal{W}_{\le N}(\word{\gamma})))=\text{Span}(\mathcal{W}_N(\word{\gamma}))$), we conclude.

\subsection{Proof of Theorem \ref{main result 0}}

The proof of Theorem \ref{main result 0} relies on the following lemma.
\begin{lemma}
\label{lemme liberté linéaire N-1}
Let $N \in \mathbb{N}$ be a truncation order and $B \subset \mathcal{W}_{\le N}$ be a subset of words. If $(\word{w} \shuffle \word{0}_{N - \|\word{w}\|})_{\word{w} \in B}$ is a linearly independent family of elements of $\text{Span}(\mathcal{W}_{N})$, then for all $m \in \{0,\cdots,N-1\}$, $(\word{w} \shuffle \word{0}_{m- \|\word{w}\|})_{\word{w} \in B \cap \mathcal{W}_{\le m}}$ is a linearly independent family of elements of $\text{Span}(\mathcal{W}_{m})$.
\end{lemma}

This lemma is an easy consequence of the formula $(\word{w} \shuffle \word{0}_{m - \|\word{w}\|}) \shuffle \word{0}_{N-m} = \binom{N - \|\word{w}\|}{N - m} \word{w} \shuffle \word{0}_{N-\| \word{w}\|}$ which follows from the associativity of the shuffle product. We can now provide the proof of Theorem \ref{main result 0}.

\begin{proof}[Proof of Theorem \ref{main result 0}]

Let $N \in \mathbb{N}$ be a truncation order, $\word{\gamma} \in \mathcal{PW}_{\le N}$ be a pure word, and $B_{\word{\gamma}} \subset \mathcal{W}_{\le N}(\word{\gamma})$ be a basis of words for $\mathcal{W}_{N}(\word{\gamma})$.
\newline
\newline
\textbf{First step: $\text{Card}(B_{\word{\gamma}}) = \binom{N}{\|\word{\gamma}\|}$.}
\newline
Since $B_{\word{\gamma}}$ is a basis of words for $\mathcal{W}_{N}(\word{\gamma})$, it follows that $(\word{w} \shuffle \word{0}_{N - \| \word{w}\|})_{\word{w} \in B_{\word{\gamma}}}$ is a basis of $\text{Span}(\mathcal{W}_{N}(\word{\gamma}))$. Hence we obtain the cardinality equality: $\text{Card}(B_{\word{\gamma}}) = \dim(\text{Span}(\mathcal{W}_{N}(\word{\gamma}))) = \text{Card}(\mathcal{W}_{N}(\word{\gamma})) = \binom{N}{\| \word{\gamma}\|}$.
\newline
\newline
\textbf{Second step: $\forall m \in \{ \|\word{\gamma}\|,\cdots,N-1\}, \text{Card}(B_{\word{\gamma}} \cap \mathcal{W}_{\le m}) \le \binom{m}{\| \word{\gamma}\|}$.}
\newline
We assume that $\|\word{\gamma}\| \le N-1$. If $B_{\word{\gamma}}$ is a basis of words for $\mathcal{W}_N(\word{\gamma})$, then $(\word{w} \shuffle \word{0}_{N - \|\word{w}\|})_{\word{w} \in B_{\word{\gamma}}}$ is a linearly independent family of elements of $\text{Span}(\mathcal{W}_{N}(\word{\gamma}))$. By  Lemma \ref{lemme liberté linéaire N-1}, we get that for all $m \in \{\|\word{\gamma}\|,\cdots,N-1\}$, $(\word{w} \shuffle \word{0}_{m - \|\word{w}\|})_{\word{w} \in B_{{\word{\gamma}}} \cap \mathcal{W}_{\le m}}$ is a linearly independent family of elements of $\text{Span}(\mathcal{W}_{m}(\word{\gamma}))$. Then for all $m \in \{ \|\word{\gamma}\|,\cdots,N-1\}$ we have $\text{Card}(B_{\word{\gamma}} \cap \mathcal{W}_{\le m}) \le \dim(\text{Span}(\mathcal{W}_{m}(\word{\gamma}))) = \binom{m}{\|\word{\gamma}\|}$.

\end{proof}

\subsection{Proof of Theorem \ref{main result 1}}

The proof of Theorem \ref{main result 1} relies on the following lemma, the proof of which is postponed after the end of its proof.

\begin{lemma}
\label{lemme theorem 1}
Let $\tilde{N} \in \mathbb{N}$. Then for any $k_1,k_2 \in \mathbb{N}$ such that $k_1 + k_2 = \tilde{N}$ and any $M \in \mathbb{N}$, we have

\begin{enumerate}[label=(\roman*)]
    \item $\forall \word{w} \in \mathcal{W}_{M}^{*}, \quad \word{w}\word{0}_{k_1} \shuffle \word{0}_{k_2} = \binom{\tilde{N}}{k_1} \word{w} \word{0}_{\tilde{N}} + \sum_{k=M+1}^{M+k_2} \sum_{\word{v} \in \mathcal{W}_{k}^{*}} c_{\word{v}} \word{v}\word{0}_{\tilde{N}-k+M}$, with $c_{\word{v}} \in \mathbb{N}$ for each word $\word{v} \in \bigcup_{k=M+1}^{M+k_2} \mathcal{W}_{k}^{*}$.
    \item $\forall \word{w} \in {}^{*}\mathcal{W}_{M}, \quad \word{0}_{k_1}\word{w} \shuffle \word{0}_{k_2} = \binom{\tilde{N}}{k_1} \word{0}_{\tilde{N}}\word{w} + \sum_{k=M+1}^{M+k_2} \sum_{\word{v} \in {}^{*}\mathcal{W}_{k}} c_{\word{v}} \word{0}_{\tilde{N}-k+M}\word{v}$, with $c_{\word{v}} \in \mathbb{N}$ for each word $\word{v} \in \bigcup_{k=M+1}^{M+k_2} {}^{*}\mathcal{W}_{k}$.
\end{enumerate}

\end{lemma}

We can now provide the proof of Theorem \ref{main result 1}.

\begin{proof}[Proof of Theorem \ref{main result 1}]

Let $N\in \mathbb{N}$ be a truncation order and $\word{\gamma} \in \mathcal{PW}_{\le N}$.  
\newline
\newline
\textbf{First step: Proof of (i).}
\newline
Let $(m_{\word{w}})_{\word{w} \in \mathcal{W}_{\le N}^{*}(\word{\gamma})}$ be a family of non-negative integers such that for each prefix $\word{w} \in \mathcal{W}_{\le N}^{*}(\word{\gamma})$, $\|\word{w}\| + m_{\word{w}} \le N$. Then $N - \|\word{w}\| - m_{\word{w}} \in \mathbb{N}$. Using (i) of Lemma \ref{lemme theorem 1} applied with $M = \|\word{w}\|$, $\tilde{N}=N - \|\word{w}\|$, $k_1 = m_{\word{w}}$ and $k_2 = N - m_{\word{w}} - \|\word{w}\|$, we have

\begin{align*}
\word{w}\word{0}_{m_{\word{w}}} \shuffle \word{0}_{N - \|\word{w}\| - m_{\word{w}}} &= \binom{N - \| \word{w} \|}{m_{\word{w}}} \word{w}\word{0}_{N - \|\word{w}\|} + \sum_{k=\|\word{w}\|+1}^{N - m_{\word{w}}} \sum_{\word{v} \in \mathcal{W}_{k}^{*}} c_{\word{v}} \word{v}\word{0}_{N -k}.
\end{align*}
The shuffle $\word{w}\word{0}_{m_{\word{w}}} \shuffle \word{0}_{N - \|\word{w}\| - m_{\word{w}}}$ lies in $\text{Span}(\mathcal{W}_{N}(\word{\gamma}))$, then $c_{\word{v}}=0$ whenever $\word{v} \in \mathcal{W}_{\le N - m_{\word{w}}}^{*} \setminus \mathcal{W}_{\le N - m_{\word{w}}}^{*}(\word{\gamma})$ . We can rewrite 

\begin{align*}
\word{w}\word{0}_{m_{\word{w}}} \shuffle \word{0}_{N - \|\word{w}\| - m_{\word{w}}} &= \binom{N - \| \word{w} \|}{m_{\word{w}}} \word{w}\word{0}_{N - \|\word{w}\|} + \sum_{k=\|\word{w}\|+1}^{N - m_{\word{w}}} \sum_{\word{v} \in \mathcal{W}_{k}^{*}(\word{\gamma})} c_{\word{v}} \word{v}\word{0}_{N -k}.
\end{align*}
Then if we sort the words $\word{v} \in \mathcal{W}_{\le N}^{*}(\word{\gamma})$ according to the lexicographic order on $(\|\word{v}\|,\word{v})$, there exists an upper triangular matrix $(P_{\word{w},\word{v}})_{\word{w},\word{v} \in \mathcal{W}_{\le N}^{*}(\word{\gamma})}$  with diagonal coefficients equal to $\binom{N - \| \word{w} \|}{m_{\word{w}}} > 0$ such that $(\word{w}\word{0}_{m_{\word{w}}} \shuffle \word{0}_{N - \|\word{w}\| - m_{\word{w}}})_{\word{w} \in \mathcal{W}_{\le N}^{*}(\word{\gamma})} = P \cdot (\word{v}\word{0}_{N - \| \word{v}\|})_{\word{v} \in \mathcal{W}_{\le N}^{*}(\word{\gamma})}$. Since $\{ \word{v} \word{0}_{N - \| \word{v} \|} ~|~ \word{v} \in \mathcal{W}_{\le N}^{*}(\word{\gamma})\} = \mathcal{W}_{N}(\word{\gamma})$, we obtain that $(\word{w}\word{0}_{m_{\word{w}}} \shuffle \word{0}_{N - \|\word{w}\| - m_{\word{w}}})_{\word{w} \in \mathcal{W}_{\le N}^{*}(\word{\gamma})}$ is a basis of $\text{Span}(\mathcal{W}_{N}(\word{\gamma}))$. Therefore, by definition, $\left\{ \word{w}\word{0}_{m_{\word{w}}} ~|~ \word{w} \in \mathcal{W}^{*}_{\le N}(\word{\gamma}) \right\}$ is a basis of words for $\mathcal{W}_{N}(\word{\gamma})$.
\newline
\newline
\textbf{Second step: Proof of (ii).}
\newline
We conclude by a similar argument replacing Lemma \ref{lemme theorem 1} (i) by Lemma \ref{lemme theorem 1} (ii).
\end{proof}

We now provide the proof of Lemma \ref{lemme theorem 1}.

\begin{proof}
\textbf{First step: Proof of (i).}
\newline
We prove this result by induction on $\tilde{N} \in \mathbb{N}$. 
\newline
\\
\underline{Initialization step.} Let $\tilde{N}=0$, $M \in \mathbb{N}$ and $\word{w} \in \mathcal{W}_{M}^{*}$. The result follows trivially since $k_1=k_2=0$ which implies that $\word{0}_{k_1} = \word{0}_{k_2} = \emptyword$.
\newline
\\
\underline{Induction step.} We assume that the property is true for $\tilde{N} \in \mathbb{N}$. Let $M \in \mathbb{N}$, $\word{w} \in \mathcal{W}_{M}^{*}$ and $k_1,k_2 \in \mathbb{N}$ be such that $k_1+k_2 =\tilde{N}+1$.
\newline
\newline
\textbf{First case: $k_2 = 0$.}
\newline
We then have $k_1 = \tilde{N}+1$ and $\word{0}_{k_2}= \emptyword$, then $\word{w} \word{0}_{k_1} \shuffle \word{0}_{k_2} = \word{w} \word{0}_{\tilde{N}+1} \shuffle \emptyword =  \word{w} \word{0}_{\tilde{N}+1}$.
\newline
\newline
\textbf{Second case: $k_1 = 0$.}
\newline
Then $k_2 = \tilde{N}+1$. If $M=0$, then $\word{w} = \emptyword$ and $\emptyword \shuffle \word{0}_{\tilde{N}+1} = \word{0}_{\tilde{N}+1}$. If $M \ge 1$, then there exist $\word{\tilde{w}} \in \mathcal{W}_{M-1}$ and a letter $\word{i} \in \{\word{1},\cdots,\word{d}\}$ such that $\word{w}=\word{\tilde{w} i}$. Using the recursive property of the shuffle product (see Definition \ref{def shuffle product}) we can write 

\begin{align*}
    \word{w} \shuffle \word{0}_{\tilde{N}+1} &= ( \word{w} \shuffle \word{0}_{\tilde{N}})\word{0} +  (\word{\tilde{w}} \shuffle \word{0}_{\tilde{N}+1})\word{i}.
\end{align*}
Using the induction hypothesis we can write

\begin{align*}
    \word{w} \shuffle \word{0}_{\tilde{N}+1} &= \left( \word{w} \word{0}_{\tilde{N}} + \sum_{k=M+1}^{M+\tilde{N}} \sum_{\word{v} \in \mathcal{W}^{*}_{k}} c_{1,\word{v}} \word{v} \word{0}_{\tilde{N}-k+M}\right)\word{0} +  \left(\word{\tilde{w}} \word{0}_{\tilde{N}+1} + \sum_{k=M}^{M+\tilde{N}} \sum_{\word{v} \in \mathcal{W}_{k}^{*}} c_{2,\word{v}} \word{v} \word{0}_{\tilde{N}-k+M} \right)\word{i},
\end{align*}
with $c_{1,\word{v}} \in \mathbb{N}$ for each word $\word{v}  \in \bigcup_{k=M+1}^{M+\tilde{N}}\mathcal{W}_{k}^{*}$ and $c_{2,\word{v}} \in \mathbb{N}$ for each word $\word{v} \in \bigcup_{k=M}^{M+\tilde{N}} \mathcal{W}_{k}^{*}$. Hence we have

\begin{align*}
   \word{w} \shuffle \word{0}_{\tilde{N}+1} &=  \word{w} \word{0}_{\tilde{N}+1} + \sum_{k=M+1}^{M+\tilde{N}} \sum_{\word{v} \in \mathcal{W}^{*}_{k}} c_{1,\word{v}} \word{v} \word{0}_{\tilde{N}+1-k+M} +  \word{\tilde{w}} \word{0}_{\tilde{N}+1} \word{i} + \sum_{k=M}^{M+\tilde{N}} \sum_{\word{v} \in \mathcal{W}_{k}^{*}} c_{2,\word{v}} \word{v} \word{0}_{\tilde{N}-k+M} \word{i}.
\end{align*}
Moreover we know that $\word{\tilde{w}} \word{0}_{\tilde{N}+1} \word{i} \in \mathcal{W}^{*}_{M + \tilde{N}+1}$ and $\word{v} \word{0}_{\tilde{N}-k+M} \word{i} \in \mathcal{W}^{*}_{M + \tilde{N}+1}$ for each $\word{v} \in \mathcal{W}^{*}_{k}$ for $k \in \{M,\cdots,M+\tilde{N}\}$. Then there exists a family of non-negative integers $(c_{3,\word{v}})_{\word{v} \in \mathcal{W}^{*}_{M+ \tilde{N}+1}}$ such that $\word{\tilde{w}} \word{0}_{\tilde{N}+1} \word{i} + \sum_{k=M}^{M+\tilde{N}} \sum_{\word{v} \in \mathcal{W}_{k}^{*}} c_{2,\word{v}} \word{v} \word{0}_{\tilde{N}-k+M} \word{i} = \sum_{\word{v} \in \mathcal{W}^{*}_{M+\tilde{N}+1}} c_{3,\word{v}} \word{v}$. Then it follows that

\begin{align*}
    \word{w} \shuffle \word{0}_{\tilde{N}+1} &=  \word{w} \word{0}_{\tilde{N}+1} + \sum_{k=M+1}^{M+\tilde{N}} \sum_{\word{v} \in \mathcal{W}^{*}_{k}} c_{1,\word{v}} \word{v} \word{0}_{\tilde{N}+1-k+M} +  \sum_{\word{v} \in \mathcal{W}^{*}_{M+\tilde{N}+1}} c_{3,\word{v}} \word{v}.
\end{align*}
If we define $c_{\word{v}}=c_{1,\word{v}}$ if $\word{v} \in \bigcup_{k=M+1}^{M+\tilde{N}} \mathcal{W}^{*}_{k}$ and $c_{\word{v}}=c_{3,\word{v}}$ if $\word{v} \in \mathcal{W}^{*}_{M+\tilde{N}+1}$, we get

\[
\word{w} \shuffle \word{0}_{\tilde{N}+1} = \word{w} \word{0}_{\tilde{N}+1} + \sum_{k=M+1}^{M+\tilde{N}+1} \sum_{\word{v} \in \mathcal{W}^{*}_{k}} c_{\word{v}} \word{v} \word{0}_{\tilde{N}+1-k+M}.
\]
\newline
\textbf{Third case: $k_1 > 0$ and $k_2 > 0$.}
\newline
Using the recursive property of the shuffle product property (see Definition \ref{def shuffle product}) we can write

\begin{align*}
    \word{w} \word{0}_{k_1} \shuffle \word{0}_{k_2} &= (\word{w} \word{0}_{k_1} \shuffle \word{0}_{k_2 - 1})\word{0} + (\word{w} \word{0}_{k_1-1} \shuffle \word{0}_{k_2})\word{0}.
\end{align*}
Using the induction hypothesis, we can write 

\begin{align*}
    \word{w} \word{0}_{k_1} \shuffle \word{0}_{k_2} &= \left(\binom{\tilde{N}}{k_1} \word{w} \word{0}_{\tilde{N}} + \sum_{k=M+1}^{M+k_2 - 1} \sum_{\word{v} \in \mathcal{W}_{k}^{*}} c_{1,\word{v}} \word{v}\word{0}_{\tilde{N}-k+M}\right)\word{0} \\
& \qquad + \left(\binom{\tilde{N}}{k_1 -1} \word{w} \word{0}_{\tilde{N}} + \sum_{k=M+1}^{M+k_2} \sum_{\word{v} \in \mathcal{W}_{k}^{*}} c_{2,\word{v}} \word{v}\word{0}_{\tilde{N}-k+M}\right)\word{0},
\end{align*}
with $c_{1,\word{v}} \in \mathbb{N}$ for each word $\word{v} \in \bigcup_{k=M+1}^{M+k_2 -1} \mathcal{W}_{k}^{*}$, and $c_{2,\word{v}} \in \mathbb{N}$ for each word $\word{v}  \in \bigcup_{k=M+1}^{M+k_2} \mathcal{W}_{k}^{*}$. We extend the first sum to $k=M+k_2$ by taking $c_{1,\word{v}}=0$ for each $\word{v} \in \mathcal{W}^{*}_{M+k_2}$. We then obtain

\begin{align*}
\word{w} \word{0}_{k_1} \shuffle \word{0}_{k_2} &= \left( \binom{\tilde{N}}{k_1} + \binom{\tilde{N}}{k_1 -1} \right) \word{w}\word{0}_{\tilde{N}}\word{0} + \sum_{k=M+1}^{M+k_2} \sum_{\word{v} \in \mathcal{W}_{k}^{*}}(c_{1,\word{v}}+c_{2,\word{v}}) \word{v}\word{0}_{\tilde{N}-k+M}\word{0} \\
&= \binom{\tilde{N}+1}{k_1} \word{w}\word{0}_{\tilde{N}+1} + \sum_{k=M+1}^{M+k_2} \sum_{\word{v} \in \mathcal{W}_{k}^{*}}(c_{1,\word{v}}+c_{2,\word{v}}) \word{v}\word{0}_{\tilde{N}+1-k+M},
\end{align*}
and $c_{1,\word{v}} + c_{2,\word{v}} \in \mathbb{N}$ for each word $\word{v}  \in \bigcup_{k=M+1}^{M+k_2} \mathcal{W}_{k}^{*}$, which finishes the proof.
\newline
\newline
\textbf{Second step: Proof of (ii).}
\newline
The main tool for the first step is the recursive formula of the shuffle product introduced in Definition \ref{def shuffle product}. The proof of (ii) is the same up to the replacement of this formula by the equivalent recursive formula (see \citep{lothaire1997combinatorics} for more details): for all $\word{w},\word{v} \in \mathcal{W}_{<\infty}$ and $\word{i},\word{j} \in \mathcal{A}$

\begin{align*}
\begin{cases}
\word{w}\shuffle\emptyword=\emptyword \shuffle \word{w} = \word{w} \\
(\word{i w}) \shuffle (\word{j v}) = \word{i}(\word{w} \shuffle \word{j v}) + \word{j}(\word{i w} \shuffle \word{v}).
\end{cases}
\end{align*}

\end{proof}

\subsection{Proof of the statement in Remark \ref{remark contre exemple}}
\label{preuve remark}

The proof of the statement relies on the following lemma.

\begin{lemma}
\label{lemma statement}

For each truncation order $N \in \mathbb{N}^*$ and each word $\word{\tilde{w}} \in \mathcal{W}_{N+1}(\word{i})$, the set $\mathcal{W}_{N}(\word{i}) \cup \{ \word{\tilde{w}}\}$ is a basis of words for $\mathcal{W}_{N+1}(\word{i})$
\end{lemma}

We now provide the proof of the statement before that of the lemma.

\begin{proof}[Proof of the statement in Remark \ref{remark contre exemple}]

We prove this statement by induction on the truncation order $N \in \mathbb{N}^*$. The property is trivially satisfied for $N=1$ since the only subset of $\mathcal{W}_{\le 1}(\word{i})$ containing exactly one word with length $k=1$ is $\{\word{i}\} = \mathcal{W}_{1}(\word{i})$. Let $N \in \mathbb{N}^*$ be such that the property is satisfied, and $B_{i} \subset \mathcal{W}_{\le N+1}(\word{i})$ be a set containing exactly one word with length $k \in \{1,\cdots,N+1\}$. Thanks to the induction hypothesis, the set $B_{\word{i}} \cap \mathcal{W}_{\le N}$ is a basis of words for $\mathcal{W}_{N}(\word{i})$, and it follows that 

\begin{align}
\label{eq statement}
\text{Span}(\word{w} \shuffle \word{0}_{N - \| \word{w}\|} ~|~ \word{w} \in B_{\word{i}} \cap \mathcal{W}_{\le N}) &= \text{Span}(\mathcal{W}_{N}(\word{i})).    
\end{align}
We consider the linear map 

\[
 \mathcal{H} : \text{Span}(\mathcal{W}_{N}(\word{i})) \ni \ell \longmapsto \ell \shuffle \word{0} \in \text{Span}(\mathcal{W}_{N}(\word{i})).
\]
We have $\mathcal{H}(\text{Span}(\word{w} \shuffle \word{0}_{N - \| \word{w}\|} ~|~ \word{w} \in B_{\word{i}} \cap \mathcal{W}_{\le N})) = \text{Span}( (\word{w} \shuffle \word{0}_{N - \| \word{w} \|}) \shuffle \word{0} ~|~ \word{w} \in B_{\word{i}} \cap \mathcal{W}_{\le N}) = \text{Span}( (N+1 - \| \word{w} \|)\word{w} \shuffle \word{0}_{N +1 - \| \word{w} \|} ~|~ \word{w} \in B_{\word{i}} \cap \mathcal{W}_{\le N})$, and since $N+1 - \|\word{w}\| >0$ for each $\word{w} \in B_{\word{i}} \cap \mathcal{W}_{\le N}$, we obtain

\begin{align}
\label{eq statement 2}
\mathcal{H}(\text{Span}(\word{w} \shuffle \word{0}_{N - \| \word{w}\|} ~|~ \word{w} \in B_{\word{i}} \cap \mathcal{W}_{\le N})) &= \text{Span}(\word{w} \shuffle \word{0}_{N +1 - \| \word{w} \|} ~|~ \word{w} \in B_{\word{i}} \cap \mathcal{W}_{\le N}).
\end{align}
Moreover, we have

\begin{align}
\label{eq statement 3}
\mathcal{H}(\text{Span}(\mathcal{W}_{N}(\word{i}))) &= \text{Span}(\word{w} \shuffle \word{0} ~|~ \word{w} \in \mathcal{W}_{N}(\word{i})).
\end{align}
Using \eqref{eq statement}, \eqref{eq statement 2} and \eqref{eq statement 3}, we get $\text{Span}(\word{w} \shuffle \word{0}_{N +1 - \| \word{w} \|} ~|~ \word{w} \in B_{\word{i}} \cap \mathcal{W}_{\le N}) = \text{Span}(\word{w} \shuffle \word{0} ~|~ \word{w} \in \mathcal{W}_{N}(\word{i}))$. We denote by $\word{\tilde{w}}$ the unique element of $B_{\word{i}}$ with length $N+1$. Since $(B_{\word{i}} \cap \mathcal{W}_{\le N}) \cup \{ \word{\tilde{w}} \} = B_{\word{i}}$ and $\word{\tilde{w}} \shuffle \word{0}_{N+1- \|\word{\tilde{w}}\|} = \word{\tilde{w}}$, the previous equality yields

\begin{align}
\label{eq statement 4}
\text{Span}(\word{w} \shuffle \word{0}_{N +1 - \| \word{w} \|} ~|~ \word{w} \in B_{\word{i}}) &= \text{Span}( \{\word{w} \shuffle \word{0} ~|~ \word{w} \in \mathcal{W}_{N}(\word{i})\} \cup \{\word{\tilde{w}}\}).   
\end{align}
Lemma \ref{lemma statement} implies that $\mathcal{W}_{N}(\word{i}) \cup \{\word{\tilde{w}}\}$ is a basis of words for $\mathcal{W}_{N+1}(\word{i})$. It follows that

\[
\text{Span}( \{\word{w} \shuffle \word{0} ~|~ \word{w} \in \mathcal{W}_{N}(\word{i})\} \cup \{\word{\tilde{w}}\}) = \text{Span}(\mathcal{W}_{N+1}(\word{i})).  
\]
Using this equality and \eqref{eq statement 4} we obtain that $\text{Span}(\word{w} \shuffle \word{0}_{N +1 - \| \word{w} \|} ~|~ \word{w} \in B_{\word{i}}) = \text{Span}(\mathcal{W}_{N+1}(\word{i}))$. Since $\text{Card}(B_{\word{i}})=N+1=\text{Card}(\mathcal{W}_{N+1}(\word{i}))$, the family $(\word{w} \shuffle \word{0}_{N+1-\| \word{w} \|})_{\word{w} \in B_{\word{i}}}$ is a basis of $\text{Span}(\mathcal{W}_{N+1}(\word{i}))$. Therefore, by definition, $B_{\word{i}}$ is a basis of words for $\mathcal{W}_{N+1}(\word{i})$.

\end{proof}

We now provide the proof of Lemma \ref{lemma statement}.

\begin{proof}[Proof of Lemma \ref{lemma statement}]

For each truncation order $N \in \mathbb{N}^*$, the $N$ elements of $\mathcal{W}_{N}(\word{i})$ are ordered as follows: for $k \in \{1,\cdots,N\}$, the $k$-th element is the word $\word{w}^{(N)}_{k}=\word{0}_{k-1} \word{i} \word{0}_{N-k}$. Let $N \in \mathbb{N}^{*}$ be a truncation order and $\word{\tilde{w}} \in \mathcal{W}_{N+1}(\word{i})$. We denote by $\word{w}_{N+1}^{(N)}=\word{\tilde{w}}$. There exists an unique $k^{*} \in \{1,\cdots,N+1\}$ such that $\word{\tilde{w}} = \word{w}_{k^*}^{(N+1)}$. We denote by $P=(P_{k,k'})_{\substack{1 \le k \le N \\ 1 \le k' \le N+1}} \in \mathbb{R}^{N \times (N+1)}$ defined as follows:

\begin{align*}
P_{k,k'}=k \mathbbm{1}_{k'=k+1} + (N+1-k) \mathbbm{1}_{k'=k}. 
\end{align*}
Let $k \in \{1,\cdots,N\}$, we can easily check that $\word{w}_{k}^{(N)} \shuffle \word{0} = \sum_{k'=1}^{N+1} P_{k,k'} \word{w}_{k'}^{(N+1)}$, and $\word{w}_{N+1}^{(N)} = \sum_{k'=1}^{N+1} \mathbbm{1}_{ \{k'=k^* \}} \word{w}_{k'}^{(N+1)}$. We denote by $Q=(Q_{k,k'})_{1 \le k,k' \le N+1} \in \mathbb{R}^{(N+1) \times (N+1)}$ the matrix defined as follows

\[
    Q_{k,k'} = \begin{cases}
        P_{k,k'} ~\text{if}~ k \in \{1,\cdots, N\} \\
        \mathbbm{1}_{ \{k'=k^{*}\} } ~\text{if}~ k=N+1
    \end{cases}.
\]
Let us check that $Q$ is invertible which implies that $(\word{w} \shuffle \word{0}_{N - \| \word{w} \|})_{\word{w} \in \mathcal{W}_{N}(\word{i}) \cup \{ \word{\tilde{w}}\}}$ is a basis of $\text{Span}(\mathcal{W}_{N+1}(\word{i}))$. By expanding the determinant on the $(N+1)$-th row, we have $\det(Q)=(-1)^{N+1+k^*}\det(M)$ with $M=(M_{k,k'})_{1 \le k,k' \le N} \in \mathbb{R}^{N \times N}$ defined as follows

\[
    M_{k,k'} = \begin{cases}
        P_{k,k'} ~\text{if}~ k'<k^{*} \\
        P_{k,k'+1} ~\text{if}~ k'\ge k^{*}
    \end{cases}.
\]
It is easy to check that $M$ is block diagonal with triangular blocks. In particular, the submatrices $(M_{k,k'})_{1 \le k,k' < k^*}$ and $(M_{k,k'})_{k^* \le k,k' \le N}$ are respectively upper and lower triangular, with non-zero entries on their respective diagonals. It follows that $\det(Q)>0$. 

\end{proof}

\subsection{Proof of Corollary \ref{corrolary length}}

The proof of Corollary \ref{corrolary length} relies on the following lemma, the proof of which is postponed at the end of its proof.

\begin{lemma}
\label{lower bound length basis of words}

Let $N \in \mathbb{N}$ be a truncation order, and $B \subset \mathcal{W}_{\le N}$ be a basis of words for $\mathcal{W}_N$. Then we have the following lower bound

\begin{align*}
    \length{B} \ge \frac{N d (d+1)^N - (d+1)^{N} +1}{d}.
\end{align*}
    
\end{lemma}

\begin{proof}[Proof of Corollary \ref{corrolary length}]
Let $N \in \mathbb{N}$ be a truncation order. Theorem \ref{main result 1} ensures that $\mathcal{W}_{\le N}^{*}$ and ${}^{*}\mathcal{W}_{\le N}$ are bases of words for $\mathcal{W}_{N}$. We have $\length{\mathcal{W}_{\le N}^{*}}=\sum_{k=0}^{N} k~\text{Card}(\mathcal{W}_{\le N}^{*} \cap \mathcal{W}_{k}) =\sum_{k=1}^{N} k (d+1)^{k-1} d = \frac{N d (d+1)^N - (d+1)^N + 1}{d}$, which is the lower bound given in Lemma \ref{lower bound length basis of words}. By symmetry, $\length{{}^{*}\mathcal{W}_{\le N}}=\length{\mathcal{W}_{\le N}^{*}}$.

\end{proof}

\begin{proof}[Proof of Lemma \ref{lower bound length basis of words}]

Let $N \in \mathbb{N}$ be a truncation order, and $B \subset \mathcal{W}_{\le N}$ be a basis of words for $\mathcal{W}_N$. We have

\begin{align*}
\length{B}  &= \sum_{k=1}^{N} k~ \text{Card}(B \cap \mathcal{W}_k) = \sum_{k=1}^{N} k~(\text{Card}(B \cap \mathcal{W}_{\le k}) - \text{Card}(B \cap \mathcal{W}_{\le k-1})) \\
&= \sum_{k=1}^{N} k~\text{Card}(B \cap \mathcal{W}_{\le k}) - \sum_{k=0}^{N-1} (k+1)~\text{Card}(B \cap \mathcal{W}_{\le k}) = N~\text{Card}(B) - \sum_{k=0}^{N-1} \text{Card}(B \cap \mathcal{W}_{\le k}).
\end{align*}
Corollary \ref{cardinal bases} implies that $\text{Card}(B)=(d+1)^N$ and for all $k \in \{0,\cdots,N-1\}$ we have the inequality $\text{Card}(B \cap \mathcal{W}_{\le k}) \le (d+1)^k$. By plugging these into the previous equality, we obtain the desired result.
    
\end{proof}

\subsection{Proof of Theorem \ref{main result 2}}

Relying on Remark \ref{remark base libre and existence and uniqueness of solution}, it suffices to show that the family $(\widehat{\mathbb{X}}_{T}^{\word{w}})_{\word{w} \in \mathcal{W}_{<\infty}^{*}}$ is linearly independent for the almost-sure equality. The proof of this statement relies on the following proposition, the proof of which is given in Section \ref{section prop brownian}.

\begin{proposition}
\label{prop brownian}
Let $\left(\Omega,\mathcal{F},\mathbb{P}\right)$ be a probability space, $(W_t)_{t \ge 0}$ be a $d$-dimensional Brownian motion under $\mathbb{P}$, and $N \in \mathbb{N}$. Then, for all $T > 0$, the family of random variables $(\widehat{\mathbb{W}}_T^{\word{w}})_{\word{w} \in \mathcal{W}_{\le N}^*}$ is linearly independent for the almost-sure equality. 
\end{proposition}

\begin{proof}[Proof of Theorem \ref{main result 2}]
\label{proof main result 2}
Let $(X_t)_{t \in [0,T]}$ be the solution to \eqref{Ito form}. We denote by $\mu$ the law of $Y$. Since the signature is invariant under translations, it is enough to prove the result conditionally on $Y=y$, for $\mu$-a.e. $y\in\mathbb R^d$. Moreover, since $Y$ is independent of the driving Brownian motion $(\beta_t)_{t \in [0,T]}$, then conditionally on $Y=y$  for $\mu$-a.e. $y\in\mathbb R^d$, the process $(X_t)_{t \in [0,T]}$ is solution to the SDE $dX_t = b(t,X_t)dt + \sigma(t,X_t) d\beta_t,  X_0=y$. For such a fixed $y$, setting $\widetilde X_t=X_t-y$, the process $\widetilde X$ solves an SDE starting from $0$ with coefficients $\widetilde b_y(t,x)=b(t,y+x), \widetilde\sigma_y(t,x)=\sigma(t,y+x)$. These coefficients satisfy the same regularity assumptions and the uniform ellipticity condition is preserved. Hence, without loss of generality, we assume from now on that $X_0=0$. Let $(c_{\word{w}})_{\word{w} \in \mathcal{W}_{<\infty}^*}$ be a family with finite number of non-zero elements such that
\begin{align}
\label{eq:forme_lineaire_nulle}
\sum_{\word{w} \in \mathcal{W}_{<\infty}^*} c_{\word{w}} \widehat{\mathbb{X}}_{T}^{\word{w}} = 0,
\quad \text{a.s.}.
\end{align}
There exists $N \in \mathbb{N}$ such that $c_{\word{w}}=0$ for every $\word{w} \in \mathcal{W}_{<\infty}^*$ satisfying $\|\word{w}\|>N$.

For any finite-dimensional Euclidean space $E$, any $q\ge 1$ and for every $M\ge \lfloor q\rfloor$, we denote by $S_M^q:
\mathcal C^{q\text{-var}}_o([0,T],G^{\lfloor q\rfloor}(E))
\longrightarrow
\mathcal C^{q\text{-var}}_o([0,T],G^M(E))$ the $M$-step truncated signature map. Here $\mathcal C^{q\text{-var}}_o([0,T],G^k(E))$ denotes the space of continuous paths with finite $q$-variation, taking values in the free nilpotent group $G^k(E)$, and starting from the unit element at time $0$. When the ambient space $E$ is clear from the argument of $S_M^q$, we omit it from the notation. In particular, $S_2^1$ denotes the canonical $2$-step lift of a finite-variation path, in whatever Euclidean space the path takes its values. Fix $p\in(2,3)$. We denote by $d_{p\text{-var}}$ the $p$-variation rough path metric defined from the Carnot--Carathéodory metric $d_{\rm CC}$; see \citep[Definitions 8.1 and 7.41]{friz2010multidimensional}. We denote by $\mathcal C^{0,p\text{-var}}_o([0,T],G^k(E))
:=
\overline{S_k^{1}\left(\mathcal C^{1\text{-var}}_o([0,T],E)\right)}^{\,d_{p\text{-var}}} \subset \mathcal C^{p\text{-var}}_o([0,T],G^k(E))$ the space of geometric $p$-rough paths, that is, the closure of the canonical lifts of finite-variation paths for the $p$-variation rough path metric. It is a proper closed subset of $\mathcal C^{p\text{-var}}_o([0,T],G^k(E))$; see \citep[Corollary 8.26]{friz2010multidimensional}. We denote by $\mathbf{X}\in \mathcal C^{0,p\text{-var}}_o([0,T],G^2(\mathbb R^d))$ the Stratonovich geometric rough path lift of $(X_t)_{t\in[0,T]}$, and by $\operatorname{Supp}(\mathbf X)$ the support of its law with respect to the $p$-variation rough path topology, i.e. $\operatorname{Supp}(\mathbf X) :=\Bigl\{\mathbf x\in \mathcal C^{0,p\text{-var}}_o([0,T],G^2(\mathbb R^d)) ~|~  \forall \varepsilon>0, \mathbb{P}(d_{p\text{-var}}(\mathbf X,\mathbf x)<\varepsilon\bigr)>0\Bigr\}$. Since $\mathcal C^{0,p\text{-var}}_o([0,T],G^2(\mathbb R^d))$ is closed in $\mathcal C^{p\text{-var}}_o([0,T],G^2(\mathbb R^d))$ and $\mathbf X$ takes almost surely its values in the former, $\operatorname{Supp}(\mathbf X)$ also coincides with the support of the law of $\mathbf X$ computed in the larger space $\mathcal C^{p\text{-var}}_o([0,T],G^2(\mathbb R^d))$. For any finite-dimensional Euclidean space $E$ and any $k\ge 1$, we denote by $\pi_1:G^k(E)\longrightarrow E$ the first-level projection. Thus, if $\mathbf{x}\in\mathcal C^{p\text{-var}}_o([0,T],G^k(E))$, then $\pi_1(\mathbf{x})_t$ denotes the first-level component of $\mathbf{x}_t$. Since $t\mapsto t$ has finite variation and $1/p+1>1$, the Young pairing theorem
\citep[Section 9.4, Theorem 9.32]{friz2010multidimensional} provides a map $\mathcal C^{0,p\text{-var}}_o([0,T],G^2(\mathbb R^d)) \ni \mathbf{x} \mapsto
 \widehat{\mathbf{x}} \in \mathcal C^{0,p\text{-var}}_o([0,T],G^2(\mathbb R^{d+1}))$ which is continuous for the $p$-variation rough path metric and which satisfies $\widehat{S_2^1(x)}=S_2^1(\widehat x)$ for every $x\in\mathcal C^{1\text{-var}}_o([0,T],\mathbb R^d)$,
where $\widehat x_t:=(t,x_t)$. Since $S_2^1\left(\mathcal C^{1\text{-var}}_o([0,T],\mathbb R^d)\right)$ is
dense in $\mathcal C^{0,p\text{-var}}_o([0,T],G^2(\mathbb R^d))$, these two properties characterise $\mathbf{x}\mapsto\widehat{\mathbf{x}}$ uniquely; we call $\widehat{\mathbf{x}}$ the time-augmentation of $\mathbf{x}$, and we note that $\pi_1(\widehat{\mathbf{x}})_t=(t,\pi_1(\mathbf{x})_t)$ for all $t\in[0,T]$. Let $\Phi:
\mathcal{C}^{0,p\text{-var}}_{o}([0,T],G^{2}(\mathbb{R}^d))
\longrightarrow \mathbb R$ be defined by
\[
\Phi(\mathbf{x})
=
\sum_{\word{w} \in \mathcal{W}_{\le N}^*}
c_{\word{w}}
\left\langle
\word{w}, S_N^2(\widehat{\mathbf{x}})_T
\right\rangle.
\]
For $r\in\mathbb{N}^*$ and $n\in\mathbb N$, we denote by $\mathcal{D}^{r,n}\subset\mathcal C^{1\text{-var}}_o([0,T],\mathbb R^r)$ the set of paths starting from $0$ that are affine on each interval $[kT2^{-n},(k+1)T2^{-n}]$, $k=0,\ldots,2^n-1$, of the $n$-th dyadic partition of $[0,T]$, and we set $\mathcal{D}^r:=\bigcup_{n\in\mathbb N}\mathcal D^{r,n}$; that is, $\mathcal D^r$ is the set of paths starting from $0$ that are affine on each interval of some dyadic partition of $[0,T]$.

\medskip
\noindent
\textbf{First step: $\forall \mathbf{x} \in \operatorname{Supp}(\mathbf{X}), \Phi(\mathbf{x})=0$.} 
\newline
The map $\left(\mathcal{C}^{0,p\text{-var}}_{o}([0,T],G^2(\mathbb{R}^d)),d_{p\text{-var}}\right)  \to \left(\mathcal{C}^{0,p\text{-var}}_{o}([0,T],G^2(\mathbb{R}^{d+1})),d_{p\text{-var}}\right), \mathbf{x} \mapsto \widehat{\mathbf{x}}$ is continuous, as recalled above. The map $S_N^2$ is continuous from $\left(\mathcal{C}^{0,p\text{-var}}_{o}([0,T],G^2(\mathbb{R}^{d+1})),d_{p\text{-var}}\right)$ to $\left(\mathcal{C}^{0,p\text{-var}}_{o}([0,T],G^N(\mathbb{R}^{d+1})),d_{p\text{-var}}\right)$ (see \citep{friz2010multidimensional}[Corollary 9.11]). The terminal evaluation map $\mathbf{x}\mapsto \mathbf{x}_T$ is continuous from $\left(\mathcal{C}^{0,p\text{-var}}_{o}([0,T],G^N(\mathbb{R}^{d+1})),d_{p\text{-var}}\right)$ to $\left(G^N(\mathbb{R}^{d+1}),d_{\rm CC}\right)$, and the coordinate embedding of $G^N(\mathbb{R}^{d+1})$ into $\mathbb{R}^{\frac{(d+1)^{N+1}-1}{d}}$is continuous. Hence $\Phi$ is continuous with respect to the $p$-variation rough path topology. By \eqref{eq:forme_lineaire_nulle}, we have $\Phi(\mathbf{X})=0$ almost surely. Since $\Phi$ is continuous, it follows that $\Phi(\mathbf{x})=0$, for all $\mathbf{x}\in\operatorname{Supp}(\mathbf{X})$.

\medskip
\noindent
\textbf{Second step: $S_2^{1}(\mathcal{D}^d) \subset \operatorname{Supp}(\mathbf{X})$.} 
\newline
For $r\in\mathbb N$, we denote by $\mathcal D^r_p
:=
\overline{S_2^{1}(\mathcal D^r)}^{\,d_{p\text{-var}}}
\subset
\mathcal{C}^{p\text{-var}}_{o}([0,T],G^{2}(\mathbb{R}^r))$. For $k=1,\ldots,m$, we denote by $\sigma_k(t,x)\in\mathbb R^d$ the $k$-th column of $\sigma(t,x)$, and define the Stratonovich drift $\bar b(t,x)
:=
b(t,x)
-
\frac12\sum_{k=1}^m
\nabla_x\sigma_k(t,x)\sigma_k(t,x)$. The Itô equation \eqref{Ito form} is equivalently written in Stratonovich form as
\[
dX_t
=
\sigma(t,X_t)\circ d\beta_t
+
\bar b(t,X_t)\,dt.
\]
Since $b$ is of class $C^3$ and $\sigma$ is of class $C^4$, both with linear growth, the coefficients $\sigma$ and $\bar b$ of this Stratonovich equation are of class $C^3$ with linear growth. Therefore, by Lyons' continuity theorem \citep{lyons2007differential}, the Itô map associated with this equation, which is defined on lifted finite-variation controls by $S_2^{1}(h)\mapsto S_2^{1}(x^h)$ for $h\in \mathcal C^{1\text{-var}}_o([0,T],\mathbb R^m)$, where $x^h$ solves
\begin{align}
\label{controlled ODE}
\dot x^h_t
=
\sigma(t,x^h_t)\dot h_t+\bar b(t,x^h_t),
\qquad
x^h_0=0,
\end{align}
extends uniquely into a continuous map $\varphi:
\mathcal{C}^{0,p\text{-var}}_{o}([0,T],G^{2}(\mathbb{R}^m))
\longrightarrow
\mathcal{C}^{0,p\text{-var}}_{o}([0,T],G^{2}(\mathbb{R}^d))$, and $\mathbf{X}=\varphi(\mathbf B)$ almost surely, where $\mathbf B$ denotes the canonical Stratonovich rough path lift of the driving Brownian motion $\beta$; see \citep[Corollary 6]{ledoux2002large}. Let us stress that the extension by density underlying the definition of $\varphi$ takes place in $\mathcal{C}^{0,p\text{-var}}_{o}([0,T],G^{2}(\mathbb{R}^m))$, which is by definition the closure of $S_2^{1}\left(\mathcal C^{1\text{-var}}_o([0,T],\mathbb R^m)\right)$, and which is a proper closed subset of $\mathcal{C}^{p\text{-var}}_{o}([0,T],G^{2}(\mathbb{R}^m))$; this is the reason why all the rough path spaces above are taken to be the geometric ones. In particular, we deduce from the continuity of $\varphi$ that
\begin{align}
\label{eq:inclusion_1}
\varphi(\mathcal D^m_p)
\subset
\overline{\varphi(S_2^{1}(\mathcal D^m))}^{\,d_{p\text{-var}}}.
\end{align}
By \citep{ledoux2002large}[Corollary 6] applied to the above Stratonovich equation,
\begin{align}
\label{eq:inclusion_2}
\operatorname{Supp}(\mathbf{X})
=
\overline{\varphi(S_2^{1}(\mathcal D^m))}^{\,d_{p\text{-var}}}.
\end{align}
Moreover $S_2^{1}(\mathcal H^m)\subset \mathcal D^m_p$ where $\mathcal H^m
=
\left\{
h=\int_0^{\cdot} \dot h_s\,ds ~|~
\dot h\in L^2([0,T],\mathbb R^m)
\right\}$ denotes the Cameron--Martin space. Indeed, the dyadic piecewise affine approximations of any $h\in\mathcal H^m$ converge, after canonical lifting, to $S_2^{1}(h)$ in the $p$-variation rough path topology; see \citep{ledoux2002large} for more details. Therefore, \eqref{eq:inclusion_1} and \eqref{eq:inclusion_2} yield
\begin{align}
\label{eq:inclusions}
\varphi(S_2^{1}(\mathcal H^m))
\subset
\varphi(\mathcal D^m_p)
\subset
\overline{\varphi(S_2^{1}(\mathcal D^m))}^{\,d_{p\text{-var}}}
=
\operatorname{Supp}(\mathbf{X}).
\end{align}
Let us now show that $S_2^{1}(\mathcal{D}^d) \subset \varphi(S_{2}^{1}(\mathcal{H}^m))$. Let $\gamma \in \mathcal{D}^d$. We construct a control $h\in\mathcal H^m$ such that $x^h=\gamma$. Since $\gamma$ is dyadic piecewise affine, its derivative $\dot\gamma_t$ exists $dt$-a.e.. Define, for $dt$-a.e. $t \in [0,T]$,
\[
\dot h_t
=
\sigma(t,\gamma_t)^\top
\left(
\sigma(t,\gamma_t)\sigma(t,\gamma_t)^\top
\right)^{-1}
\left(
\dot\gamma_t
-
\bar{b}(t,\gamma_t)
\right).
\]
The path $\gamma$ has compact range, $\dot\gamma$ is piecewise constant and bounded, and the coefficients $b,\sigma_k,\nabla_x\sigma_k$ are continuous. By uniform ellipticity, the inverse of $\sigma(t,\gamma_t)\sigma(t,\gamma_t)^\top$ is uniformly bounded along the path $\gamma$. Hence $\dot h\in L^2([0,T],\mathbb R^m)$, so $h\in\mathcal H^m$. Furthermore, by construction, $\dot\gamma_t
=
\sigma(t,\gamma_t)\dot h_t+\bar b(t,\gamma_t)$ $dt$-a.e.. Therefore $\gamma$ solves the controlled ODE associated with $h$. Uniqueness of the solution to the controlled ODE \eqref{controlled ODE} yields $x^h=\gamma$. Consequently, $S_2^{1}(\gamma)
=
S_2^{1}(x^h)
=
\varphi(S_2^{1}(h))$. It comes that $S_2^{1}(\mathcal{D}^d) \subset \varphi(S_{2}^{1}(\mathcal{H}^m))$, and \eqref{eq:inclusions} yields $S_2^{1}(\mathcal{D}^d) \subset \operatorname{Supp}(\mathbf{X})$.

\medskip
\noindent
\textbf{Third step.}
\newline
Let $(W_t)_{t\in[0,T]}$ be a $d$-dimensional Brownian motion and let $(W^{(n)})_{n\ge1}$ be its dyadic piecewise affine interpolation. For every $n \in \mathbb{N}^*$ and every $\omega \in \Omega$, the path $W^{(n)}(\omega)$ belongs to $\mathcal D^d$. Hence, by the first two steps, $\Phi\left(S_2^{1}(W^{(n)}(\omega))\right)=0$ for all $n\in\mathbb{N}^*$. By the almost-sure convergence of dyadic piecewise affine approximations to the canonical Brownian rough path in the $p$-variation rough path topology \citep{lyons2007differential, ledoux2002large,friz2014course}, $S_2^{1}(W^{(n)})
\longrightarrow
\mathbf W$ a.s., for the $p$-variation rough path topology, where $\mathbf W$ denotes the canonical Stratonovich Brownian rough path. Since $\Phi$ is continuous, we obtain $\Phi(\mathbf W)=0$ almost-surely. Equivalently, $\sum_{\word{w}\in\mathcal W^*_{\le N}}
c_{\word{w}}
\widehat{\mathbb W}^{\word{w}}_T
=
0$, almost-surely. By Proposition \ref{prop brownian}, it follows that $c_{\word{w}}=0$, for all $\word{w}\in\mathcal W^{*}_{\le N}$. This proves the result.
\end{proof}

\subsubsection{Proof of Proposition \ref{prop brownian}}
\label{section prop brownian}

The proof of Proposition \ref{prop brownian} relies on the following lemmas. Their proofs are postponed after the proof of Proposition \ref{prop brownian}. In the remainder of this section, $(\Omega,\mathcal{F},\mathbb{P})$ and $(W_t)_{t \ge 0}$ denote respectively the probability space and the $d$-dimensional Brownian motion introduced in Proposition \ref{prop brownian}. 

\begin{lemma}
\label{liberté trajectorielle}

Let $T > 0$, $N \in \mathbb{N}$ and $(c_{\word{w}})_{\word{w} \in \mathcal{W}_{\le N}}$ be such that:

\begin{align*}
\forall t \le T, \sum_{\word{w} \in \mathcal{W}_{\le N}} c_{\word{w}} \widehat{\mathbb{W}}^{\word{w}}_t = 0 ~~\text{a.s.}. 
\end{align*}
Then $c_{\word{w}}=0$ for each $\word{w} \in \mathcal{W}_{\le N}$.

\end{lemma}

\begin{lemma}
\label{lemme utile strato}

Let $\word{i} \in \{\word{0},\cdots,\word{d} \}$, $\word{j} \in \{\word{1},\cdots,\word{d} \}$ and $\word{w} \in \mathcal{W}_{< \infty}$. Then 

\begin{align*}
    \widehat{\mathbb{W}}_{t}^{\word{wij}} &= \int_{0}^{t} \widehat{\mathbb{W}}_{s}^{\word{wi}} dW_{s}^{\word{j}} + \frac{1}{2} \mathbbm{1}_{\word{i}=\word{j}} \int_{0}^{t} \widehat{\mathbb{W}}_{s}^{\word{w}} ds,
\end{align*}
where the first integral is understood as an Itô integral.   
\end{lemma}

\begin{lemma}
\label{esp cond}
Let $\word{w} \in \mathcal{W}_{< \infty} \setminus \{ \emptyword \}$ and $\word{i} \in \{\word{0},\cdots,\word{d}\}$. If $\word{w}=\word{\gamma i}$ with $\word{i} \neq \word{0}$, then there exist a finite family of polynomials $(P_{k,\word{w}})_{1 \le k \le M^{\word{w}}}$ and a finite family of words $(\word{v}_{k,\word{w}})_{1 \le k \le M^{\word{w}}}$ of length smaller than the length of $\word{w}$ such that:

\begin{enumerate}
\item[(i)] $\forall T > 0, \forall t \le T, \mathbb{E}\left[\widehat{\mathbb{W}}^{\word{w}}_T|\mathcal{F}_t\right]=\widehat{\mathbb{W}}^{\word{w}}_t + \sum_{k=1}^{M^{\word{w}}}P_{k,\word{w}}(T)\widehat{\mathbb{W}}^{\word{v}_{k,\word{w}}}_t$, 
\item[(ii)] $\forall k \le M^{\word{w}}, \text{deg}(P_{k,\word{w}}) + \|\word{v}_{k,\word{w}}\| \le \|\word{w}\|-1$.
\end{enumerate}
If $\word{w}=\word{\gamma 0}$, then there exist a finite family of polynomials $(P_{k,\word{w}})_{1 \le k \le M^{\word{w}}}$ and a finite family of words $(\word{v}_{k,\word{w}})_{1 \le k \le M^{\word{w}}}$ of length smaller than the length of $\word{w}$ such that:

\begin{enumerate}
\item[(i)] $\forall T > 0, \forall t \le T, \mathbb{E}\left[\widehat{\mathbb{W}}^{\word{w}}_T|\mathcal{F}_t\right]= \sum_{k=1}^{M^{\word{w}}}P_{k,\word{w}}(T)\widehat{\mathbb{W}}^{\word{v}_{k,\word{w}}}_t$,
\item[(ii)] $\forall k \le M^{\word{w}}, \text{deg}(P_{k,\word{w}}) + \|\word{v}_{k,\word{w}}\| \le \|\word{w}\|$.
\end{enumerate}

\end{lemma}

\begin{remark}
\label{esp cond emptyword}
Since $\widehat{\mathbb{W}}_{T}^{\emptyword}=1$, we trivially have $\mathbb{E}\left[\widehat{\mathbb{W}}_{T}^{\emptyword} ~|~ \mathcal{F}_t\right] = \widehat{\mathbb{W}}_{t}^{\emptyword}$.
\end{remark}

We are ready to provide the proof of Proposition \ref{prop brownian}.

\begin{proof}[Proof of Proposition \ref{prop brownian}]
\label{proof of theorem brownian motion}
Let $T > 0$, $N \in \mathbb{N}$, and $(c_{\word{w}})_{\word{w} \in \mathcal{W}_{\le N}^{*}} \in \mathbb{R}^{\mathcal{W}_{\le N}^{*}}$ be such that 
\[
F:=\sum_{\word{w} \in  \mathcal{W}_{\le N}^{*}} c_{\word{w}} \widehat{\mathbb{W}}^{\word{w}}_T = 0 \text{ a.s.}.
\]
For all $t \le T$ we have $0=\mathbb{E}\left[F ~|~\mathcal{F}_t\right]= \sum_{\word{w} \in \mathcal{W}_{\le N}^{*}} c_{\word{w}} \mathbb{E}\left[\widehat{\mathbb{W}}^{\word{w}}_T|\mathcal{F}_t\right]$. By Lemma \ref{esp cond} and Remark \ref{esp cond emptyword} we can write

\begin{align*}
\mathbb{E}\left[F|\mathcal{F}_t\right] &= \sum_{\word{w} \in \mathcal{W}_{\le N}} c_{\word{w}} \mathbbm{1}_{\{ \word{w} \in \mathcal{W}_{\le N}^{*} \}} \left( \widehat{\mathbb{W}}_{t}^{\word{w}}+ \sum_{\word{w'} \in \mathcal{W}_{\le N}} \lambda_{\word{w},\word{w'}} \mathbbm{1}_{\{ \|\word{w'}\| < \|\word{w}\| \}} \widehat{\mathbb{W}}_{t}^{\word{w'}} \right).
\end{align*}
We define $\hat{c}_{\word{w}} = c_{\word{w}}\mathbbm{1}_{\{\word{w} \in \mathcal{W}_{\le N}^{*}\}}$ and $\left(P \hat{c}\right)_{\word{w'}} = \sum_{\word{w} \in \mathcal{W}_{\le N}} P_{\word{w'},\word{w}} \hat{c}_{\word{w}}$ with 
\[
\left(P_{\word{w'},\word{w}} = \mathbbm{1}_{\{ \word{w}=\word{w'}\}} + \lambda_{\word{w},\word{w'}} \mathbbm{1}_{\{ \|\word{w'}\| < \|\word{w}\| \}}\right)_{\word{w'},\word{w} \in \mathcal{W}_{\le N}}.
\]
We can write

\[
\mathbb{E}\left[F|\mathcal{F}_t\right]= \sum_{\word{w} \in \mathcal{W}_{\le N}} \hat{c}_{\word{w}} \sum_{\word{w'} \in \mathcal{W}_{\le N}} \left( \mathbbm{1}_{\{ \word{w}=\word{w'}\}} + \lambda_{\word{w},\word{w'}} \mathbbm{1}_{\{ \|\word{w'}\| < \|\word{w}\| \}} \right) \widehat{\mathbb{W}}_{t}^{\word{w'}} = \sum_{\word{w'} \in \mathcal{W}_{\le N}} \left(P\hat{c}\right)_{\word{w'}} \widehat{\mathbb{W}}_{t}^{\word{w'}}.
\]
Then for all $t \le T$, $ \sum_{\word{w'} \in \mathcal{W}_{\le N}} \left(P\hat{c}\right)_{\word{w'}} \widehat{\mathbb{W}}_{t}^{\word{w'}}=0$ almost surely. Lemma \ref{liberté trajectorielle} implies that $\left(P\hat{c}\right)_{\word{w'}} = 0$ for all $\word{w'} \in \mathcal{W}_{\le N}$. If we sort the words according to lexicographic order on $(\|\word{w}\|,\word{w})$, then it is clear that $P$ is an upper triangular matrix with diagonal coefficients equal to $1$. Therefore $P$ is invertible and $\hat{c}_{\word{w}}=0$ for all $\word{w} \in \mathcal{W}_{\le N}$. Finally, we obtain $c_{\word{w}}=0$ for all $\word{w} \in \mathcal{W}_{\le N}$.

\end{proof}

\begin{proof}[Proof of Lemma \ref{liberté trajectorielle}]
We prove this result by induction on $N \in \mathbb{N}$.
\newline
\\
\underline{Initialization step.}
Let $N=0$. The conclusion is trivial since for every $t \le T$, $\widehat{\mathbb{W}}^{\emptyword}_t=1$.
\newline
\\
\underline{Induction step.}
We assume that the property holds for $N \in \mathbb{N}$. Let $(c_{\word{w}})_{\word{w} \in \mathcal{W}_{\le N+1}}$ be such that for all $t \le T$ we have $\sum_{\word{w} \in \mathcal{W}_{\le N+1}} c_{\word{w}} \widehat{\mathbb{W}}^{\word{w}}_t = 0$ almost surely. We have $\mathcal{W}_{\le N+1} = \{\emptyword\} \cup \{\word{w0} ~|~ \word{w} \in \mathcal{W}_{\le N} \} \cup \{\word{wi} ~|~ \word{w} \in \mathcal{W}_{\le N}, \word{i} \in \{ \word{1}, \cdots, \word{d} \} \}$. Let $t \le T$. We can write:

\begin{align}
\label{eq recurrence}
0&\overset{\text{a.s.}}{=}\sum_{\word{w} \in \mathcal{W}_{\le N+1}} c_{\word{w}} \widehat{\mathbb{W}}^{\word{w}}_t= c_{\emptyword} + \sum_{\word{w} \in \mathcal{W}_{\le N}} c_{\word{w0}} \widehat{\mathbb{W}}^{\word{w0}}_t + \sum_{\word{i} \in \{\word{1},\cdots,\word{d}\}} \sum_{\word{w} \in \mathcal{W}_{\le N}} c_{\word{wi}} \widehat{\mathbb{W}}^{\word{wi}}_t.
\end{align}
For $\word{j} \in \{\word{1},\cdots,\word{d}\}$, we obtain:

\begin{align*}
0 &\overset{\text{a.s.}}{=}\left<c_{\emptyword} + \sum_{\word{w} \in \mathcal{W}_{\le N}} c_{\word{w0}} \widehat{\mathbb{W}}^{\word{w0}}_{.} + \sum_{\word{i} \in \{\word{1},\cdots,\word{d}\}}\sum_{\word{w} \in \mathcal{W}_{\le N}} c_{\word{wi}} \widehat{\mathbb{W}}^{\word{wi}}_{.}, W_{.}^{\word{j}}\right>_t = \sum_{\word{w} \in \mathcal{W}_{\le N}} c_{\word{wj}} \left<\widehat{\mathbb{W}}^{\word{wj}}_{.},W_{.}^{\word{j}}\right>_t \\
&= \sum_{\word{w} \in \mathcal{W}_{\le N}} c_{\word{wj}} \left<\int_{0}^{.}\widehat{\mathbb{W}}^{\word{w}}_u \circ dW_u^{\word{j}}, W_{.}^{\word{j}} \right>_t= \sum_{\word{w} \in \mathcal{W}_{\le N}} c_{\word{wj}} \int_{0}^{t} \widehat{\mathbb{W}}^{\word{w}}_u du = \int_{0}^{t} \sum_{\word{w} \in \mathcal{W}_{\le N}} c_{\word{wj}} \widehat{\mathbb{W}}^{\word{w}}_u du.
\end{align*}
It holds for all $t \le T$, and thanks to the continuity of the sample paths of $\left(\sum_{\word{w} \in \mathcal{W}_{\le N}} c_{\word{wi}} \widehat{\mathbb{W}}^{\word{w}}_t\right)_{t \ge 0}$ we get $\sum_{\word{w} \in \mathcal{W}_{\le N}} c_{\word{wj}} \widehat{\mathbb{W}}^{\word{w}}_t \overset{ \text{a.s.}}{=} 0$ for all $t \le T$. The induction hypothesis yields $c_{\word{wj}}=0$, for all $\word{w} \in \mathcal{W}_{\le N}$ and for all $\word{j} \in \{\word{1},\cdots,\word{d}\}$. By plugging these coefficients in \eqref{eq recurrence} we obtain

\[
\forall t \le T,  -c_{\emptyword} \overset{\text{a.s.}}{=} \sum_{\word{w} \in \mathcal{W}_{\le N}}c_{\word{w0}} \widehat{\mathbb{W}}^{\word{w0}}_t= \int_{0}^{t} \sum_{\word{w} \in \mathcal{W}_{\le N}} c_{\word{w0}} \widehat{\mathbb{W}}^{\word{w}}_u du.
\]
We can again use the continuity of $\left(\sum_{\word{w} \in \mathcal{W}_{\le N}} c_{\word{w0}} \widehat{\mathbb{W}}^{\word{w}}_t\right)_{t \ge 0}$ and conclude that $c_{\word{w0}}=0$ for all $\word{w} \in \mathcal{W}_{\le N}$. Finally, we have $c_{\emptyword}=0$.

\end{proof}

We now provide the proof of Lemma \ref{lemme utile strato}

\begin{proof}[Proof of Lemma \ref{lemme utile strato}]

Let $\word{i} \in \{\word{0},\cdots,\word{d} \}$, $\word{j} \in \{\word{1},\cdots,\word{d} \}$, and $\word{w} \in \mathcal{W}_{< \infty}$. Then

\begin{align*}
    \widehat{\mathbb{W}}_{t}^{\word{wij}}&=\int_{0}^{t} \widehat{\mathbb{W}}_{s}^{\word{wi}} \circ dW_{s}^{\word{j}} =\int_{0}^{t} \widehat{\mathbb{W}}_{s}^{\word{wi}} dW_{s}^{\word{j}} + \frac{1}{2} \left< \widehat{\mathbb{W}}_{\textbf{.}}^{\word{wi}},W_{\textbf{.}}^{\word{j}}\right>_t \\
 &= \int_{0}^{t} \widehat{\mathbb{W}}_{s}^{\word{wi}} dW_{s}^{\word{j}} + \frac{1}{2}\left<\int_{0}^{\textbf{.}} \widehat{\mathbb{W}}_{u}^{\word{w}} dW_{u}^{\word{i}} + \frac{1}{2}\left< \widehat{\mathbb{W}}_{.}^{\word{w}}, W_{.}^{\word{i}}\right>_{\textbf{.}},W^{\word{j}}_{\textbf{.}}\right>_{t} \\ &=  \int_{0}^{t} \widehat{\mathbb{W}}_{s}^{\word{wi}} dW_{s}^{\word{j}} + \frac{1}{2} \int_{0}^{t} \widehat{\mathbb{W}}_{s}^{\word{w}} d\left< W_{\textbf{.}}^{\word{i}},W_{\textbf{.}}^{\word{j}}\right>_s = \int_{0}^{t} \widehat{\mathbb{W}}_{s}^{\word{wi}} dW_{s}^{\word{j}} + \frac{1}{2} \mathbbm{1}_{\word{i}=\word{j}} \int_{0}^{t} \widehat{\mathbb{W}}^{\word{w}}_s ds.
\end{align*}

\end{proof}

\begin{proof}[Proof of Lemma \ref{esp cond}]
We prove this result by induction on the length of the word $N$. Because the signature of any Brownian motion is defined with the Stratonovich integration rule, the behavior of $\mathbb{E}\left[\widehat{\mathbb{W}}^{\word{w}}_T | \mathcal{F}_t\right]$ for a given word $\word{w}$ of length non smaller than $2$ will depend strongly on the last two letters of $\word{w}$. To address this issue, we need to initialize the inductive proof by showing that the property is true for $N \in \{1,2\}$.
\newline
\\
\underline{Initialization step.}
\newline
\\
\textbf{First case: $N=1$.}
\newline
Let $0 \le t \le T$. We have $\mathbb{E}\left[\widehat{\mathbb{W}}^{\word{0}}_T|\mathcal{F}_t\right]=T=T \widehat{\mathbb{W}}^{\emptyword}_t$. For $\word{i} \in \{ \word{1},\cdots,\word{d} \}$,  $\mathbb{E}\left[\widehat{\mathbb{W}}^{\word{i}}_T|\mathcal{F}_t\right]=\mathbb{E}\left[W_{T}^{\word{i}}|\mathcal{F}_t\right]=W_{t}^{\word{i}}=\widehat{\mathbb{W}}^{\word{i}}_t$.
\newline
\newline
\textbf{Second case: $N=2$.}
\newline
Let $0 \le t \le T$. We have $\mathcal{W}_2= \{\word{00} \} \cup \{ \word{ij}~|~ \word{j} \neq \word{0} , \word{i} \neq \word{j}\} \cup \{\word{i0} ~|~ \word{i} \neq \word{0} \} \cup \{ \word{ii} ~|~ \word{i} \neq \word{0} \}$. Let $\word{w}=\word{00}$. Because $\widehat{\mathbb{W}}^{\word{00}}_T=\frac{T^2}{2}$ is deterministic, we can use the same argument as for $\widehat{\mathbb{W}}_{T}^{\word{0}}$. 
\newline
Let $\word{w}=\word{i0}$ with $\word{i} \neq \word{0}$. We have $\widehat{\mathbb{W}}^{\word{i0}}_T = \int_{0}^{T} W_u^{\word{i}} du$, and so $\mathbb{E}\left[\widehat{\mathbb{W}}^{\word{i0}}_T|\mathcal{F}_t\right]=\int_{0}^{t}W_u^{\word{i}}du + W_t^{\word{i}}(T-t) = \int_{0}^{t}W_u^{\word{i}}du + T W_t^{i} - \int_{0}^{t}u dW_u^{\word{i}} - \int_{0}^{t} W_u^{\word{i}} du=T\widehat{\mathbb{W}}^{\word{i}}_t - \widehat{\mathbb{W}}^{\word{0i}}_t$. 
\newline
Let $\word{w}=\word{ij}$, with $\word{j} \ne \word{0}$ and $\word{i} \neq \word{j}$. We have $\mathbb{E}\left[\widehat{\mathbb{W}}^{\word{ij}}_T | \mathcal{F}_t\right]=\mathbb{E}\left[\int_{0}^{T}\widehat{W}_{u}^{\word{i}} \circ dW_u^{\word{j}} | \mathcal{F}_t\right]$. For any values of $\word{i}$, we have $\int_{0}^{t}\widehat{W}_{u}^{\word{i}} \circ dW_u^{\word{j}}=\int_{0}^{t}\widehat{W}_{u}^{\word{i}} dW_u^{\word{j}}$ (see Lemma \ref{lemme utile strato}) and thus $\left(\int_{0}^{t}\widehat{W}_{u}^{\word{i}} \circ dW_u^{\word{j}}\right)_{t \le T}$ is a martingale. Then $\mathbb{E}\left[\widehat{\mathbb{W}}^{\word{ij}}_T | \mathcal{F}_t\right]=\int_{0}^{t}\widehat{W}_{u}^{\word{i}} \circ dW_u^{\word{j}}=\widehat{\mathbb{W}}^{\word{ij}}_t$. Let $\word{w}=\word{ii}$ with $\word{i} \neq \word{0}$. Then $\mathbb{E}\left[\widehat{\mathbb{W}}^{\word{ii}}_T | \mathcal{F}_t\right]= \mathbb{E}\left[\frac{\left(W_{T}^{\word{i}}\right)^2}{2}| \mathcal{F}_t\right]=\frac{\left(W_t^{\word{i}}\right)^{2}}{2} + \frac{T-t}{2}=\widehat{\mathbb{W}}^{\word{ii}}_t + \frac{T}{2}\widehat{\mathbb{W}}^{\emptyword}_t - \frac{1}{2}\widehat{\mathbb{W}}^{\word{0}}_t$.
\newline
\\
\underline{Induction step.}
Let $N \ge 2$ be such that the property holds for all $n \le N$. Let $T > 0$ and $t \le T$. We have $\mathcal{W}_{N+1} = \{\word{\gamma 00} ~|~\word{\gamma} \in \mathcal{W}_{N-1} \} \cup \{ \word{\gamma ij}~|~ \word{\gamma} \in \mathcal{W}_{N-1}, \word{j} \neq \word{0} , \word{i} \neq \word{j}\} \cup \{\word{\gamma i0} ~|~ \word{\gamma} \in \mathcal{W}_{N-1}, \word{i} \neq \word{0} \} \cup \{ \word{\gamma ii} ~|~ \word{\gamma} \in \mathcal{W}_{N-1}, \word{i} \neq \word{0} \}$. Let $\gamma \in \mathcal{W}_{N-1}$.
\newline
\newline
\textbf{First case}: $\word{\gamma ij}, \word{j} \neq \word{0}, \word{i} \neq \word{j}.$
\newline
We have $\widehat{\mathbb{W}}^{\word{\gamma ij}}_T=\int_{0}^{T}\widehat{\mathbb{W}}^{\word{\gamma i}}_u \circ dW_u^{\word{j}}$, and by Lemma \ref{lemme utile strato}, we have $\int_{0}^{t}\widehat{\mathbb{W}}^{\word{\gamma i}}_u \circ dW_u^{\word{j}} = \int_{0}^{t}\widehat{\mathbb{W}}^{\word{\gamma i}}_u dW_u^{\word{j}}$. Then $\mathbb{E}\left[\widehat{\mathbb{W}}^{\word{\gamma ij}}_T|\mathcal{F}_t \right]=\widehat{\mathbb{W}}^{\word{\gamma ij}}_t$.
\newline
\\
\textbf{Second case: $\word{\gamma00}$.}
\newline
We know that $\widehat{\mathbb{W}}^{\word{\gamma00}}_T=\int_{0}^{T}\widehat{\mathbb{W}}^{\word{\gamma0}}_u du$, then 

\begin{align*}
\mathbb{E}\left[\widehat{\mathbb{W}}^{\word{\gamma00}}_T | \mathcal{F}_t\right]=  \mathbb{E}\left[\int_{0}^{T} \widehat{\mathbb{W}}^{\word{\gamma0}}_u du | \mathcal{F}_t\right] =\int_{0}^{t}\widehat{\mathbb{W}}^{\word{\gamma0}}_udu + \int_{t}^{T} \mathbb{E}\left[\widehat{\mathbb{W}}^{\word{\gamma0}}_u|\mathcal{F}_t \right]du =\widehat{\mathbb{W}}^{\word{\gamma00}}_t + \int_{t}^{T} \mathbb{E}\left[\widehat{\mathbb{W}}^{\word{\gamma0}}_u|\mathcal{F}_t \right]du.
\end{align*}
The induction hypothesis yields $\mathbb{E}\left[\widehat{\mathbb{W}}^{\word{\gamma0}}_u|\mathcal{F}_t \right] = \sum_{k=1}^{M^{\word{\gamma 0}}}P_{k,\word{\gamma0}}(u)\widehat{\mathbb{W}}^{\word{v}_{k,\word{\gamma0}}}_t$ for all $u \ge t$. Then, if $Q_{k,\word{\gamma0}}$ denotes the antiderivative of the polynomial $P_{k,\word{\gamma0}}$ with no constant term, we can write 

\begin{align*}
\mathbb{E}\left[\widehat{\mathbb{W}}^{\word{\gamma00}}_T|\mathcal{F}_t \right] &=\widehat{\mathbb{W}}^{\word{\gamma00}}_t + \sum_{k=1}^{M^{\word{\gamma0}}} \widehat{\mathbb{W}}^{\word{v}_{k,\word{\gamma0}}}_t \left[Q_{k,\word{\gamma0}}(T) - Q_{k,\word{\gamma0}}(t)\right] \\
&= \widehat{\mathbb{W}}^{\word{\gamma00}}_t + \sum_{k=1}^{M^{\word{\gamma0}}} \widehat{\mathbb{W}}^{\word{v_{k,\word{\gamma0}}}}_tQ_{k,\word{\gamma0}}(T) - \sum_{k=1}^{M^{\word{\gamma0}}} \widehat{\mathbb{W}}^{\word{v}_{k,\word{\gamma0}}}_t Q_{k,\word{\gamma0}}(t).
\end{align*}
Let $k \le M^{\word{\gamma0}}$. We have $d_{k,\word{\gamma0}}:=\deg(Q_{k,\word{\gamma0}})=\deg(P_{k,\word{\gamma0}})+1$. By writing $Q_{k,\word{\gamma0}}(t)=\sum_{m=1}^{d_{k,\word{\gamma0}}} a_{m,k} t^{m} = \sum_{m=1}^{d_{k,\word{\gamma0}}} a_{m,k} m! \widehat{\mathbb{W}}^{\word{0}_{m}}_t$, and using the shuffle product property of the signature, we get 

\begin{align*}
Q_{k,\word{\gamma0}}(t)\widehat{\mathbb{W}}^{\word{v}_{k,\word{\gamma0}}}_t = \sum_{m=1}^{d_{k,\word{\gamma0}}} a_{m,k} m! \widehat{\mathbb{W}}^{\word{0}_{m}}_t\widehat{\mathbb{W}}^{\word{v}_{k,\word{\gamma0}}}_t &=\sum_{m=1}^{d_{k,\word{\gamma0}}} a_{m,k} m! \left<\word{0}_{m},\widehat{\mathbb{W}}_t \right>\left<\word{v}_{k,\word{\gamma0}},\widehat{\mathbb{W}}_t\right> \\
&=\sum_{m=1}^{d_{k,\word{\gamma0}}} a_{m,k} m! \left<\word{0}_{m} \shuffle \word{v}_{k,\word{\gamma0}},\widehat{\mathbb{W}}_t\right>.
\end{align*}
Finally, we have 

\[
\mathbb{E}\left[\widehat{\mathbb{W}}^{\word{\gamma00}}_T | \mathcal{F}_t\right]=\widehat{\mathbb{W}}^{\word{\gamma00}}_t + \sum_{k=1}^{M^{\word{\gamma0}}} \widehat{\mathbb{W}}^{\word{v}_{k,\word{\gamma0}}}_t Q_{k,\word{\gamma0}}(T) - \sum_{k=1}^{M^{\word{\gamma0}}} \sum_{m=1}^{d_{k,\word{\gamma0}}} a_{m,k} m! \left<\word{0}_{m} \shuffle \word{v}_{k,\word{\gamma0}},\widehat{\mathbb{W}}_t\right>.
\]
For $k \le M^{\word{\gamma0}}$, the induction hypothesis yields $\deg(P_{k,\word{\gamma0}}) + \| \word{v}_{k,\word{\gamma0}} \| \le N$ so $\deg(Q_{k,\word{\gamma0}})+ \|\word{v}_{k,\word{\gamma0}}\| = \deg(P_{k,\word{\gamma0}})+ 1+\|\word{v}_{k,\word{\gamma0}}\| \le N+1$. Moreover, for $m \le d_{k,\word{\gamma0}}=\text{deg}(Q_{k,\word{\gamma 0}})$, $\word{0}_{m} \shuffle \word{v}_{k,\word{\gamma0}}$ is a linear combination of words of length $m + \|\word{v}_{k,\word{\gamma0}}\| \le \|\word{v}_{k,\word{\gamma0}}\|+  \deg(P_{k,\word{\gamma0}}) + 1 \le N+1$.
\newline
\newline
\textbf{Third case: $\word{\gamma i0}, \word{i} \neq \word{0}$.} 
\newline
We have 

\[
\mathbb{E}\left[\widehat{\mathbb{W}}^{\word{\gamma i0}}_T | \mathcal{F}_t\right]=  \mathbb{E}\left[\int_{0}^{T} \widehat{\mathbb{W}}^{\word{\gamma i}}_u du | \mathcal{F}_t\right] =\int_{0}^{t}\widehat{\mathbb{W}}^{\word{\gamma i}}_udu + \int_{t}^{T} \mathbb{E}\left[\widehat{\mathbb{W}}^{\word{\gamma i}}_u|\mathcal{F}_t \right]du =\widehat{\mathbb{W}}^{\word{\gamma i0}}_t + \int_{t}^{T} \mathbb{E}\left[\widehat{\mathbb{W}}^{\word{\gamma i}}_u|\mathcal{F}_t \right]du.
\] 
The induction hypothesis yields $\mathbb{E}\left[\widehat{\mathbb{W}}^{\word{\gamma i0}}_T | \mathcal{F}_t\right]=\widehat{\mathbb{W}}^{\word{\gamma i0}}_t + \int_{t}^{T} \left(\widehat{\mathbb{W}}^{\word{\gamma i}}_t +  \sum_{k=1}^{M^{\word{\gamma i}}}P_{k,\word{\gamma i}}(u)\widehat{\mathbb{W}}^{\word{v}_{k,\word{\gamma i}}}_t\right)du$. If $Q_{k,\word{\gamma i}}$ denotes the antiderivative of the polynomial $P_{k,\word{\gamma i}}$ with no constant term, and $Q_{k,\word{\gamma i}}(t) = \sum_{m=1}^{d_{k, \word{\gamma i}}} a_{m,k} t^m$, we can write 

\begin{align*}
\mathbb{E}\left[\widehat{\mathbb{W}}^{\word{\gamma i0}}_T | \mathcal{F}_t\right]&= \widehat{\mathbb{W}}^{\word{\gamma i0}}_t + \widehat{\mathbb{W}}^{\word{\gamma i}}_t(T-t) + \sum_{k=1}^{M^{\word{\gamma i}}} \widehat{\mathbb{W}}^{\word{v}_{k,\word{\gamma i}}}_t \left[Q_{k,\word{\gamma i}}(T) - Q_{k,\word{\gamma i}}(t)\right] \\
&= \widehat{\mathbb{W}}^{\word{\gamma i0}}_t + \widehat{\mathbb{W}}^{\word{\gamma i}}_{t} T - t\widehat{\mathbb{W}}^{\word{\gamma i}}_t +  \sum_{k=1}^{M^{\word{\gamma i}}} \widehat{\mathbb{W}}^{\word{v}_{k,\word{\gamma i}}}_t Q_{k,\word{\gamma i}}(T) - \sum_{k=1}^{M^{\word{\gamma i}}} \widehat{\mathbb{W}}^{\word{v}_{k,\word{\gamma i}}}_t Q_{k,\word{\gamma i}}(t) \\
&=\widehat{\mathbb{W}}^{\word{\gamma i0}}_t + \widehat{\mathbb{W}}^{\word{\gamma i}}_tT - \left<\word{0},\widehat{\mathbb{W}}_t\right> \left<\word{\gamma i},\widehat{\mathbb{W}}_t\right> +  \sum_{k=1}^{M^{\word{\gamma i}}} \widehat{\mathbb{W}}^{\word{v}_{k,\word{\gamma i}}}_t  Q_{k,\word{\gamma i}}(T) - \sum_{k=1}^{M^{\word{\gamma i}}} \widehat{\mathbb{W}}^{\word{v}_{k,\word{\gamma i}}}_t Q_{k,\word{\gamma i}}(t).
\end{align*}
By using the shuffle product property of the signature on $\left<\word{0},\widehat{\mathbb{W}}_t\right> \left<\word{\gamma i},\widehat{\mathbb{W}}_t\right>$, we can rewrite:

\begin{align*}
\mathbb{E}\left[\widehat{\mathbb{W}}^{\word{\gamma i0}}_T | \mathcal{F}_t\right]&=\widehat{\mathbb{W}}^{\word{\gamma i0}}_t + \widehat{\mathbb{W}}^{\word{\gamma i}}_tT - \left<\word{0} \shuffle \word{\gamma i},\widehat{\mathbb{W}}_t\right>  +  \sum_{k=1}^{M^{\word{\gamma i}}} \widehat{\mathbb{W}}^{\word{v}_{k,\word{\gamma i}}}_t Q_{k,\word{\gamma i}}(T) - \sum_{k=1}^{M^{\word{\gamma i}}} \widehat{\mathbb{W}}^{\word{v}_{k,\word{\gamma i}}}_t Q_{k,\word{\gamma i}}(t).
\end{align*}
Using the same arguments as in the previous case we finally have:

\begin{align*}
&\mathbb{E}\left[\widehat{\mathbb{W}}^{\word{\gamma i0}}_T | \mathcal{F}_t\right] \\
& \qquad=\widehat{\mathbb{W}}^{\word{\gamma i0}}_t + \widehat{\mathbb{W}}^{\word{\gamma i}}_tT - \left<\word{0} \shuffle \word{\gamma i},\widehat{\mathbb{W}}_t\right> + \sum_{k=1}^{M^{\word{\gamma i}}} \widehat{\mathbb{W}}^{\word{v}_{k,\word{\gamma i}}}_t Q_{k,\word{\gamma i}}(T) - \sum_{k=1}^{M^{\word{\gamma i}}} \sum_{m=1}^{d_{k,\word{\gamma i}}} a_{m,k} m!  \left<\word{0}_{m} \shuffle \word{v}_{k,\word{\gamma i}},\widehat{\mathbb{W}}_t\right>.
\end{align*}
For $k \le M^{\word{\gamma i}}$, the induction hypothesis yields $\deg(P_{k,\word{\gamma i}})+ \|\word{v}_{k,\word{\gamma i}}\| \le N$, then  $\deg(Q_{k,\word{\gamma i}})+\|\word{v}_{k,\word{\gamma i}}\| \le N+1$. Moreover for $m \le d_{k,\word{\gamma i}}=\text{deg}(Q_{k,\word{\gamma i}})$, $\word{0}_{m} \shuffle \word{v}_{k,\word{\gamma i}}$ is a linear combination of words of length $m + \|\word{v}_{k,\word{\gamma i}}\|$ with  $m + \|\word{v}_{k,\word{\gamma i}}\| \le \|\word{v}_{k,\word{\gamma i}}\| + \deg(P_{k,\word{\gamma i}}) + 1 \le N+1$. With the same arguments, we know that $\word{0} \shuffle \word{\gamma i}$ is a linear combination of words of length $N+1$.
\newline
\\
\textbf{Fourth case: $\word{w}=\word{\gamma ii}, \word{i} \neq \word{0}$.}
\newline
By Lemma \ref{lemme utile strato}, we know that $\widehat{\mathbb{W}}^{\word{\gamma ii}}_T=\int_{0}^{T}\widehat{\mathbb{W}}^{\word{\gamma i}}_u \circ dW_u^{i} = \int_{0}^{T}\widehat{\mathbb{W}}^{\word{\gamma i}}_u dW_u^{i} + \frac{1}{2}\int_{0}^{T}\widehat{\mathbb{W}}^{\word{\gamma}}_udu$. Then 

\begin{align*}
\mathbb{E}\left[\widehat{\mathbb{W}}^{\word{\gamma ii}}_T | \mathcal{F}_t\right]&=  \mathbb{E}\left[\int_{0}^{T} \widehat{\mathbb{W}}^{\word{\gamma i}}_u dW_u^{\word{i}} + \frac{1}{2}\int_{0}^{T}\widehat{\mathbb{W}}^{\word{\gamma}}_udu | \mathcal{F}_t\right] \\
&=\int_{0}^{t}\widehat{\mathbb{W}}^{\word{\gamma i}}_udW_u^{\word{i}} + \frac{1}{2}\int_{0}^{t} \widehat{\mathbb{W}}^{\word{\gamma}}_udu + \frac{1}{2}\int_{t}^{T} \mathbb{E}\left[\widehat{\mathbb{W}}^{\word{\gamma}}_u|\mathcal{F}_t \right]du \\
&=\widehat{\mathbb{W}}^{\word{\gamma ii}}_t + \frac{1}{2}\int_{t}^{T} \mathbb{E}\left[\widehat{\mathbb{W}}^{\word{\gamma}}_u|\mathcal{F}_t \right]du.
\end{align*}
Using the induction hypothesis, we know that regardless $\word{\gamma}$ ends with $\word{0}$ or not, for every $u \ge t$, $\mathbb{E}\left[\widehat{\mathbb{W}}^{\word{\gamma}}_u|\mathcal{F}_t\right]=\sum_{k=1}^{M^{\word{\gamma}}}P_{k,\word{\gamma}}(u)\widehat{\mathbb{W}}^{\word{v}_{k,\word{\gamma}}}_t$ with $\deg(P_{k,\word{\gamma}}) + \|\word{v}_{k,\word{\gamma}}\| \le \|\word{\gamma}\| =N - 1$. If $Q_{k,\word{\gamma}}$ denotes the antiderivative of the polynomial $P_{k,\word{\gamma}}$ with no constant term, and $Q_{k,\word{\gamma}}(t) = \sum_{m=1}^{d_{k, \word{\gamma}}} a_{m,k} t^m$, we can write 

\begin{align*}
\mathbb{E}\left[\widehat{\mathbb{W}}^{\word{\gamma ii}}_T | \mathcal{F}_t\right]&=\widehat{\mathbb{W}}^{\word{\gamma ii}}_t + \frac{1}{2}\int_{t}^{T} \sum_{k=1}^{M^{\word{\gamma}}}P_{k,\word{\gamma}}(u)\widehat{\mathbb{W}}^{\word{v}_{k,\word{\gamma}}}_t du \\
&= \widehat{\mathbb{W}}^{\word{\gamma ii}}_t + \frac{1}{2}\sum_{k=1}^{M^{\word{\gamma}}} \widehat{\mathbb{W}}^{\word{v}_{k,\word{\gamma}}}_t Q_{k,\word{\gamma}}(T) - \frac{1}{2}\sum_{k=1}^{M^{\word{\gamma}}} \sum_{m=1}^{d_{k,\word{\gamma}}} a_{m,k} m! \left< \word{0}_{m} \shuffle \word{v}_{k,\word{\gamma}} ,\widehat{\mathbb{W}}_t\right>.
\end{align*}
For $k \le M^{\word{\gamma}}$, we know that $\deg(P_{k,\word{\gamma}})+ \|\word{v}_{k,\word{\gamma}}\| \le N-1$ so that  $\deg(Q_{k,\word{\gamma}}) + \|\word{v}_{k,\word{\gamma}}\|  \le N $. Moreover for $m \le d_{k,\word{\gamma}}=\text{deg}(P_{k,\word{\gamma}})$, $\word{0}_{m} \shuffle \word{v}_{k,\word{\gamma}}$ is a linear combination of words of length $m + \|\word{v_{k,\word{\gamma}}}\| \le \|\word{v_{k,\word{\gamma}}}\| +\deg(P_{k}^{\word{\gamma}}) + 1 \le N$.

\end{proof}

\subsection{Proof of Theorem \ref{main result 3}}

Let $(X_t)_{t \le T}$ be the solution to \eqref{Ito form}, $N \in \mathbb{N}^{*}$ be a truncation order and $B \subset \mathcal{W}_{\le N}$ be a basis of words for $\mathcal{W}_{N}$. In the following, $\|.\|$ stands for the classical Euclidean norm and $\|.\|_{\rm CC}$ stands for the Carnot-Carathéodory norm (see \citep{friz2010multidimensional}[Theorem 7.32] for more details). Since the signature is invariant under translation, and by the same arguments as in the proof of Theorem \ref{main result 2}, we assume without loss of generality that $X$ starts from $0$ at $t=0$. Let $p \in (2,3)$. We use the rough path notation introduced in the proof of Theorem \ref{main result 2}. In particular, $S_2^1$ denotes the canonical
$2$-step lift of a path of finite variation, and $S_N^p$ denotes the
Lyons extension map from step $2$ to step $N$ in the $p$-variation rough path topology.
\newline
\newline
\textbf{First case: $N=1$.}
\newline
Because the terminal value of $X^{(n)}$ is the same as the terminal value of $X$ for all $n \in \mathbb{N}^*$, we then have $\widehat{\mathbb{X}^{(n)}}_{T}^{\le 1} = \widehat{\mathbb{X}}_{T}^{\le 1}$ for all $n \in \mathbb{N}^{*}$.
\newline
\newline
\textbf{Second case: $N \ge 2$.}
\newline
By applying \citep{coutin2005semi}[Proposition 2] it follows that 

\begin{align}
\label{conv_S2}
S_{2}^{1}(\widehat{X}^{(n)}) \underset{n \rightarrow \infty}{\overset{\mathbb{P}}{\longrightarrow}} \widehat{\mathbb{X}}^{\le 2}  
\end{align}
for the $p$-variation rough path topology. We check that
\begin{align}
\label{chasles signature}
S^{1}_N = S^{p}_{N} \circ S^{1}_{2}.
\end{align} 
This equality means that computing the $N$-step truncated signature of a $\mathbb{R}^{d+1}$-valued path of finite variation is equivalent to computing the $N$-step truncated signature of its $2$-step truncated signature. Let us quickly prove \eqref{chasles signature}. We denote by $\pi_{0,2} : G^{N}(\mathbb{R}^{d+1}) \rightarrow G^{2}(\mathbb{R}^{d+1})$ the canonical projection (see \citep{friz2010multidimensional} after Remark 7.3). It follows that $\pi_{0,2} \circ S^{1}_{N} = S^{1}_2$. Since $\mathcal{C}^{1-\text{var}}_{o}\left([0,T],G^{2}\left(\mathbb{R}^{d+1}\right)\right) \subset \mathcal{C}^{p-\text{var}}_{o}\left([0,T],G^{2}\left(\mathbb{R}^{d+1}\right)\right)$, we can compose both terms of the previous equality with $S^{p}_N$ and we obtain $S_{N}^{p} \circ \pi_{0,2} \circ S_{N}^{1} = S_{N}^{p} \circ S_{2}^{1}$. By \citep{friz2010multidimensional}[Theorem 9.5] we know that $S_{N}^{p} \circ \pi_{0,2} = id$. Hence we obtain \eqref{chasles signature}. By the continuity of $S^{p}_N$ for the $p$-variation rough path metric (see \citep{friz2010multidimensional}[Theorem 9.5]) and using \eqref{conv_S2},\eqref{chasles signature}, we get 

\begin{align}
\label{convergence proba}
    \widehat{\mathbb{X}^{(n)}}^{\le N}&=S_{N}^{1}(\widehat{X}^{(n)})= S_{N}^{p}(S_{2}^{1}(\widehat{X}^{(n)})) \overset{\mathbb{P}}{\underset{n \rightarrow \infty}{\longrightarrow}} S_{N}^{p}\left(\widehat{\mathbb{X}}^{\le 2}\right)=\widehat{\mathbb{X}}^{\le N}
\end{align}
for the $p$-variation rough path topology. It follows that $\widehat{\mathbb{X}^{(n)}}_{T}^{\le N} \overset{\mathbb{P}}{\underset{n \rightarrow \infty}{\longrightarrow}} \widehat{\mathbb{X}}^{\le N}_{T}$ for $d_{\text{CC}}$. Using classical ball-box estimates \citep{friz2010multidimensional}[Proposition 7.49] we know that

\begin{align*}
\left\|\widehat{\mathbb{X}^{(n)}}_{T}^{\le N} - 
\widehat{\mathbb{X}}^{\le N}_{T} \right\| \lesssim \max \left(d_{\text{CC}}\left(\widehat{\mathbb{X}^{(n)}}_{T}^{\le N},\widehat{\mathbb{X}}^{\le N}_{T}\right) \max \left(1,\|\widehat{\mathbb{X}}^{\le N}_{T}\|_{\text{CC}}^{N-1}\right), d_{\text{CC}}\left(\widehat{\mathbb{X}^{(n)}}_{T}^{\le N},\widehat{\mathbb{X}}^{\le N}_{T} \right)^{N} \right).   
\end{align*}
By \citep{friz2010multidimensional}[Proposition 7.45], we know that $\|\widehat{\mathbb{X}}^{\le N}_{T}\|_{\text{CC}} \lesssim \max (\|\widehat{\mathbb{X}}^{\le N}_{T}\|,\|\widehat{\mathbb{X}}^{\le N}_{T}\|^{\frac{1}{N}})$ so that $\|\widehat{\mathbb{X}}^{\le N}_{T}\|_{\text{CC}}^{N-1} < \infty$ almost surely. We obtain $\widehat{\mathbb{X}^{(n)}}_{T}^{\le N} \overset{\mathbb{P}}{\underset{n \rightarrow \infty}{\longrightarrow}} \widehat{\mathbb{X}}^{\le N}_{T}$ for the Euclidean metric. It immediately follows that $\left(\widehat{\mathbb{X}^{(n)}}_{T}^{\word{w}}\right)_{\word{w} \in B} \overset{\mathbb{P}}{\underset{n \rightarrow \infty}{\longrightarrow}} (\widehat{\mathbb{X}}^{\word{w}}_{T})_{\word{w} \in B}$, where $(\widehat{\mathbb{X}}^{\word{w}}_{T})_{\word{w} \in B}$ is linearly independent for the almost-sure equality by Theorem \ref{main result 2}. We finally use the fact that, for any finite dimension $m \in \mathbb{N}$, the set of random vectors taking values in $\mathbb{R}^m$ whose components are linearly independent for the almost-sure equality is open for the weak convergence topology. Hence, it follows that there exists a deterministic integer $n_0 \in \mathbb{N}$ such that for all $n \ge n_0$, the family $\left(\widehat{\mathbb{X}^{(n)}}_{T}^{\word{w}}\right)_{\word{w} \in B}$ is linearly independent for the almost-sure equality. Moreover, since $B$ is a basis of words for $\mathcal{W}_{N}$ and by (i) of Remark \ref{remark base libre and existence and uniqueness of solution}, the family  $\left(\widehat{\mathbb{X}^{(n)}}_{T}^{\word{w}}\right)_{\word{w} \in B}$ is a basis of $\text{Span}\left(\widehat{\mathbb{X}^{(n)}}_{T}^{\word{w}} ~|~ \word{w} \in \mathcal{W}_{\le N}\right)$ for all $n \ge n_0$. 

For the sake of completeness, we prove that the set $LD(\mathbb{R}^m)$ of random vectors taking values in $\mathbb{R}^m$ whose components are linearly dependent is closed for the weak convergence topology. Let $(Z_n)_{n \in \mathbb{N}}$ be a sequence of elements of $LD(\mathbb{R}^m)$ which weakly converges to $Z$. For each $n \in \mathbb{N}$, there exists $\lambda_n \in \mathbb{R}^{m}$ such that $\| \lambda_n\|=1$ and $\mathbb{P}(\lambda_n^{\top} Z_n=0)=1$. Since the unit sphere is compact, we can find a subsequence $(\lambda_{n_k})_{k \in \mathbb{N}}$ which converges to $\lambda$ with $\|\lambda\|=1$. With a slight abuse of notation, we still denote this subsequence by $(\lambda_n)_{n \in \mathbb{N}}$. We can write $\lambda^{\top} Z_n =\lambda_n^{\top} Z_n + (\lambda - \lambda_n)^{\top} Z_n = (\lambda - \lambda_n)^{\top} Z_n$. Using the tightness of $(Z_n)_{n \in \mathbb{N}}$ (which results from its weak convergence to $Z$), we can show that $\lambda^{\top} Z_n = (\lambda- \lambda_n)^{\top} Z_n \xrightarrow{\mathbb{P}} 0$. Moreover, by continuity of the scalar product, $\lambda^{\top} Z_n$ weakly converges to $\lambda^{\top} Z$ and we conclude using the uniqueness of the limit.

\subsection{Proof of Theorem \ref{main result 4}}

Let $N \in \mathbb{N}^*$ be a truncation order, $B \subset \mathcal{W}_{\le N}$ be a basis of words for $\mathcal{W}_N$, and $K \ge 2$ be the number of concatenated affine paths. Since $B \subset \Phi_{\text{f}}(B)$, we obtain $\length{\Phi_{\text{f}}(B)} \ge \length{B}$ and $\text{Card}(\Phi_{\text{f}}(B)) \ge \text{Card}(B)=(d+1)^N$ (using Corollary \ref{cardinal bases}). Therefore, we obtain

\begin{align}
\label{ineq 1}
    C_{\text{f}}(B,K) \ge \text{Card}(B) + 2(K-1) \length{B} \ge (d+1)^{N} + 2(K-1) \length{B}.
\end{align}
Since $\Phi_{\text{f}}({}^{*}\mathcal{W}_{\le N})={}^{*}\mathcal{W}_{\le N}$ and $\text{Card}({}^{*}\mathcal{W}_{\le N})=(d+1)^N$, we have the following equality

\begin{align}
\label{ineq 2}
    C_{\text{f}}({}^{*}\mathcal{W}_{\le N},K) = (d+1)^{N} + 2(K-1) \length{{}^{*}\mathcal{W}_{\le N}}.
\end{align}
By Corollary \ref{corrolary length} we have $\length{{}^{*}\mathcal{W}_{\le N}} \le \length{B}$. Finally, combining \eqref{ineq 1} and \eqref{ineq 2}, we obtain

\begin{align*}
    C_{\text{f}}({}^{*}\mathcal{W}_{\le N},K) &\le C_{\text{f}}(B,K).
\end{align*}
Since $\Phi_{\text{b}}(\mathcal{W}_{\le N}^{*})=\mathcal{W}_{\le N}^{*}$, we can use similar arguments to prove that $C_{\text{b}}(\mathcal{W}_{\le N}^{*},K) \le C_{\text{b}}(B,K)$.

\subsection{Proof of Theorem \ref{main result EFM}} 

In this section, we denote by $\mathcal{L}$ the Lebesgue measure on $\mathbb{R}$: $\forall E\in \mathcal{B}(\mathbb{R}), \mathcal{L}(E)=\int_{\mathbb{R}} \mathbbm{1}_{E}(t) dt$. The proof of Theorem \ref{main result EFM} relies on the following elementary lemma.

\begin{lemma}    
\label{lemme C support}
Let $C \in \mathcal{B}([0,T])$ be a Borel set such that $\mathcal{L}(C \cap I)>0$ for each non-empty open interval $I \subset [0,T]$. Let $f : [0,T] \rightarrow \mathbb{R}$ be a continuous function. If $f \not\equiv 0$, then $\mathcal{L}(\{t \in C ~|~ f(t) \neq 0\})>0$.
\end{lemma}

We now provide the proof of Theorem \ref{main result EFM}.

\begin{proof}[Proof of Theorem \ref{main result EFM}]
We prove by induction on $N \in \mathbb{N}$ the equivalent assertion: for all $N \in \mathbb{N}$, there exists $A_N \in \mathcal{F}$ with $\mathbb{P}(A_N)=1$, such that for all $\omega \in A_N$:
\begin{align}
\label{equiv assertion}
\forall (c_{\word{w}})_{\word{w} \in \mathcal{W}_N} \in \mathbb{R}^{\mathcal{W}_N} \setminus \{0\}, \quad \mathcal{L}\left(\left\{t \in C ~|~\sum_{\word{w} \in \mathcal{W}_N} c_{\word{w}} \widehat{\mathbb{X}}_{t}^{\mathbf{\lambda},\word{w}}(\omega) \neq 0\right\}\right) >0.
\end{align}
For $N \in \mathbb{N}$, $c \in \mathbb{R}^{\mathcal{W}_N}$ and $t \in [0,T]$, we denote by $F_{t}^{c}:=\sum_{\word{w} \in \mathcal{W}_{N}} c_{\word{w}} \widehat{\mathbb{X}}_{t}^{\mathbf{\lambda},\word{w}}$. We have the following decomposition

\begin{align*}
F_t^{c}&= \sum_{\word{w} \in \mathcal{W}_{N-1}} c_{\word{w0}} \widehat{\mathbb{X}}_{t}^{\mathbf{\lambda},\word{w0}} + \sum_{\word{i} \in \{\word{1},\cdots,\word{d}\}} \sum_{\word{w} \in \mathcal{W}_{N-1}} c_{\word{wi}} \widehat{\mathbb{X}}_{t}^{\mathbf{\lambda},\word{wi}} \\
&= \sum_{\word{w} \in \mathcal{W}_{N-1}} c_{\word{w0}} \int_{-\infty}^{t}e^{-\lambda_{\word{w0}}(t-s)} \widehat{\mathbb{X}}_{s}^{\mathbf{\lambda},\word{w}} ds + \sum_{\word{i} \in \{\word{1},\cdots,\word{d} \} }\sum_{\word{w} \in \mathcal{W}_{N-1}} c_{\word{wi}} \int_{-\infty}^{t}e^{-\lambda_{\word{wi}}(t-s)} \widehat{\mathbb{X}}_{s}^{\mathbf{\lambda},\word{w}} \circ dX_{s}^{\word{i}} \\
&=\sum_{\word{w} \in \mathcal{W}_{N-1}} c_{\word{w0}} \int_{-\infty}^{t}e^{-\lambda_{\word{w0}}(t-s)} \widehat{\mathbb{X}}_{s}^{\mathbf{\lambda},\word{w}} ds + \sum_{\word{i} \in \{\word{1},\cdots,\word{d} \} }\sum_{\word{w} \in \mathcal{W}_{N-1}} c_{\word{wi}} e^{-\lambda_{\word{wi}}t}\int_{-\infty}^{0}e^{\lambda_{\word{wi}}s} \widehat{\mathbb{X}}_{s}^{\mathbf{\lambda},\word{w}} d\eta_{s}^{\word{i}} \\
&\quad + \sum_{\word{i} \in \{\word{1},\cdots,\word{d} \} } \sum_{j=1}^{m}\sum_{\word{w} \in \mathcal{W}_{N-1}} c_{\word{wi}} \int_{0}^{t}e^{-\lambda_{\word{wi}}(t-s)} \widehat{\mathbb{X}}_{s}^{\mathbf{\lambda},\word{w}} b_s^{\word{i}} ds \\
&\quad + \sum_{\word{i} \in \{\word{1},\cdots,\word{d} \} } \sum_{j=1}^{m}\sum_{\word{w} \in \mathcal{W}_{N-1}} c_{\word{wi}} \int_{0}^{t}e^{-\lambda_{\word{wi}}(t-s)} \widehat{\mathbb{X}}_{s}^{\mathbf{\lambda},\word{w}} \sigma_{s}^{\word{i},j} \circ dW_{s}^{j} \\
&=V_t + P_t,
\end{align*}

with $P_t := \sum_{j=1}^{m} \sum_{\word{i} \in \{\word{1},\cdots,\word{d}\}} \sum_{\word{w} \in \mathcal{W}_{N-1}}e^{-\lambda_{\word{wi}}t} \int_{0}^{t} e^{\lambda_{\word{wi}}s} \widehat{\mathbb{X}}_{s}^{\mathbf{\lambda},\word{w}} \sigma_{s}^{\word{i},j} dW_s^{j}$ and $V$ is a finite variation process. Moreover, It\^o's lemma gives

\[
dP_t = \sum_{j=1}^{m} \sum_{\word{i} \in \{\word{1},\cdots,\word{d}\}} \sum_{\word{w} \in \mathcal{W}_{N-1}} c_{\word{wi}}\widehat{\mathbb{X}}_{t}^{\mathbf{\lambda},\word{w}} \sigma_{t}^{\word{i},j} dW_t^{j} - \sum_{j=1}^{m} \sum_{\word{i} \in \{\word{1},\cdots,\word{d}\}} \sum_{\word{w} \in \mathcal{W}_{N-1}} c_{\word{wi}} \lambda_{\word{wi}}e^{-\lambda_{\word{wi}}t} \int_{0}^{t} e^{\lambda_{\word{wi}}s}\widehat{\mathbb{X}}_{t}^{\mathbf{\lambda},\word{w}} \sigma_{s}^{\word{i},j} dW_s^{j} dt.
\]

Hence, the local martingale part $M$ of $F^c$ satisfies

\begin{align}
\label{decomposition F}
dM_t = \sum_{j=1}^{m} \sum_{\word{i} \in \{\word{1},\cdots,\word{d}\}} \sum_{\word{w} \in \mathcal{W}_{N-1}} c_{\word{wi}}\widehat{\mathbb{X}}_{t}^{\mathbf{\lambda},\word{w}} \sigma_{t}^{\word{i},j} dW_t^{j}
\end{align}

Let us define the full-measure event $A_N$ on which \eqref{equiv assertion} holds. We denote by $A_{\sigma}$ the event on which the non-degeneracy of $\sigma$
holds pathwise:
\[
A_\sigma
:=
\left\{
\omega\in\Omega:
\mathcal{L}\left(
\left\{
t\in C:
\sigma_t(\omega)\sigma_t^\top(\omega)
\text{ is non-singular}
\right\}
\right)
=
\mathcal{L}(C)
\right\} \in \mathcal{F}.
\]
Since $\sigma_t\sigma_t^\top$ is non-singular
$\mathbbm 1_C(t)dt\otimes d\mathbb P$-a.e., Fubini's theorem gives $\mathbb P(A_\sigma)=1$. For a word $\word{w}$, we denote by $r(\word{w})$ the number of terminal
zeros of $\word{w}$, namely
\[
r(\word{w})
:=
\max
\left\{
r\in\{0,\dots,\|\word{w}\|\}:
\word{w}=\word{v}\word{0_r}
\text{ for some }
\word{v}\in\mathcal W_{\|\word{w}\|-r}
\right\}.
\]

For $N\in \mathbb{N}$ and $r\in\{0,\dots,N\}$, set $\mathcal{W}_{\le N}^{(r)}
:=
\{\word{w}\in\mathcal{W}_{\le N} : r(\word{w})\ge r\}$. Let $\pi_k^0=\{0=t_0^k<t_1^k<\cdots<t_{m_k}^k=T\}$ be a deterministic sequence of partitions of $[0,T]$ such that $|\pi_k^0|\to0$. For $\word{w}\in\mathcal{W}_{\le N}^{(0)}$, the process $Y^{0,\word{w}}:=\widehat{\mathbb{X}}^{\mathbf{\lambda},\word{w}}$ is a continuous semimartingale. For $r \in \{1,\cdots,N\}$ and $\word{w}\in\mathcal{W}_{\le N}^{(r)}$, the process $\widehat{\mathbb{X}}^{\mathbf{\lambda},\word{w}}$ is $r$-times continuously differentiable and $Y^{r,\word{w}}:=\frac{d^r}{dt^r}
\widehat{\mathbb{X}}^{\mathbf{\lambda},\word{w}}$ is a continuous semimartingale. Since the collection of continuous semimartingales $\left(
Y^{r,\word{w}}
\right)_{\word{w}\in\mathcal{W}_{\le N}^{(r)},r\in \{0,\cdots,N\}}$ is finite, and using \cite{revuz1999continuous}[Chapter IV. Theorem 1.9], there exists a deterministic
subsequence of $(\pi_k^0)_{k\ge1}$, denoted by
$(\pi_k)_{k\ge1}$, and a full-measure event $\Omega_{\mathrm{qv}}^{N} \in \mathcal{F}$, such that for every pair $Y^{(1)},Y^{(2)}$ in this finite
family and for every $\omega \in \Omega_{\mathrm{qv}}^{N}$,
\begin{align}
\label{pathwise bracket}
\sum_{\substack{i=0,\dots,m_k-1, t_{i+1}^k\le t}}
\left(Y^{(1)}_{t_{i+1}^k}(\omega)-Y^{(1)}_{t_i^k}(\omega)\right)
\left(Y^{(2)}_{t_{i+1}^k}(\omega)-Y^{(2)}_{t_i^k}(\omega)\right)
\underset{k \to \infty}{\longrightarrow}
\langle Y^{(1)},Y^{(2)}\rangle_t(\omega)
\end{align}
uniformly in $t\in[0,T]$. On this full-measure event, by bilinearity, every linear combination $Y$ of the finite family $\left(
Y^{r,\word{w}}
\right)_{\word{w}\in\mathcal{W}_N^{(r)}, r\in \{0,\cdots,N\}}$ satisfies \eqref{pathwise bracket} when replacing $Y^{(1)}$ and $Y^{(2)}$ by $Y$. Finally, we define
\[
A_N
:=
A_\sigma
\cap
\Omega_{\mathrm{qv}}^{N} \in \mathcal{F},
\]
which is a finite intersection of events of probability one. Hence $\mathbb P(A_N)=1$.
\newline
\newline
\underline{Initialization step.} The property is trivially satisfied for $N=0$ since $\widehat{\mathbb{X}}^{\mathbf{\lambda},\emptyword}_t=1$.
\newline
\newline
\underline{Induction step.} Let $N\in\mathbb N^*$, and assume that the assertion holds at all levels
$0,\dots,N-1$. Let $\omega\in A_{N}$, and let $(c_{\word{w}})_{\word{w}\in\mathcal W_{N}}
\in
\mathbb R^{\mathcal{W}_{N}}\setminus\{0\}$. Let us prove that $\mathcal{L}\left(
\left\{
t\in C~|~ F_t^{c}(\omega)\neq0
\right\}
\right)>0$. Since $F^{c}(\omega)$ is continuous, it suffices to prove that $F^{c}(\omega) \not\equiv 0$ before applying Lemma \ref{lemme C support}. Let $(c_{\word{w}})_{\word{w} \in \mathcal{W}_{N}} \in \mathcal{W}_{N} \setminus \{0\}$. 
\newline
\newline
\textbf{First case: $c_{\word{w}}=0, \forall \word{w} \in \mathbb{R}^{\mathcal{W}_{N}} \setminus \{\word{0_{N}}\}$.}
\newline
Since $c_{\word{w}}=0, \forall \word{w} \in \mathbb{R}^{\mathcal{W}_{N}} \neq 0$, we necessarily have $c_{\word{0_{N}}}\neq0$. Hence $F_t^{c}(\omega)=c_{\word{0_{N}}}\widehat{\mathbb{X}}^{\mathbf{\lambda},\word{0_{N}}}_t(\omega)
=
\frac{c_{\word{0_{N}}}}{\lambda_0^{N}N!}
\neq0$ for all $t \in [0,T]$.
\newline
\newline
\textbf{Second case: $\exists \word{i^{*}} \in \{\word{1},\cdots,\word{d}\}, (c_{\word{vi^{*}}})_{\word{v} \in\mathcal{W}_{N-1}}\neq0$.}
\newline
For $\word{i} \in \{\word{1},\cdots,\word{d}\}$ and $t \in [0,T]$, we denote by $A^{\word{i}}_t:=\sum_{\word{v} \in \mathcal{W}_{N-1}} c_{\word{vi}} \widehat{\mathbb{X}}_{t}^{\mathbf{\lambda},\word{v}}$. Since $\omega \in A_{N} \subset A_{N-1}$, and using the induction at level $N-1$, we deduce that $\mathcal{L}(\{t \in C ~|~ A^{\word{i^{*}}}_t(\omega) \neq 0\})>0$. Moreover, since $\omega \in A_{\sigma}$, we know that $\mathcal{L}(\{t \in C ~|~ \sigma_{t}(\omega) \sigma_{t}^{\top}(\omega) \text{ is non-singular}\})=\mathcal{L}(C)>0$. Hence, the intersection $\mathcal{I}:=\{t \in C ~|~ A^{\word{i^{*}}}_t(\omega) \neq 0\} \cap \{t \in C ~|~ \sigma_{t}(\omega) \sigma_{t}^{\top}(\omega) \text{ is non-singular}\}$ is of positive Lebesgue measure. Let $t \in \mathcal{I}$. Since $A_{t}^{\word{i^{*}}}(\omega) \neq 0$ then $(A_{t}^{\word{i}}(\omega))_{\word{i} \in \{\word{1}, \cdots,\word{d}\}} \neq 0$. Moreover, $\sigma_{t}(\omega) \sigma_{t}^{\top}(\omega)$ is non singular, then $\sum_{j=1}^{m} \left(\sum_{\word{i} \in \{\word{1},\cdots,\word{d}\}}A_{t}^{\word{i}}(\omega) \sigma_{t}^{\word{i},j}(\omega)\right)^2 >0$. Then, the following inclusion holds: $\mathcal{I} \subset \left\{ t \in C ~|~ \sum_{j=1}^{m} \left(\sum_{\word{i} \in \{\word{1},\cdots,\word{d}\}}A_{t}^{\word{i}}(\omega) \sigma_{t}^{\word{i},j}(\omega)\right)^2 >0\right\}$. Hence, it implies the following inequality: $\mathcal{L}(\{ t \in C ~|~ \sum_{j=1}^{m} (\sum_{\word{i} \in \{\word{1},\cdots,\word{d}\}}A_{t}^{\word{i}}(\omega) \sigma_{t}^{\word{i},j}(\omega))^2 >0\})>0$. It comes that 
\begin{align}
\label{positivité first case}
\int_{0}^{T}\sum_{j=1}^{m} \left(\sum_{\word{i} \in \{\word{1},\cdots,\word{d}\}}A_{t}^{\word{i}}(\omega) \sigma_{t}^{\word{i},j}(\omega)\right)^2 dt>0.
\end{align}
Using \eqref{decomposition F}, we obtain $\left<F^{c}\right>_t =\int_{0}^{t}\sum_{j=1}^{m} \left(\sum_{\word{i} \in \{\word{1},\cdots,\word{d}\}}A_{s}^{\word{i}} \sigma_{s}^{\word{i},j}\right)^2 ds$ and the inequality \eqref{positivité first case} gives $\left<F^{c}\right>_T(\omega)>0$. Since $\omega \in \Omega_{\mathrm{qv}}^{N}$, then $\sum_{(t_i^k)_{0 \le i \le \operatorname{Card}(\pi_k)-1}} |F^{c}_{t_{i+1}^k}(\omega)-F^{c}_{t_{i}^k}(\omega)|^2 \xrightarrow[k \to \infty]{} \left<F^{c}\right>_T(\omega) >0$. Finally, we recall that if $\lim_{k \to \infty}\sum_{(t_i^k)_{0 \le i \le \operatorname{Card}(\pi_k)-1}} |F^{c}_{t_{i+1}^k}(\omega)-F^{c}_{t_{i}^k}(\omega)|^2>0$, then $F^{c}(\omega)$ cannot be of finite variation on $[0,T]$ and in particular we deduce that $F^{c}(\omega) \not\equiv 0$. This well known result comes from the following inequality
\[
\sum_{(t_i^k)_{0 \le i \le \operatorname{Card}(\pi_k)-1}} |F^{c}_{t_{i+1}^k}(\omega)-F^{c}_{t_{i}^k}(\omega)|^2 \le \sup_{0 \le i \le \operatorname{Card}(\pi_k)-1} |F^{c}_{t_{i+1}^k}(\omega)-F^{c}_{t_{i}^k}(\omega)|  \times \|F^{c}\|_{1-\text{var};[0,T]},
\]
where the first term in the right-hand side of the inequality converges to $0$ when $k \to \infty$ by continuity of $F^{c}(\omega)$ on $[0,T]$.
\newline
\newline
\textbf{Third case: $\forall \word{i} \in \{\word{1},\cdots,\word{d}\}, (c_{\word{vi}})_{\word{v} \in \mathcal{W}_{N-1}}=0$.}
\newline
Since $(c_{\word{w}})_{\word{w} \in \mathcal{W}_{N}} \neq 0$, we necessarily have $(c_{\word{v0}})_{\word{v} \in \mathcal{W}_{N-1}} \neq 0$. If $N=1$ or if the only non-zero coefficient is $c_{\word{0_{N}}}$, then the first case applies. Otherwise, $N \ge2$ and there exists at least one word $\word{v} \in \mathcal{W}_{N-1}\setminus\{\word{0_{N-1}}\}$ such that $c_{\word{v0}} \neq 0$. We denote by $\mathcal{B}:=\{\word{v} \in \mathcal{W}_{N-1} \setminus \{\word{0_{N-1}}\} ~|~ c_{\word{v0}} \neq 0\} \neq \varnothing$. Let us define
\[
r_*
:=
\min
\{
r(\word{v}) ~|~ \word{v} \in \mathcal{B}\}.
\]
Then $0\le r_*\le N-2$. The path $F^c(\omega) = \sum_{\word{v} \in \mathcal{B}} c_{\word{v0}} \widehat{\mathbb{X}}^{\mathbf{\lambda},\word{v0}}(\omega)$ is $(r_*+1)$-times continuously differentiable. For any word $\word{v} \in \mathcal{W}_N$, an easy computation gives 
\begin{align}
\label{dérivée pathwise}
\frac{d}{dt}\widehat{\mathbb{X}}^{\mathbf{\lambda},\word{v0}}_t(\omega)=\widehat{\mathbb{X}}^{\mathbf{\lambda},\word{v}}_t(\omega) - \lambda_{\word{v0}} \widehat{\mathbb{X}}^{\mathbf{\lambda},\word{v0}}_t(\omega).
\end{align}
For all $\word{v} \in \mathcal{B}$, there exists a unique word $\word{\tilde{v}} \in \mathcal{W}_{N -1- r_*}$ such that $\word{v}=\word{\tilde{v}0_{r_*}}$, and we denote by $\mathcal{B}_{r_*}:=\{\word{\tilde{v}} \in \mathcal{W}_{N-1-r_*} ~|~ \word{\tilde{v}0_{r_{*}}} \in \mathcal{B}\}$. By iterating \eqref{dérivée pathwise} $r_*+1$ times on $\widehat{\mathbb{X}}^{\mathbf{\lambda},\word{v0}}=\widehat{\mathbb{X}}^{\mathbf{\lambda},\word{\tilde{v}0_{r_*+1}}}$, we obtain that there exists a family of coefficients $(\mu_{\word{\tilde{v}}})_{\word{\tilde{v}} \in \mathcal{W}_{\le N-1}}$ such that
\begin{align}
\label{derivée pathwise 2}
\frac{d^{r_*+1}}{dt^{r_*+1}}F^{c}_{t}(\omega)=\sum_{\word{\tilde{v}} \in \mathcal{B}_{r_*}} c_{\word{\tilde{v}0_{r_*+1}}}\widehat{\mathbb{X}}^{\mathbf{\lambda},\word{\tilde{v}}}_t(\omega) + \sum_{\word{\tilde{v}} \in \mathcal{W}_{\le N-1}} \mu_{\word{\tilde{v}}}\widehat{\mathbb{X}}^{\mathbf{\lambda},\word{\tilde{v}0}}_t(\omega),
\end{align}
where the second term in the right-hand side of the equality \eqref{derivée pathwise 2} is the value at time $t$ of a finite variation continuous path. By definition of $r_*$, the subset $\widetilde{\mathcal{B}}_{r_*}$ of $\mathcal{B}_{r_*}$ which consists in words not ending with the letter $\word{0}$ is not empty. Hence, we can decompose the first term in the right-hand side of \eqref{derivée pathwise 2} into $\sum_{\word{\tilde{v}} \in \mathcal{B}_{r_*}} c_{\word{\tilde{v}0_{r_*+1}}}\widehat{\mathbb{X}}^{\mathbf{\lambda},\word{\tilde{v}}}_t(\omega) = \sum_{\word{\tilde{v}} \in \widetilde{\mathcal{B}}_{r_*}} c_{\word{\tilde{v}0_{r_*+1}}}\widehat{\mathbb{X}}^{\mathbf{\lambda},\word{\tilde{v}}}_t(\omega) + \sum_{\word{\tilde{v}} \in \mathcal{B}_{r_*} \setminus \widetilde{\mathcal{B}}_{r_*}}c_{\word{\tilde{v}0_{r_*+1}}}\widehat{\mathbb{X}}^{\mathbf{\lambda},\word{\tilde{v}}}_t(\omega)$. The second part of the previous equality is also the value at time $t$ of a finite variation path. We denote by $Z$ the continuous semimartingale defined as $Z_t:=\sum_{\word{\tilde{v}} \in \mathcal{B}_{r_*}} c_{\word{\tilde{v}0_{r_*+1}}}\widehat{\mathbb{X}}^{\mathbf{\lambda},\word{\tilde{v}}}_t + \sum_{\word{\tilde{v}} \in \mathcal{W}_{\le N-1}} \mu_{\word{\tilde{v}}}\widehat{\mathbb{X}}^{\mathbf{\lambda},\word{\tilde{v}0}}_t$ for all $t \in [0,T]$. Its local martingale part is the local martingale part of $\sum_{\word{\tilde{v}} \in \widetilde{\mathcal{B}}_{r_*}} c_{\word{\tilde{v}0_{r_*+1}}}\widehat{\mathbb{X}}^{\mathbf{\lambda},\word{\tilde{v}}}_t$. Since $0 \le r_* \le N-2$, it comes that for all words $\word{\tilde{v}} \in \widetilde{\mathcal{B}}_{r_*}= \mathcal{B}_{r_*} \cap \mathcal{W}_{<\infty}^{*}\subset\mathcal{W}_{N-1-r_*}^*$, we have $\|\word{\tilde{v}}\|\ge 1$. Hence, for any word $\word{\tilde{v}} \in \widetilde{\mathcal{B}}_{r_*}$, there exists a unique letter $\word{i} \in \{\word{1},\cdots,\word{d}\}$ and a unique word $\word{\tilde{u}} \in \mathcal{W}_{N-r_*-2}$ such that $\word{\tilde{v}}=\word{\tilde{u}i}$. For each letter $\word{i} \in \{\word{1},\cdots,\word{d}\}$, we denote by $\widetilde{\mathcal{B}}_{r_*,\word{i}}:=\{\word{\tilde{u}} \in \mathcal{W}_{N-r_*-2} ~|~ \word{\tilde{u}i} \in \widetilde{\mathcal{B}}_{r_*}\}$, and 
\begin{align}
\label{union}
\widetilde{\mathcal{B}}_{r_*}=\bigsqcup_{\word{i} \in \{\word{1},\cdots,\word{d}\}}\widetilde{\mathcal{B}}_{r_*,\word{i}}.
\end{align}
For $\word{i}\in\{\word{1},\cdots,\word{d}\}$, set $A_t^{\word{i}}
:=
\sum_{\word{\tilde{u}}\in\mathcal W_{N-r_*-2}}
c_{\word{\tilde{u}}\word{i}\word{0_{r_*+1}}}
\widehat{\mathbb{X}}^{\mathbf{\lambda},\word{\tilde{u}}}_t$. Since $\widetilde{\mathcal{B}}_{r_*} \neq \varnothing$ and using \eqref{union}, there exists
$\word{i^{*}}\in\{\word{1},\cdots,\word{d}\}$ such that $\widetilde{\mathcal{B}}_{r_*,\word{i}}\neq \varnothing$. In particular, it comes that
\[
\left(
c_{\word{\tilde{u}}\word{i^{*}}\word{0_{r_*+1}}}
\right)_{\word{\tilde{u}}\in\mathcal{W}_{N-r_*-2}}
\neq 0.
\]
Since $\omega\in A_{N}\subset A_{N-r_*-2}$, the induction hypothesis gives $\mathcal{L}\left(
\left\{
t\in C:
A_t^{\word{i^{*}}}(\omega)\neq0
\right\}
\right)>0$. Similar arguments as for the previous case gives 

\[
\sum_{(t_i^k)_{0 \le i \le \operatorname{Card}(\pi_k)-1}} \left|\frac{d^{r_*+1}}{dt^{r_*+1}}F_{t_{i+1}^{k}}^c(\omega)-\frac{d^{r_*+1}}{dt^{r_*+1}}F_{t_{i}^{k}}^c(\omega)\right|^2 \xrightarrow[k \to \infty]{} \sum_{j=1}^{m}
\left(
\sum_{\word{i}\in\{\word{1},\cdots,\word{d}\}}
A_t^{\word{i}}(\omega)\sigma_t^{i,j}(\omega)
\right)^2dt>0.
\]
Consequently $\frac{d^{r_*+1}}{dt^{r_*+1}}F^c(\omega)
\not\equiv0$. Hence, $F^c(\omega)\not\equiv0$.
\end{proof}

\mbox{}
\clearpage
\bibliographystyle{plain}
\bibliography{ref}

\end{document}